\documentclass[twoside,11pt]{article}

\usepackage{blindtext}
\usepackage{multirow}
\usepackage{xcolor}
\usepackage{hyperref}
\usepackage{epstopdf} % 自动处理 EPS 文件
\usepackage{graphics,graphicx,epstopdf,url}
\usepackage{arydshln}
\usepackage{color}
\usepackage{algorithmic}
\usepackage{url,supertabular}
\usepackage{enumerate,makecell}
\usepackage{graphicx,epstopdf}
\usepackage{color,verbatim}
\usepackage{mathrsfs,subfigure}
\usepackage{listings,array}
\usepackage{booktabs}
\usepackage{dashrule}
\usepackage{arydshln}
\usepackage{graphicx}
\usepackage{amsfonts}
\usepackage{graphicx}
\usepackage{subfigure}
\usepackage{lineno}
\usepackage{multirow,bm}
\usepackage{color}
\usepackage{algorithm}
\usepackage{algorithmic}
\usepackage{longtable}
\usepackage{booktabs}
\usepackage{rotating}
\usepackage{enumitem}
\usepackage{lineno}
\usepackage{enumerate,makecell}
\usepackage{amsfonts}
\usepackage{graphicx,epstopdf}
\usepackage{color,verbatim}
\usepackage{mathrsfs,subfigure}
\usepackage{listings,array}
\usepackage{booktabs}
\usepackage{dashrule}
\usepackage{arydshln}
\usepackage{graphicx}
\usepackage{longtable}
\usepackage{booktabs}
\usepackage{rotating}
\usepackage{enumitem}
\usepackage{lineno}
\usepackage{url}
\usepackage{pdflscape}
\usepackage{makecell}
\usepackage{lipsum}
\usepackage{clrscode}
\usepackage{appendix}
\usepackage{booktabs}
\usepackage{url}
\usepackage{multirow}
\usepackage{textcomp}
\usepackage{amsmath}
\usepackage{amsfonts}
\usepackage{amssymb}
\usepackage{mathrsfs}

\usepackage{graphicx,graphics,subfigure}
\usepackage{epsf,epstopdf}
\usepackage{algorithm}
\usepackage{algorithmic}
\newcommand{\ba}{\begin{array}}
\newcommand{\ea}{\end{array}}

\newcommand{\bit}{\begin{itemize}}
\newcommand{\eit}{\end{itemize}}
\newcommand{\be}{\begin{equation}}
\newcommand{\bee}{\begin{equation*}}
\newcommand{\eee}{\end{equation*}}
\newcommand{\ee}{\end{equation}}
\newcommand{\bea}{\begin{eqnarray}}
\newcommand{\eea}{\end{eqnarray}}

\newcommand{\st}{\mathrm{s.t.}}

\newcommand{\argmin}{\mathop{\mathrm{arg\,min}}}

\newcommand{\Rmn}[1]{\uppercase\expandafter{\romannumeral#1}}
\renewcommand{\theenumi}{\roman{enumi}}

\newcommand{\Mcal}{\mathcal{M}}

\newcommand{\R}{\mathbb{R}}

\newcommand{\prox}{\mathrm{prox}}
\newcommand{\iprod}[2]{\left \langle #1, #2 \right \rangle }

\newcommand{\tr}{\mathrm{tr}}

\usepackage{lipsum}
\usepackage{clrscode}
\usepackage{appendix}
\usepackage{longtable}
\usepackage{booktabs}
\usepackage{url}
\usepackage{multirow}
\usepackage{textcomp}
\usepackage{amssymb}
\usepackage{mathrsfs}
\usepackage{graphicx,graphics,subfigure}
\usepackage{epsf,epstopdf}
\usepackage{algorithm}
\usepackage{algorithmic}
\usepackage{dsfont}

\newtheorem{assum}{Assumption}
\usepackage{indentfirst}
\usepackage{pdfpages}
\usepackage{pdflscape}
\newlist{assumptions}{enumerate}{1}
\setlist[assumptions,1]{label=A\arabic*, ref=A\arabic*}
\usepackage{longtable}
\usepackage{enumitem}
\usepackage[abbrvbib, preprint]{jmlr2e}
\usepackage{enumerate}
\usepackage{empheq}
\usepackage{lastpage}
\jmlrheading{23}{2024}{1-\pageref{LastPage}}{1/21; Revised 5/22}{9/22}{21-0000}{Deng, Hu, Deng and Wen}

\ShortHeadings{A primal dual semismooth Newton method for semidefinite programming}{Deng, Hu, Deng and Wen}
\firstpageno{1}

\begin{document}

\title{An efficient primal dual semismooth Newton method for semidefinite programming}

\author{\name Zhanwang Deng \email dzw\_opt2022@stu.pku.edu.cn\\
       \addr Academy for Advanced Interdisciplinary Studies\\
       Peking University\\
       Beijing, 100871, China \\
       \name Jiang Hu\thanks{Corresponding author} \email hujiangopt@gmail.com\\
       Department of Mathematics,  \\
       University of California, Berkeley,\\
       Berkeley, CA 94720, US \\
       \name Kangkang Deng \email freedeng1208@gmail.com\\
       \addr Department of Mathematics\\  National University of Defense Technology\\ Changsha, 410073,
CHINA  \\
       \name Zaiwen Wen \email wenzw@pku.edu.cn\\
       \addr Beijing International Center for Mathematical Research \\ Center for Machine Learning Research \\
       Changsha Institute for Computing and Digital Economy \\
       Peking University \\
       Beijing 100871, CHINA}
\editor{My editor}

\maketitle

\begin{abstract}%   <- trailing '%' for backward compatibility of .sty file
In this paper, we present an efficient semismooth Newton method, named SSNCP, for solving a class of semidefinite programming problems. Our approach is rooted in an equivalent semismooth system derived from the saddle point problem induced by the augmented Lagrangian duality.  An additional correction step is incorporated after the semismooth Newton step to ensure that the iterates eventually reside on a manifold where the semismooth system is locally smooth. Global convergence is achieved by carefully designing inexact criteria and leveraging the $\alpha$-averaged property to analyze the error. The correction steps address challenges related to the lack of smoothness in local convergence analysis. Leveraging the smoothness established by the correction steps and assuming a local error bound condition, we establish the local superlinear convergence rate without requiring the stringent assumptions of nonsingularity or strict complementarity. Furthermore, we prove that SSNCP converges to an $\varepsilon$-stationary point with an iteration complexity of $\widetilde{\mathcal{O}}(\varepsilon^{-3/2})$. Numerical experiments on various datasets, especially the Mittelmann benchmark, demonstrate the high efficiency and robustness of SSNCP compared to state-of-the-art solvers.
\end{abstract}


\section{Introduction} \label{sec1}
The goal of this paper is to design an efficient and practical algorithm that is comparable to state-of-the-art solvers for solving the following large-scale convex composite semidefinite programming (SDP) problem:
\begin{equation}\label{prob}
     \min \left<\bm{c},\bm{x}\right>+h(\bm{x}), ~\mathrm{s.t.}~ \mathcal{A}(\bm{x}) \in \mathcal{Q},  \bm{x}\in\mathcal{K},
\end{equation}
where $\bm{c} \in \mathbb{S}^{ n \times n}$, and $\mathcal{K} = \mathbb{S}_+^n$ denotes the positive semidefinite cone. The map $\mathcal{A}: \mathbb{S}^{n \times n} \rightarrow \mathbb{R}^m$ is linear, $\mathcal{Q}$ represents a box constraint, and $h$ is an elementwise indicator function on a polyhedral convex set.
 The function $h$ provides flexibility to handle additional constraints, such as nonnegative or box constraints.
The original problem \eqref{prob} encompasses both the standard SDP and SDP with nonnegative constraints (SDP+).  Although the examples presented in \eqref{prob} focus on specific SDP problems, the proposed algorithm can handle more general forms of $h$, provided that its proximal operator is computationally tractable.

 SDP is a fundamental component of classical conic optimization, with extensive applications in diverse domains. This problem class has been the subject of intensive research for several decades and continues to attract significant attention \citep{zhao2010newton,sun2020sdpnal+,wang2023decomposition,de2025instance,tavakoli2024semidefinite}.  For a comprehensive theoretical analysis and an overview of the critical roles of SDP in fields such as finance, engineering, optimal control, statistics, and machine learning, we refer readers to \citep{anjos2011handbook,wolkowicz2012handbook}. Given its widespread applications in engineering and machine learning, the development of efficient solvers for SDP has emerged as a central focus in the field of optimization in recent decades. We provide a survey of algorithms for solving SDP problems in the following subsection.

\subsection{Literature review}

First-order methods are popular for solving
large-scale optimization problems due to the low complexity at each iteration. For SDP and SDP+ problems, the alternating direction method of multipliers (ADMM), as implemented in SDPAD \citep{wen2010alternating}, has demonstrated considerable numerical efficacy.
A two-block decomposition hybrid proximal extragradient method, termed 2EBD, is proposed in \citep{monteiro2014first}.
Subsequently, a convergent symmetric Gauss–Seidel based three-block ADMM method is developed in \citep{sun2015convergent}, which is capable of handling SDP problems with additional polyhedral set constraints. These first-order methods often achieve a low-accuracy solution rapidly but may require many iterations to attain a highly accurate solution. Consequently, further efforts are necessary to enhance their practicality.

The interior point method \citep{nesterov1994interior} is a well-established, efficient classical approach for SDP. This method employs a sequence of log-barrier penalty subproblems, requiring one step of Newton's method to solve each subproblem.  Open-source solvers such as SeDuMi \citep{sturm1999using} and SDPT3 \citep{toh1999sdpt3} have been widely adopted. They achieve considerable success in various application domains. Furthermore, commercial solvers based on interior point methods, such as MOSEK \citep{MOSEK}, are also well-developed.
Nevertheless, each step of the interior point method requires a factorization of the Schur complement matrix. When the number of constraints is large, the computational complexity and memory requirements increase significantly, leading to substantially increased runtime. If iterative methods are employed, the low-rank or high-rank structure of the solution cannot be exploited.
Furthermore, when handling SDP with affine constraints, the interior point method needs to introduce extra nonnegative variables to transform the original problem into a conic problem. This transformation will lead to an increase in the scale of the linear systems, thereby requiring more memory.

The decomposition method has also proven to be effective for low-rank SDP.
To circumvent the positive semidefinite constraint, Burer and Monteiro \citep{burer2003nonlinear} recast the linear SDP by factorizing the original variables.  They apply the augmented Lagrangian (AL) method
 and use a quasi-Newton method to solve the factorized subproblem. Recently, a decomposition augmented Lagrangian method for low-rank SDP has been proposed in \citep{wang2023decomposition}. This method utilizes a semismooth Newton method on the manifold to solve the subproblem in the AL method, leveraging the low-rank property. However, the initialization and adjustment of the rank parameter are crucial for the decomposition method.
%  This additional requirement may lead to unstable
% numerical efficiency.
Furthermore, for problems that have high-rank solutions, the decomposition method may not be efficient.

In addition to the methods mentioned above, semismooth Newton methods have also demonstrated effectiveness in addressing large-scale SDP problems. In the context of the AL method, a Newton-CG AL method to solve SDPs,  SDPNAL, is proposed in \citep{zhao2010newton}. This method is later extended to a MATLAB software package, SDPNAL+, for SDP with bound constraints \citep{sun2020sdpnal+}.
Furthermore, motivated by the equivalence between the Douglas-Rachford splitting (DRS) iteration and the ADMM, a semismooth Newton method based solver, SSNSDP, for SDP problems arising in electronic structure calculations has been proposed \citep{li2018semismooth}. However, its convergence analysis relies on switching between the semismooth Newton method and ADMM.
This approach is subsequently generalized to address optimal transport problems \citep{liu2022multiscale}.
Recently, a primal dual semismooth Newton method has been proposed for multi-block convex composite optimization problems in \citep{deng2023augmented}. This method is based on the nonlinear system derived from the AL saddle point problem. However, the strict complementarity (SC) condition is necessary for the convergence analysis, which may not hold in general. This motivates us to design a more practical algorithm with relaxed requirements for superlinear convergence analysis.

\subsection{Contributions}
The contributions of this paper are listed as follows:

    (1) Using the AL duality, we transform the original problem \eqref{prob} into a saddle point problem and formulate a semismooth system of nonlinear equations to characterize the optimality conditions. Different from the construction in \citep{li2018semismooth} which depends on the DRS iteration for two blocks of composite optimization problems, our reformulation is easily accessible and better suited for handling \eqref{prob}.
    Subsequently, we develop SSNCP, a regularized semismooth Newton method with nonmonotone line search and carefully designed inexact criteria. Furthermore, to address challenges related to the lack of smoothness in local analysis, an additional correction step is incorporated.
    Compared with the two loops in the augmented Lagrangian methods \citep{zhao2010newton}, our method updates both primal and dual variables in a single semismooth Newton step simultaneously.

    (2) We establish the global convergence of SSNCP by carefully designing inexact criteria and leveraging the inherent $\alpha$-averaged property to control the error.
    For the local convergence, thanks to the correction step, the manifold identification of the iteration sequence is proved, namely, iterates eventually reside on a manifold. Consequently, without requiring the stringent nonsingularity or SC assumption, we show that the local smooth transition to the superlinear convergence of the iterate sequences under the error bound condition. Furthermore, the worst iteration complexity of SSNCP is $\widetilde{\mathcal{O}}(\varepsilon^{-3/2})$. To our knowledge, this is probably the first iteration complexity analysis on semismooth Newton-type algorithms.

    (3) Promising and extensive numerical results are presented across various SDP and SDP+ datasets. SSNCP demonstrates superior performance compared to state-of-the-art solvers on all tested problems. Notably, SSNCP is able to compete with MOSEK on challenging Mittelmann benchmark problems. The results  indicate that SSNCP is robust and efficient in solving complex and large-scale SDP problems.

\subsection{Notation} For a linear operator $\mathcal{A}$, we denote its adjoint operator by $\mathcal{A}^*$.  The dual cone of $\mathcal{K}$ is defined as $\mathcal{K}^*: = \{\bm{y}:\left<\bm{y},\bm{x}\right>\geq 0, \forall \bm{x}\in \mathcal{K}\}$. For a proper convex function $g$,  we define its domain as ${\rm dom}(g):=\{ \bm{x}: g(\bm{x}) < \infty\}$. The Fenchel conjugate function of $g$ is $g^*(\bm{z}) := \sup_{\bm{x}}\{\left<\bm{x},\bm{z}\right> - g(\bm{x})  \}$ and the subdifferential  is $ \partial g(\bm{x}): = \{\bm{z}:~ g(\bm{y}) - g(\bm{x}) \geq \left<\bm{z}, \bm{y} - \bm{x}  \right>,~\forall \bm{y}  \}. $
For a convex set $\mathcal{Q}$, we use $\delta_{\mathcal{Q}}$ to denote its indicator function, which takes value $0$ on $\mathcal{Q}$ and $+\infty$ elsewhere. The relative interior of $ \mathcal{Q}$ is denoted by ${\rm ri}(\mathcal{Q})$ and the boundary of $\mathcal{Q}$ is denoted by ${\rm bd}(\mathcal{Q}).$
 % and its conjugate function is $\delta_{\mathcal{Q}}^*(\bm{p})=\sup\{\iprod{\bm{d}}{\bm{p}}| \bm{d} \in \mathcal{Q} \}$.
The projection operator onto a closed convex set $\mathcal{C}$ is defined by $ \Pi_{\mathcal{C}}(\bm{x}) = \arg\min_{\bm{y}\in \mathcal{C}}\|\bm{y} - \bm{x} \|^2.$ For any proper closed convex function $g$, and constant $t>0$, the proximal operator of $g$ is defined by $
    \prox_{tg}(\bm{x}) = \arg\min_{\bm{y}}\{g(\bm{y}) + \frac{1}{2t}\|  \bm{y} - \bm{x}\|^2  \}.
$  The notion $n_1 = \mathcal{O}(n)$ means there exists a constant $C$ such that $n_1 \le Cn$ and $\widetilde{\mathcal{O}}$ indicates that logarithmic factors of $\mathcal{O}$ are ignored.

\subsection{Organization}
The rest of this paper is organized as follows. In Section \ref{sec2}, we present a nonlinear system characterization of the optimality condition of \eqref{prob} based on the AL saddle point, under the assumption of Slater’s condition to ensure strong duality. In Section \ref{sec3}, we propose a semismooth Newton method to solve problem \eqref{prob}.  Theoretical analysis, including  global and local convergence as well as iteration complexity, is established in Section \ref{sec4}. Extensive numerical
experiments are presented in Section \ref{sec5} and we conclude this paper in Section \ref{sec6}.

\section{Preliminaries} \label{sec2}
In this section, we first transform the original problem \eqref{prob}  into a saddle point problem using the AL duality.  A monotone nonlinear system derived from the saddle point problem is then presented and analyzed. Furthermore, under mild assumptions on the proximal operator of $h$, such a nonlinear system can be semismooth and equivalent to the Karush–Kuhn–Tucker (KKT) optimality condition of problem \eqref{prob}.
\subsection{An equivalent saddle point problem}
 In this subsection, we present the construction of the saddle point problem. First, the dual problem of \eqref{prob} is given by
\begin{equation}\label{prob:dual0}
\begin{aligned}
     \min_{ \bm{y}, \bm{z}, \bm{s}}  ~~&   h^*(-\bm{z}) + \delta_{\mathcal{Q}}^*(-\bm{y}) + \delta_{\mathcal{K}}^*(-\bm{s}),   \\
     \mathrm{s.t.}~~&~ \mathcal{A}^*(\bm{y}) + \bm{s} + \bm{z} = \bm{c}, \\
    % &~ \bm{s}\in \mathcal{K}^*, \bm{y}\in\mathcal{Y}.
\end{aligned}
\end{equation}
where $\bm{y} \in \mathbb{R}^m, \bm{s} \in \mathbb{S}^{n}, \bm{z} \in \mathbb{S}^{n}.$
Throughout this paper, we make the following assumption.
\begin{assum}\label{assum}
 The dual problem \eqref{prob:dual0} has an optimal solution $\bm{y}_*, \bm{z}_*, \bm{s}_*.$ Furthermore,   Slater’s condition holds for the dual problem \eqref{prob:dual0}, i.e., there exist $-\bm{y} \in {\rm ri}({\rm dom}(\delta_{\mathcal{Q}}^*)), - \bm{s} \in {\rm ri}({\rm dom}(\delta_{\mathcal{K}}^*)),$ and $ -\bm{z} \in {\rm ri}({\rm dom}(h^*)) $ such that $\mathcal{A}^*(\bm{y}) + \bm{s} + \bm{z} = \bm{c}.$
\end{assum}

Under Assumption \ref{assum},  the optimal solution $(\bm{y}_*,\bm{z}_*,\bm{s}_*,\bm{x}_*)$ of the primal problem \eqref{prob} and the dual problem \eqref{prob:dual0}  satisfies the following KKT optimality conditions \citep{karush1939minima,kuhn2014nonlinear}:
\begin{equation}\label{kkt}
    \begin{aligned}
     0  &=  \mathcal{A}^*(\bm{y}_*) + \bm{s}_* + \bm{z}_* - \bm{c},\,
     0  =   \mathcal{A} \bm{x}_* - \Pi_{\mathcal{Q}}(\mathcal{A}\bm{x}_*-\bm{y}_*),\\
     0 &=    \bm{x}_* - \Pi_{\mathcal{K}}(\bm{x}_* -\bm{s}_*),\,
    \bm{x}_*  =   \prox_h(\bm{x}_* -\bm{z}_* ).
    \end{aligned}
\end{equation}

To reformulate the above set-valued KKT system into a single-valued semismooth  system, a modified AL function and the associated saddle point problem are introduced in \citep{deng2023augmented}.
 We first introduce two auxiliary variables $\bm{v},\bm{p}$ to decouple the variables for the constraint $\mathcal{A}^*(\bm{y}) + \bm{s} + \bm{z} = \bm{c}$ from the possibly nonsmooth term $h^*$ and the function $\delta_{\mathcal{Q}}^*$.
 By associating  Lagrangian multipliers $(\bm{x},\bm{u},\bm{q})$ to the nonsmooth term $h^*, \delta_{\mathcal{Q}}^*$, and $\delta_{\mathcal{K}}^*$, and replacing $\bm{v},\bm{p}$, and $\bm{s}$ with their closed-form solutions, the modified AL function of \eqref{prob:dual0} is given by
\begin{equation}\label{eq:alfunc}
\begin{aligned}
\Phi(\bm{y}, \bm{z}, \bm{x},\bm{u},\bm{q}) = &  \underbrace{ h^*(\prox_{ h^*/\sigma}(\bm{q}/\sigma - \bm{z} )) + \frac{1}{2\sigma } \| \prox_{ \sigma h}(\bm{q} - \sigma\bm{z}) \|^2}_{ \text{Moreau envelope of } h^*}\\
     + & \underbrace{\frac{1}{2\sigma }\|\Pi_{\mathcal{K}}(\bm{x} + \sigma(\mathcal{A}^*(\bm{y})  + \bm{z} -\bm{c}) )     \|^2}_{ \text{Moreau envelope of } \delta_{\mathcal{K}}^*} \\
     + &\underbrace{\frac{1}{2\sigma} \|\Pi_{\mathcal{Q}}(\bm{u} - \sigma \bm{y}   ) \|^2}_{\text{Moreau envelope of }
 \delta^*_{\mathcal{Q}}}- \frac{1}{2\sigma}( \|\bm{x}\|^2 + \|\bm{u}\|^2 + \|\bm{q}\|^2),
\end{aligned}
\end{equation}
where $\sigma > 0$ is the penalty parameter.  The Moreau envelope for $h^*$ is defined by:
$
e_{\sigma}h^*(\bm{x}): = \min_{\bm{y}} h^*(\bm{y}) + \frac{\sigma}{2} \| \bm{y} - \bm{x} \|^2.
$
The Moreau envelopes $e_{\sigma}\delta_{\mathcal{Q}}^*$ and $e_{\sigma}\delta_{\mathcal{K}}^*$  can be defined similarly.
Then the corresponding saddle point problem associated with the above AL function is
\begin{equation} \label{prob:saddle}
     \min_{\bm{y},\bm{z}}\max_{\bm{x},\bm{u},\bm{q}} \Phi(\bm{y}, \bm{z}, \bm{x},\bm{u},\bm{q}).
\end{equation}

 Denote $\bm{w} := (\bm{y}, \bm{z}, \bm{x},\bm{u},\bm{q})$. Similar to \cite[Lemma 2.1]{deng2023augmented}, we next give the strong AL duality of $\Phi(\bm{w})$ without proof.
\begin{lemma}
Suppose that Assumption \ref{assum} holds. Given $\sigma > 0$, the strong duality holds for \eqref{prob:saddle}, i.e.,
\begin{equation}\label{lemma:strong}
    \min_{\bm{y},\bm{z}}\max_{\bm{x},\bm{u},\bm{q}} \Phi(\bm{y}, \bm{z}, \bm{x},\bm{u},\bm{q}) = \max_{\bm{x},\bm{u},\bm{q}} \min_{\bm{y},\bm{z}}\Phi(\bm{y}, \bm{z}, \bm{x},\bm{u},\bm{q}),
\end{equation}
where both sides of \eqref{lemma:strong} are equivalent to problem \eqref{prob:dual0}.
\end{lemma}

\subsection{A monotone nonlinear system} \label{subsec:strong}
 In this subsection, we present and analyze the nonlinear system induced from the saddle point problem \eqref{prob:saddle}.
It follows from the Moreau envelope theorem \citep{beck2017first} that $e_{\sigma}h^*$, $e_{\sigma}\delta_{\mathcal{Q}}^*$, and $e_{\sigma}\delta_{\mathcal{K}}^*$ are continuously differentiable, which implies that $\Phi$ is also continuously differentiable.  The gradient of function $\Phi$ in \eqref{eq:alfunc} is given by:
\begin{equation} \label{eq:grad-F}
    \begin{aligned}
    \nabla_{\bm{y}} \Phi(\bm{w}) & =   \mathcal{A} \Pi_{\mathcal{K}}( \sigma (  \mathcal{A}^*(\bm{y})  + \bm{z} -\bm{c} ) + \bm{x}) -\Pi_{\mathcal{Q}}(\bm{u} -\sigma \bm{y}  ),    \\
    \nabla_{\bm{z}} \Phi(\bm{w}) & =    \Pi_{\mathcal{K}}(  \sigma( \mathcal{A}^*(\bm{y})  + \bm{z} -\bm{c})   + \bm{x})- \prox_{\sigma h}(\bm{q} -\sigma \bm{z}  ),   \\
    \nabla_{\bm{x}} \Phi(\bm{w}) & =    -\frac{1}{\sigma}  \bm{x} + \frac{1}{\sigma} \Pi_{\mathcal{K}}( \sigma (\mathcal{A}^*(\bm{y})  + \bm{z} -\bm{c})  + \bm{x} ),\\
   \nabla_{\bm{u}} \Phi(\bm{w}) & =  -\frac{1}{\sigma}  \bm{u} +  \frac{1}{\sigma} \Pi_{\mathcal{Q}}( \bm{u} -\sigma \bm{y}), \\
  \nabla_{\bm{q}} \Phi(\bm{w}) & =  -\frac{1}{\sigma}  \bm{q} + \frac{1}{\sigma} \prox_{\sigma h}(\bm{q} -\sigma \bm{z}).
    \end{aligned}
\end{equation}

  We focus on the following nonlinear operator
\begin{equation} \label{eq:F}
     F(\bm{w}) :=
     \begin{pmatrix}
          \nabla_{\bm{y}} \Phi(\bm{w})^\top,
          \nabla_{\bm{z}} \Phi(\bm{w})^\top,
          -\nabla_{\bm{x}} \Phi(\bm{w})^\top,
         - \nabla_{\bm{u}} \Phi(\bm{w})^\top,
          -\nabla_{\bm{q}} \Phi(\bm{w})^\top
     \end{pmatrix}^\top.
\end{equation}
It follows from \cite[Lemma 2.5]{deng2023augmented} that $\bm{w}_*$ is a solution of the saddle point problem \eqref{prob:saddle} if and only if it satisfies $F(\bm{w}_*) = 0.$ Consequently, we can solve $F(\bm{w}) = 0$ to find the solution of the original problem \eqref{prob}. Before analyzing the properties of $F$, we give the definition of the Clarke subdifferential and semismoothness, which will play an important role in the subsequent analysis.

 \begin{definition} \label{def:Jacboian}
    Let $F$ be a locally Lipschitz continuous mapping. Denote by $D_F$ the set of differentiable points of $F$. The B-Jacobian of $F$ at $\bm{w}$ is defined by
\[
\partial_B F(\bm{w}) := \left\{\lim_{k \rightarrow \infty} J(\bm{w}^k)\, |\,  \bm{w}^k \in D_F, \bm{w}^k \rightarrow \bm{w}\right\},
\]
where $J(\bm{w})$ denotes the Jacobian of $F$ at $\bm{w} \in D_F$. The set $\partial F(\bm{w})$ = $co(\partial_B F(\bm{w}))$ is called the Clarke subdifferential, where $co$ denotes the convex hull.

The mapping $F$ is called semismooth  at $\bm{w}$ if $F$ is directionally differentiable at $\bm{w}$ and  for any $\bm{d}$, $J \in \partial F(\bm{w}+\bm{d})$, it holds that
$ \| F(\bm{w}+\bm{d}) -  F(\bm{w}) - J\bm{d} \| = o(\|\bm{d}\|), \;\; \bm{d} \rightarrow 0. $
Moreover, $F$ is said to be strongly semismooth at $\bm{w}$ if $F$ is directionally differentiable at $\bm{w}$ and
$\| F(\bm{w}+\bm{d}) -  F(\bm{w}) - J\bm{d} \| = O(\|\bm{d}\|^2), \;\; \bm{d} \rightarrow 0.$
We say $F$ is semismooth (respectively, strongly semismooth) if $F$ is semismooth (respectively, strongly semismooth) for any $\bm{w}$ \citep{mifflin1977semismooth}.
\end{definition}

By \citep[Lemma 2]{deng2023augmented}, we know that $F$ is always monotone, and it is semismooth if $\prox_h$ is semismooth. Unlike the semismooth system constructed in the AL method \citep{zhao2010newton} that relies on the knowledge of the multipliers, our approach implies the optimality conditions directly.  Furthermore, although there are other potential ways to construct semismooth systems induced from the KKT
conditions \citep{liang2023squared}, the resulting systems are generally nonmonotone and even nonsymmetric. Such nonmonotonicity
can raise some practical and theoretical issues including the lack of globalization.
In contrast, our approach to constructing the semismooth system is more unified and does not depend on the integration of first-order methods \citep{li2018semismooth,xiao2018regularized}.

\section{An efficient semismooth Newton method} \label{sec3}
In this section, we present a semismooth Newton algorithm to solve $F(\bm{w}) =0$.

 \subsection{Generalized Jacobian of $F$} We first present the generalized Jacobian operator of $F$.  Denote
 $$
 \begin{aligned}
     \partial \Pi_{\mathcal{Q}}& := \partial  \Pi_{\mathcal{Q}}(\bm{u} -\sigma \bm{y} ), \;\;
     \partial\prox_h   := \partial  \prox_{\sigma h}(\bm{q} -\sigma\bm{z} ),\\
      \partial\Pi_{\mathcal{K}}& := \partial \Pi_{\mathcal{K}} ( \sigma(\mathcal{A}^*(\bm{y})  + \bm{z} -\bm{c} )  + \bm{x})
 \end{aligned}
 $$
as the Clarke subdifferentials of $\Pi_{\mathcal{Q}},  \prox_{\sigma h}$ and $\Pi_{\mathcal{K}}$ at $  \bm{u} - \sigma\bm{y}, \bm{q}  - \sigma\bm{z}, \bm{x} + \sigma( \mathcal{A}^*(\bm{y})  + \bm{z} -\bm{c} )$, respectively.
We also define the following matrix operators:
 \begin{equation} \label{def:H123}
     \mathcal{H}_1 = \left(
     \begin{array}{cc}
\sigma (D_{\mathcal{Q}}+\mathcal{A}D_{\mathcal{K}}\mathcal{A}^*) & \sigma  \mathcal{A}D_{\mathcal{K}}  \\
          \sigma  D_{\mathcal{K}}\mathcal{A}^* & \sigma (D_h +D_{\mathcal{K}})
     \end{array}
     \right),  \mathcal{H}_2 = \left(
     \begin{array}{ccc}
          -\mathcal{A}D_{\mathcal{K}} & D_{\mathcal{Q}} & 0   \\
            -D_{\mathcal{K}}  & 0 & D_h
     \end{array}
     \right),
 \end{equation}
 and
 \[
  \mathcal{H}_3 = \mathrm{blkdiag}\left\{\frac{1}{\sigma} (I - D_{\mathcal{K}}), \frac{1}{\sigma} (I - D_{\mathcal{Q}} ),   \frac{1}{\sigma} (I - D_h )\right\},
 \]
 where $\mathrm{blkdiag}$ denotes the block diagonal matrix operator, $D_{\mathcal{Q}} \in \partial \Pi_{\mathcal{Q}},D_h \in \partial\prox_h  ,$ and $ D_{\mathcal{K}} \in \partial\Pi_{\mathcal{K}}$.
 For any $\bm{w}$,  define
 \begin{equation}\label{equ:jaco}
      \hat{\partial} F(\bm{w}) : = \left\{ \left(
    \begin{array}{cc}
    \mathcal{H}_1     & \mathcal{H}_2  \\
    -\mathcal{H}_2^{\top}     & \mathcal{H}_3   \\
    \end{array}
    \right): {\rm where~} \mathcal{H}_1, \mathcal{H}_2,  {\rm ~and~} \mathcal{H}_3    {\rm ~are~defined~in~} \eqref{def:H123} \right\}.
\end{equation}
 It follows from \citep{hiriart1984generalized} and the definition of $ \hat{\partial} F$ that
$ \hat{\partial} F(\bm{w})[\bm{d}] =  \partial F(\bm{w})[\bm{d}]$ for any $\bm{d}$.
Hence, $\hat{\partial} F(\bm{w})$ can be used to construct the Newton equation for solving $F(\bm{w}) = 0$.

% Given an integer $i$, we define $\bm{d}^{k,i}$ as the solution of the following linear system
% \be \label{eq:ssn} (J^k + \tau_{k,i} I) \bm{d}^{k, i} = -  F(\bm{w}^k) + \bm{\eta}^k, \ee
% where $J^k \in \hat{\partial}F(\bm{w}^k)$ is defined by \eqref{equ:jaco}, $\tau_{k,i} >  0$ is the regularized parameter. Here, $\bm{\eta}^k$ is the residual to measure the inexactness of the linear system. We assume that there exist $C_{\eta}$ such that the residual $\|\bm{\eta}^k\| \le  C_{\eta }k^{-\frac{\beta}{2} }$ for $\beta \in (\frac{1}{3},1)$. The corresponding trial semismooth Newton step is defined by
% \be \label{eq:ssn-step}
%   \bar{\bm{w}}^{k,i}  = \bm{w}^k + \bm{d}^{k,i}.
% \ee
 \subsection{A correction step}

To address the challenge related to the lack of smoothness, we propose a correction step.  Specifically, we define
\begin{equation} \label{eqn:correct}
\begin{aligned}
\hat{\bm{q}} &:= \argmin_{ \hat{\bm{q}}_{i,j} \in \mathcal{C}_{i,j} } \| \bm{q}-\hat{\bm{q}}\| ,\quad  \hat{\bm{u}} := \argmin_{\hat{\bm{u}}_i \in \text{bd}(\mathcal{Q}_i)} \|\bm{u}-\hat{\bm{u}}\|, \\
\hat{\bm{z}} &:= \argmin_{ \hat{\bm{z}}_{i,j} \in - \text{bd}(\partial h_{i,j}(\hat{\bm{q}}_{i,j})) } \| \hat{\bm{z}} - \bm{z} \|,\quad  \hat{\bm{y}} := \argmin_{ \hat{\bm{y}}_i \in -\text{bd}(\mathcal{N}_{\mathcal{Q}_i}(\hat{\bm{u}}_i) ) }\|\hat{\bm{y}} -  \bm{y} \|,
\end{aligned}
\end{equation}
where $h_{i,j}$ is $(i,j)$-th component of $h$,  $\mathcal{Q}_i$ denotes the constraint obtained by restricting the decomposable set $ \mathcal{Q}$ to the $i$-th coordinate, $\mathcal{N}_{\mathcal{Q}_i}$ denotes the normal cone of $\mathcal{Q}_i$, and $\mathcal{C}_{i,j}$ denotes the set of nondifferentiable points of $h_{i,j}$.
It can be seen that $\hat{\bm{q}}_{i,j}$ and $\hat{\bm{u}}_{i,j}$ correspond to the orthogonal projections of $\bm{q}_{i,j}$ and $\bm{u}_{i,j}$ onto the nondifferentiable points of $h_{i,j}$ and $\text{bd}(\mathcal{Q}_{i})$, respectively. Analogously, $\hat{\bm{z}}_{i,j}$ and $\hat{\bm{y}}_i$ are the orthogonal projections of $\bm{z}_{i,j}$ and $\bm{y}_i$ onto the boundaries of 1-dimensional sets $-\partial h_{i,j}(\hat{\bm{q}}_{i,j})$  and $ - \mathcal{N}_{\mathcal{Q}_{i}}(\hat{\bm{u}}_i) )$, respectively. The values of $(\hat{\bm{q}} , \hat{\bm{u}},\hat{\bm{z}},\hat{\bm{y}})$ in \eqref{eqn:correct} can be easily obtained for commonly used $h$ and $\mathcal{Q}$. For example, if $h(\bm{x}) = \delta_{\bm{x} \ge 0}(\bm{x}), $ it follows that $\mathcal{C}_{i,j} = \{0\},\text{bd}(\partial h_{i,j}(\hat{\bm{q}}_{i,j} )) = \{0\}$. Therefore,   $\hat{\bm{q}} = \hat{\bm{z}} = \bm{0} \in \mathbb{R}^{n \times n}$ by \eqref{eqn:correct}.

For given $\sigma$ and $\bm{w} = (\bm{y},\bm{z},\bm{x},\bm{u},\bm{q} ),$ we denote $ Q \Lambda Q^*$  as the eigenvalue decomposition of $\bm{x} + \sigma(\mathcal{A}^*(\bm{y}) + \bm{z} - \bm{c})$, $\lambda_i = \Lambda_{i,i}$ as the $i$-th eigenvalue, and $Q_i$ as the associated eigenvector of $\lambda_i$.
Let $\bm{\upsilon } :=   \bm{q} - \hat{\bm{q}} - \sigma(\bm{z} + \hat{\bm{z}}) $ and
 $  \bm{\vartheta} :=   \bm{u} - \hat{\bm{u}} - \sigma(\bm{y} + \hat{\bm{y}})$.
Combining the definitions of $(\hat{\bm{q}} , \hat{\bm{u}},\hat{\bm{z}},\hat{\bm{y}})$ in \eqref{eqn:correct},
the correction step with nonnegative constants $\theta, l$, and $\rho$ is constructed as:
\begin{equation} \label{def:project}
\begin{aligned}
 \tilde{\bm{w}}
=\mathcal{P}_{\theta,l,\rho,\sigma}(\bm{y},\bm{z},\bm{x},\bm{u},\bm{q})
  := (\bm{y},  \bm{z},\bm{x}- \sum_{i \in I_{\mathcal{K} }} \lambda_i Q_i Q_i^*, \bm{u} - \sum_{i \in I_{Q} }  \bm{\vartheta}_i , \bm{q} - \sum_{(i,j) \in I_h}\bm{\upsilon}_{ij} ) ,
\end{aligned}
\end{equation}
where
\[
\begin{aligned}
I_{\mathcal{K}} &:= \left\{i| |\lambda_i| < \frac{\theta}{2} \right\} , I_h = \left\{ (i,j)|  | \bm{z}_{i,j}- \hat{\bm{z}}_{i,j}| < \frac{l}{2\sigma},\,  |\bm{q}_{i,j}-\hat{\bm{q}}_{i,j}| < \frac{l}{2 } \right\}, \\
I_{\mathcal{Q} } &:= \left\{i| | \bm{y}_i-  \hat{\bm{y}}_i| \le  \frac{\rho}{2\sigma}, |\bm{u}_i- \hat{\bm{u}}_i| < \frac{\rho}{2}  \right\}.
\end{aligned}
\]
To provide a more concise explanation, $\mathcal{P}_{\theta,l,\rho,\sigma}$ adds a correction term on $\bm{x}$ to threshold the eigenvalues of
$\bm{x} + \sigma(\mathcal{A}^*(\bm{y}) + \bm{z} -\bm{c})$ with value $\theta/2$ on their absolute values. A similar approach has also been adopted in solving nonlinear SDP problems \citep{feng2024quadratically}. The operator also performs truncations on $\bm{u}$ and $\bm{q}$ to ensure that the last two components, $- \nabla_{\bm{u}} \Phi(\bm{w})$ and $-\nabla_{\bm{q}} \Phi(\bm{w})$, have desired smoothness within a specified regime. Note that if $(\theta,l,\rho) $ is zero, then $\mathcal{P}_{\theta,l,\rho,\sigma}$ reduces to the identity operator.  Furthermore, since the Clarke subdifferentials of $\Pi_{\mathcal{K}},\text{prox}_{\sigma h}$, and $\Pi_{\mathcal{Q}}$ are all computed at each step, no additional computations arise when carrying out  $\mathcal{P}_{\theta,l,\rho,\sigma}$ operation.

Essentially speaking, the main idea of $\mathcal{P}_{\theta,l,\rho,\sigma}$ is to project $\bm{w}$ onto a manifold where $F$ is locally smooth when $\bm{w}$ is close to the solution set, see Lemma \ref{lem:local-smooth} for details. In early iteration, it can be the identity operator. This local smoothness allows us to show the local transition to the
superlinear convergence. Note that the local superlinear convergence of the update \eqref{eq:ssn-step} in \citep{hu2025analysis} relies on the SC at the solution point, which does not necessarily hold.
Designing a superlinear convergent semismooth Newton method is challenging without assuming the SC and BD regularity of $F$.

 \subsection{A globalized semismooth Newton method}

To globalize the semismooth Newton method, we adopt a nonmonotone line search strategy on the residual sequence $\{\|F(\bm{w}^k)\|\}$. Given $\nu \in (0,1),$ $\beta \in (1/3,1]$, and a fixed integer $i_{\max} \ge 0$, we search over $i = 0, \dots, i_{\max}$ to find a suitable regularization parameter
$
\tau_{k,i} = \kappa \gamma^i \|F(\bm{w}^k)\|,
$
where $\kappa > 0$ and $\gamma > 1$ are fixed constants.
For each $i$, we define the direction $\bm{d}^{k,i}$ as the solution to the linear system
\be \label{eq:ssn}
(J^k + \tau_{k,i} I) \bm{d}^{k,i} = -F(\bm{w}^k) + \bm{\eta}^k,
\ee
where $J^k \in \hat{\partial}F(\bm{w}^k)$ is defined by \eqref{equ:jaco}, and $\bm{\eta}^k$ is a residual vector representing the inexactness. We require that there exists a constant $C_\eta > 0$ such that
$
\|\bm{\eta}^k\| \le C_\eta k^{-\beta}.
$
The corresponding trial point is then defined as
\be \label{eq:ssn-step}
\bar{\bm{w}}^{k,i} = \bm{w}^k + \bm{d}^{k,i}.
\ee

Next, for $i=0,\dots, i_{\max}$, we evaluate whether the candidate update satisfies the nonmonotone line search condition:
\begin{equation} \label{eq:decrease-1}
\|F(\tilde{\bm{w}}^{k,i})\| \le \nu \max_{\max(1,k-\zeta+1)\le j \le k} \|F(\bm{w}^j)\| + \varsigma_k,
\end{equation}
where $\tilde{\bm{w}}^{k,i} = \mathcal{P}_{\theta, l, \rho, \sigma}(\bar{\bm{w}}^{k,i})$ and $\{\varsigma_k\}_{k=1}^\infty$ is a nonnegative sequence satisfying $\sum_{k=1}^\infty \varsigma_k^2 < \infty.$
If \eqref{eq:decrease-1} is satisfied for some $i$, we stop searching $i$ and set the next iterate as $\bm{w}^{k+1} = \tilde{\bm{w}}^{k,i}$. If \eqref{eq:decrease-1} fails for all $i \le i_{\max}$, we set
\begin{equation} \label{eq:decrease-2}
 \tau_{k,i} = \kappa_1 k^\beta,
\end{equation}
where $\kappa_1 \geq 1$ is a given constant, and set $\bm{w}^{k+1} = \bar{\bm{w}}^{k,i}$ defined in \eqref{eq:ssn-step}.
% This ensures that each iteration satisfies either (3.7) or (3.8), guaranteeing a well-defined globalization mechanism.
%  Regarding the globalization, we adopt a nonmonotone line search on the residual $\{\|F(\bm{w}^k)\| \}$. Specifically, for $\nu \in (0,1)$ and $\beta \in (1/2, 1)$, we find the appropriate $\tau_{k,i}$ such that one of the following conditions holds:
% \begin{align}
%     \|F(\tilde{\bm{w}}^{k,i})\|  & \leq  \nu \max_{\max(1, k-\zeta+1) \leq j \leq k}\|F(\bm{w}^j)\| + \varsigma_k, \label{eq:decrease-1} \\
%    \kappa_1 k^{\beta} &\le \tau_{k,i}  \le \gamma \kappa_1 k^{\beta} \label{eq:decrease-2},
% \end{align}
% where $\tilde{\bm{w}}^{k,i} =\mathcal{P}_{\theta,l,\rho,\sigma}(\bar{\bm{w}}^{k,i})$ and $\{\varsigma_k\}_{k=1}^{\infty}$ is a nonnegative sequence that satisfies $\sum_{k=1}^{\infty} \varsigma_k < \infty$, $\kappa_1 > 1$ is given constants. For given $i_{\max} > 0$ and $i = 0,\cdots,i_{\max}$, we set $\tau_{k,i} = \kappa \gamma^i \|F(\bm{w}^k)\|$ with constants $\kappa > 0$ and $\gamma > 1$.
% If \eqref{eq:decrease-1} holds, we set $\bm{w}^{k+1} = \tilde{\bm{w}}^{k,i}$ with the smallest $i$. Otherwise,  we choose appropriate $\tau_{k,i}$ such that \eqref{eq:decrease-2} holds and set $\bm{w}^{k+1} = \bar{\bm{w}}^{k,i}$ which is defined in \eqref{eq:ssn-step}.
Conditions \eqref{eq:decrease-1} or \eqref{eq:decrease-2} require that either the norm of $F$ evaluated at $\bar{\bm{w}}^{k,i}$ has sufficient decrease property or $\tau_{k,i} $ is chosen as in \eqref{eq:decrease-2}.

The line search over the regularization parameter ensures that a regularized second-order step is always used.
 Since the condition \eqref{eq:decrease-1} is quite loose, the regularized parameter $\tau_{k,i}$ may not be too large in practice, ensuring that the steps closely approximate exact semismooth Newton steps in most cases.   As will be shown later, \eqref{eq:decrease-1} eventually holds with $i = 0$ provided certain conditions (e.g., local error bound) are satisfied, which eliminates the need to search over $i$. Moreover, superlinear convergence is achieved, and \eqref{eq:decrease-2} occurs only finite times in this case.  In numerical experiments, we also find that in most cases even though $ \|F(\bm{w})\| $ may increase in the current step,  it will decrease in the next few steps. Consequently,  condition \eqref{eq:decrease-1} is effective and efficient both in theory and numerics.

\begin{algorithm}[h]
\caption{A semismooth Newton method for solving \eqref{prob}} \label{alg:ssn}
\begin{algorithmic}[1]
\REQUIRE The constants $\kappa_1>0$, $\gamma > 1$, $\nu \in (0,1)$, $\beta \in (1/3, 1]$, $\kappa >0$, positive integers $\zeta, i_{\max}$, nonnegative sequence $\{\varsigma_k\}, \{ \theta^k\}, \{l^k \}, \{\rho^k\}$, an initial point $\bm{w}^0$ and set $k = 0$.
\WHILE {\emph{stopping condition not met}}
\STATE Calculate $F(\bm{w}^k)$ and select $J(\bm{w}^k) \in \partial F(\bm{w}^k)$.
\STATE
Find $\tau_{k,i}$ such that  either \eqref{eq:decrease-1} holds with $\tilde{\bm{w}}^{k,i} = \mathcal{P}_{\theta^k,l^k,\rho^k,\sigma}(\bar{\bm{w}}^{k, i})$ or \eqref{eq:decrease-2} holds.

\STATE Set $\bm{w}^{k+1} = \tilde{\bm{w}}^{k,i}$ if \eqref{eq:decrease-1} holds; otherwise, define $\bm{w}^{k+1} = \bar{\bm{w}}^{k,i}$ as in \eqref{eq:ssn-step}.

\STATE Set $k=k+1$.
\ENDWHILE
\end{algorithmic}
\end{algorithm}

\subsection{An efficient implementation of Newton step}
When Algorithm \ref{alg:ssn} is applied to solve $F(\bm{w}) = 0$, the key part is to obtain the search direction $\bm{d}^k$ from the linear system. In this subsection, we show how to solve the linear system efficiently.
  Let $ \bm{d} = (d_{\bm{y}};d_{\bm{z}};d_{\bm{x}};d_{\bm{u}};d_{\bm{q}})$, which corresponds to the directions of variables $\bm{y},\bm{z},\bm{x},\bm{u},\bm{q}$, respectively.  We consider the following Newton  equation:
\begin{equation}\label{equ:linear}
    (J^k + \tau_k I) \bm{d}^k = -  \tilde{F}^k,
\end{equation}
where $\tilde{F}^k =  F(\bm{w}^k) - \bm{\eta}^k$.
Denote $D^{\tau}_{\mathcal{K}} = ((\frac{1}{\sigma }+ \tau_{\bm{x}}) \mathcal{I}- \frac{1}{\sigma }  D_{\mathcal{K}})$, $\tilde{D}_{\mathcal{K}} = D_{\mathcal{K}}(D^{\tau}_{\mathcal{K}})^{-1}D_{\mathcal{K}}  $. $D^{\tau}_{h}, \tilde{D}_h, D^{\tau}_{\mathcal{Q}}$, and $\tilde{D}_{\mathcal{Q}}$ can be defined analogously.
It follows from the proper, closedness, and convexity of $h$, $\delta_{\mathcal{Q}}$, and $\delta_{\mathcal{K}}$, that $\mathcal{I} - D_h$, $\mathcal{I} - D_{\mathcal{Q}}$, and $\mathcal{I} - D_{\mathcal{K}}$ are positive semidefinite \citep[Lemma 3.3.5]{milzarek2016numerical}, implying that $D_h^\tau$, $D_{\mathcal{Q}}^\tau$, and $D_{\mathcal{K}}^\tau$ are nonsingular.
% It follows from $h, \delta_{\mathcal{Q}}, \delta_{\mathcal{K}}$ are proper closed and \cite[Lemma 3.3.5]{milzarek2016numerical} that the eigenvalue of $D_{\mathcal{K}}, D_{\mathcal{Q}}$ and $D_h$ are in [0,1]. Hence $\mathcal{I} -  D_{\mathcal{K}}, \mathcal{I} - D_{\mathcal{Q}}$, and $\mathcal{I} - D_h$ are positive semidefinite and henceforce, $D^{\tau}_{\mathcal{K}}, D^{\tau}_{h}$, and $D^{\tau}_{\mathcal{Q}}$ are nonsingular.

After Gaussian elimination and ignoring superscript $k$,
if $d_{\bm{y}}$ and $d_{\bm{z}}$ are obtained,
one can compute the explicit solution of $d_{\bm{x}},d_{\bm{u}},d_{\bm{q}}$ by:
\begin{equation}\label{eq:d3}
\begin{aligned}
    d_{\bm{x}} &= (D^{\tau_{\bm{x}}}_{\mathcal{K}})^{-1}( -D_{\mathcal{K}}\mathcal{A}^*d_{\bm{y}} - D_{\mathcal{K}}d_{\bm{z}} - \tilde{F}_{\bm{x}}), \\
    d_{\bm{u}} & = (D^{\tau_{\bm{u}}}_{\mathcal{Q}})^{-1}(D_{\mathcal{Q}}d_{\bm{y}} - \tilde{F}_{\bm{u}}), \quad \quad d_{\bm{q}}  = (D^{\tau_{\bm{q}}}_{h})^{-1}(D_h d_{\bm{z}} - \tilde{F}_{\bm{q}}).
\end{aligned}
\end{equation}
Then, the linear system \eqref{equ:linear} is reduced to
 \begin{equation}\label{eq:d12}
     \left(
    \begin{array}{cccc}
     \mathcal{M}_1 & \mathcal{M}_2  \\
       \mathcal{M}_3 & \mathcal{M}_4
    \end{array}
    \right) \left(  \begin{array}{ccc}
         d_{\bm{y}}  \\
         d_{\bm{z}}
    \end{array} \right)  = \left(  \begin{array}{ccc}
        \tilde{R}_{\bm{y}}  \\
        \tilde{R}_{\bm{z}}
    \end{array} \right),
 \end{equation}
 where
 \begin{equation*}
     \begin{aligned}
      \mathcal{M}_1 & = \overline{D}_{\mathcal{Q}}  + \tau_{\bm{y}} \mathcal{I} + \mathcal{A}\overline{D}_{\mathcal{K}} \mathcal{A}^* , ~~   \mathcal{M}_2 = \mathcal{A}\overline{D}_{\mathcal{K}},~~
      \mathcal{M}_3  = \overline{D}_{\mathcal{K}}\mathcal{A}^*,~~\mathcal{M}_4  =   \tau_{\bm{z}} \mathcal{I} + \overline{D}_{h} + \overline{D}_{\mathcal{K}}, \\
      \overline{D}_{\mathcal{K}} &= \sigma D_{\mathcal{K}} + \tilde{D}_{\mathcal{K}}, \quad
      \overline{D}_{h} = \sigma D_{h} + \tilde{D}_{h}, \quad
      \overline{D}_{\mathcal{Q}} = \sigma D_{\mathcal{Q}} + \tilde{D}_{\mathcal{Q}},
      \\ \tilde{R}_{\bm{y}}  & = -\mathcal{A}D_{\mathcal{K}}(D^{\tau}_{\mathcal{K}})^{-1} \tilde{F}_{\bm{x}} + D_{\mathcal{Q}}(D^{\tau}_{\mathcal{Q}})^{-1}\tilde{F}_{\bm{u}}(\bm{w}) -  \tilde{F}_{\bm{y}}, \\
      \tilde{R}_{\bm{z}} & = -  D_{\mathcal{K}}(D^{\tau}_{\mathcal{K}})^{-1} \tilde{F}_{\bm{x}}  + D_h  (D^{\tau}_{h})^{-1} \tilde{F}_{\bm{q}} - \tilde{F}_{\bm{z}}.
     \end{aligned}
 \end{equation*}
 Specifically, the explicit formulation of equation \eqref{eq:d12} can be represented by:
  \begin{equation} \label{eqn:dydz}
  \left[ \left(
    \begin{array}{cccc}
     \tau_{\bm{y}} \mathcal{I} + \overline{D}_{\mathcal{Q}}& 0  \\
       0 & \tau_{\bm{z}} \mathcal{I} +\overline{D}_{h}
    \end{array}
    \right)+\left(
    \begin{array}{c}
     \mathcal{A}   \\
        \mathcal{I}
    \end{array}
    \right) \overline{D}_{\mathcal{K}}  \left(\mathcal{A}^* ,\mathcal{I} \right) \right]
      \left(  \begin{array}{ccc}
         d_{\bm{y}}  \\
         d_{\bm{z}}
    \end{array} \right)  = \left(  \begin{array}{ccc}
        \tilde{R}_{\bm{y}}  \\
        \tilde{R}_{\bm{z}}
    \end{array} \right).
 \end{equation}
It follows from the definition of $D_{\mathcal{Q}}$ that
\[
(D_{\mathcal{Q}})_i = \begin{cases}
    1, & \mbox{if}\,\, (\bm{u} - \sigma \bm{y})_i \in \mathcal{Q}_i,\\
    0, & \mbox{else}.
\end{cases}
\]
Since $h$ is an elementwise function, $D_h$ is also elementwise. Hence, the main computation is to calculate $\overline{D}_{\mathcal{K}} (\mathcal{A}^*d_{\bm{y}} + d_{\bm{z}} )$ once for every matrix-vector operation.

Let $\bm{y},\bm{z}$ and $\bm{x}$ be fixed. Consider the following eigenvalue decomposition:
\begin{equation}\label{eq:sdp+:decompo}
 \bm{x} + \sigma(\mathcal{A}^*(\bm{y}) + \bm{z}   -\bm{c})     = Q \Gamma_{\bm{y}} Q^{\mathrm{T}},
\end{equation}
where $\Gamma_{\bm{y}}$ is the diagonal matrix of eigenvalues and  $Q \in \R^{n \times n}$ is the eigenvector matrix. The diagonal elements of $\Gamma_{\bm{y}}$ are arranged in the nonincreasing order: $\lambda_1 \geq \lambda_2 \geq \cdots \geq \lambda_n$. Define the following index sets $ \alpha:=\left\{i \mid \lambda_i>0\right\}, \; \bar{\alpha}:=\left\{i \mid \lambda_i \leq 0\right\}. $
Then we have that the operator $D_{\mathcal{K}}: \mathbb{S}^n \rightarrow \mathbb{S}^n$ satisfies
\[
D_{\mathcal{K}}(H):=Q\left(\Sigma \circ\left(Q^{\mathrm{T}} H Q\right)\right) Q^{\mathrm{T}}, \quad H \in \mathbb{S}^n,
\]
where $\circ$ denotes the Hadamard product of two matrices and
\begin{equation}\label{eq:sdp+:sigma}
    \Sigma=\left[\begin{array}{cc}
E_{\alpha \alpha} & v_{\alpha \bar{\alpha}} \\
v_{\alpha \bar{\alpha}}^{\mathrm{T}} & 0
\end{array}\right], \quad v_{i j}:=\frac{\lambda_i}{\lambda_i-\lambda_j}, \quad i \in \alpha, \quad j \in \bar{\alpha},
\end{equation}
where $E_{\alpha \alpha} \in \mathbb{S}^{|\alpha|}$ is the matrix of ones. Note that the matrix form of $D_{\mathcal{K}}$  can be expressed as:
$
\textbf{mat}(D_{\mathcal{K}})=\tilde{Q} \Lambda \tilde{Q}^T,
$
where \textbf{mat} denotes the matrix form of the operator, $\tilde{Q} = Q \otimes Q, \Lambda = \operatorname{diag}(\operatorname{vec}(\Sigma)).$  Consequently, we have $\mathcal{A} \overline{D}_{\mathcal{K}} \mathcal{A}^* d_{\bm{y}}
 = \mathcal{A} Q\left(( \overline{\Sigma} ) \circ\left(Q^{\mathrm{T}} (\mathcal{A}^* d_{\bm{y}}) Q\right)\right) Q^{\mathrm{T}}$, where
\begin{equation} \label{hatsig}
  \overline{\Sigma} := \left[\begin{array}{cc}
\frac{ (1 + \sigma\tau_{\bm{x}})}{\tau_{\bm{x}}} E_{\alpha \alpha} & l_{\alpha \bar{\alpha}} \\
l_{\alpha \bar{\alpha}}^{\mathrm{T}} & 0
\end{array}\right], \quad l_{i j}:=\frac{   \sigma(\sigma\tau_{\bm{x}} +1 )v_{ij}}{1+\sigma\tau_{\bm{x}} - v_{ij}}.
\end{equation}
Note that $Y$ can be represented by:
\[
Y = [Q_\alpha \, Q_{\bar{\alpha}}]  \left[\begin{array}{cc}
\frac{1}{\tau} Q_\alpha^{\mathrm{T}} D_{\mathcal{K}} Q_\alpha & \nu_{\alpha \bar{\alpha}} \circ (Q_\alpha^{\mathrm{T}} D_{\mathcal{K}} Q_{\bar{\alpha}} ) \\
\nu_{\alpha \bar{\alpha}}^{\mathrm{T}} \circ (Q_{\bar{\alpha}}^{\mathrm{T}} D_{\mathcal{K}} Q_\alpha) & 0
\end{array}\right]
\left[\begin{array}{c} Q_\alpha^{\mathrm{T}}\\ Q_{\bar{\alpha}}^{\mathrm{T}} \end{array}\right] = H + H^{\mathrm{T}},
\]
where $H= Q_\alpha[\frac{1}{2\tau_{\bm{x}}}(UQ_\alpha)Q_\alpha^{\mathrm{T}}+(\nu_{\alpha \bar{\alpha}}) \circ (UQ_{\bar{\alpha}})Q_{\bar{\alpha}}^{\mathrm{T}}]$ with $U = Q_\alpha^{\mathrm{T}}D_{\mathcal{K}}$. If $\alpha > \frac{n}{2},$ by letting $Y = \frac{1}{\tau_{\bm{x}}}D_{\mathcal{K}} - Q((\frac{1}{\tau_{\bm{x}}} E-\Omega)\circ (Q^{\mathrm{T}}D_{\mathcal{K}}Q))Q^{\mathrm{T}}$, the high rank property can also be used. Hence, $Y$ can be computed in at most $8 \min\{|\alpha|,|\bar{\alpha}|\}n^2$ flops. The linear system \eqref{eqn:dydz} can be efficiently solved to obtain $d_{\bm{y}}$, $d_{\bm{z}}$. Note that the following identity holds:
 \be
 (D^{\tau_{\bm{x}}}_{\mathcal{K}})^{-1} = ((\frac{1}{\sigma} + \tau_{\bm{x}}) I - \frac{1}{\sigma} D_{\mathcal{K}})^{-1} = \frac{\sigma}{1 + \sigma\tau_{\bm{x}}} I + \frac{1}{1+ \sigma\tau_{\bm{x}}}T,
 \ee
 where $T = \tilde{Q} L \tilde{Q}^T$ and $L_{ii}=\frac{\sigma \lambda_{i}}{1+\sigma\tau_{\bm{x}}-\lambda_i}.$ Then, it follows that:
 \[
 (D_{\mathcal{K}}^{\tau}  )^{-1}(H)  = Q\left(\Sigma_{\tau_{\bm{x}}} \circ\left(Q^{\mathrm{T}} H Q\right)\right) Q^{\mathrm{T}} = \frac{\sigma}{1+ \sigma\tau_{\bm{x}}} H +  \frac{1}{1 + \sigma\tau_{\bm{x}}} Q\left(\Sigma_{T} \circ\left(Q^{\mathrm{T}} H Q\right)\right) Q^{\mathrm{T}},
 \]
 where $\Sigma_{T}   =\left[\begin{array}{cc}
\frac{1}{\tau_{\bm{x}}} E_{\alpha \alpha} & k_{\alpha \bar{\alpha}} \\
k_{\alpha \bar{\alpha}}^{\mathrm{T}} & 0
\end{array}\right], \; k_{i j}:=\frac{\sigma v_{ij} }{1 + \sigma \tau_{\bm{x}} -   v_{ij}  }. \nonumber$
Consequently, the low-rank or high-rank property can also be used when computing $d_{\bm{x}}$.

For the standard SDP problem, since the corresponding problem involves only variables $\bm{y}$ and $\bm{s}$, it follows that $\bm{w}$ and $\varphi(\bm{w})$ reduce to $(\bm{y},\bm{x})$ and  $\Pi_{\mathcal{K}}(   \mathcal{A}^*(\bm{y})  -\bm{c} + \frac{1}{\sigma}\bm{x})$, respectively. The corresponding Newton system is reduced to:

 \begin{equation}\label{eq:d122}
     \left(
    \begin{array}{cccc}
     \mathcal{N}_1 & \mathcal{N}_2  \\
       \mathcal{N}_3 & \mathcal{N}_4
    \end{array}
    \right) \left(  \begin{array}{ccc}
         \bm{d_y}  \\
         \bm{d_x}
    \end{array} \right)  = \left(  \begin{array}{ccc}
        -\tilde{F}_{\bm{y}}  \\
        -\tilde{F}_{\bm{x}}
    \end{array} \right),
 \end{equation}
where $\bm{d}_{\bm{y}} \in \R^m, \bm{d}_{\bm{x}} \in \R^{n \times n}.$ $\mathcal{N}_1 = \sigma\mathcal{A}D_{\mathcal{K}}\mathcal{A}^* + \tau_{\bm{y}} I, \mathcal{N}_2 =   \mathcal{A}D_{\mathcal{K}} ,\mathcal{N}_3 = -D_{\mathcal{K}} \mathcal{A}^*, \mathcal{N}_4 =   \frac{1}{\sigma} (I- D_{\mathcal{K}}) + \tau_{\bm{x}} I = D_{\mathcal{K}}^{\tau_{\bm{x}}},$
 $\tilde{F}_{\bm{y}} = - \bm{b} + \sigma\mathcal{A} \varphi(\bm{w}) - \bm{\eta}_{\bm{y}} $ and $\tilde{F}_{\bm{x}} =   -  \varphi(\bm{w}) + \frac{1}{\sigma}  \bm{x} - \bm{\eta}_{\bm{x}} .$ Consequently, we only need to solve the following equation corresponding to $\bm{d}_{\bm{y}}$:
 \begin{equation} \label{eq:d1}
 \begin{aligned}
 & \mathcal{A} Q\left(( \overline{\Sigma} ) \circ\left(Q^{\mathrm{T}} (\mathcal{A}^* \bm{d}_{\bm{y}}) Q\right)\right) Q^{\mathrm{T}}  + \tau_{\bm{y}} \bm{d}_{\bm{y}} = - \tilde{F}_{\bm{y}} + \mathcal{N}_2(\mathcal{N}_4)^{-1} \tilde{F}_{\bm{x}},
  \end{aligned}
 \end{equation}
 where $\overline{\Sigma}$ is defined in \eqref{hatsig}. For classical SDP+ problems, the variables in AL function \eqref{eq:alfunc} are reduced to $(\bm{y},\bm{z},\bm{x},\bm{q})$. In this case, the computational process is similar to \eqref{eqn:dydz}, except that $\overline{D}_{\mathcal{Q}}$ is zero.

\section{Convergence analysis} \label{sec4}
In this section, we present the global, local convergence, and iteration complexity analysis of Algorithm \ref{alg:ssn}.

\subsection{Global convergence}
% The global convergence is established using two types of descent. The first is an explicit decrease in the residual norm $\|F(\bm{w}^k)\|$ from condition~\eqref{eq:decrease-1}, and an implicit decrease from the control of $\tau_{k,i}$ in condition~\eqref{eq:decrease-2}, which contributes even when the residual does not decrease immediately. Specifically, when the correction step $\mathcal{P}_{\theta^k, l^k, \rho^k, \sigma}(\bar{\bm{w}}^{k,i})$ is accepted, condition~\eqref{eq:decrease-1} ensures a reduction in $\|F(\bm{w}^k)\|$. If not, the inexact first-order step maintains $\tau_{k,i}$ within the desired range, guiding the iterates toward optimality. The following lemma shows that the sequence ${\|F(\bm{w}^k)\|}$ generated by Algorithm~\ref{alg:ssn} is bounded.

The proof of global convergence relies on two distinct types of descent to establish global convergence. The first is an explicit decrease in the residual norm $\|F(\bm{w}^k)\|$, directly enforced by condition~\eqref{eq:decrease-1}. The second is an implicit decrease, arising from the control of the regularization parameter $\tau_{k,i}$ under condition~\eqref{eq:decrease-2}, which indirectly contributes to convergence even when the residual does not decrease immediately. Specifically, when the correction step $\mathcal{P}_{\theta^k, l^k, \rho^k, \sigma}(\bar{\bm{w}}^{k,i})$ is accepted, condition~\eqref{eq:decrease-1} guarantees a measurable reduction in $\|F(\bm{w}^k)\|$. On the other hand, if the correction step is not applied, the inexact first-order step governed by \eqref{eq:decrease-2} helps maintain $\tau_{k,i}$ within a desirable range, thereby guiding the iterates toward optimality.
 We first prove in the following lemma that the sequence $\{\|F(\bm{w}^k)\|\}$ generated by Algorithm \ref{alg:ssn} is bounded.
\begin{lemma} \label{lem:glo}
 Suppose that Assumption \ref{assum} holds.  Let $\{\bm{w}^{k}\}_{k \ge 0}$ be the iteration sequence  generated by Algorithm \ref{alg:ssn}.  It holds that $\| F(\bm{w}^k)\|$ is bounded from above.
  If the update \eqref{eq:decrease-2} is conducted at the $(k+1)$-th iteration, the following inequality holds:
    \be \label{eq:residual-increase2} \| F(\bm{w}^{k+1})\|^2 \leq \|F(\bm{w}^k)\|^2 + \frac{M}{k^{3\beta}},  \ee
    where $M > 0$ is a constant independent of $k$.
\end{lemma}

\begin{proof}
The residual mapping $F$ defined in \eqref{eq:F} is monotone and $L$-Lipschitz continuous.
Define $\bar{F} = F/L$. The operator $\bar{T}:= I - \bar{F}$ is nonexpansive. For ease of analysis and without loss of generality, we use $F$ and $T$ to denote $\bar{F}$ and $\bar{T}$, respectively. Note that the $\lambda$-averaged operator of $T$ is $T_{\lambda} :=(1-\lambda ) I + \lambda T$. Let us denote $\hat{\bm{w}}^{k+1} = T_{\lambda}(\bm{w}^k) = \bm{w}^k - \lambda F(\bm{w}^k)$. The Newton update \eqref{eq:ssn-step} can be written as
\begin{equation} \label{eqn:updateX2}
\begin{aligned}
\bm{w}^k + \bm{d}^{k,i} &= \bm{w}^k - (J^k + \tau_{k,i} I)^{-1} (F(\bm{w}^k) - \bm{\eta}^k) \\
& = \bm{w}^k - \tau_{k,i}^{-1}F(\bm{w}^k) + (J^k + \tau_{k,i} I)^{-1}\bm{\eta}^k  + (\tau_{k,i}^{-1}I - (J^k + \tau_{k,i} I)^{-1}) F(\bm{w}^k) \\
&= T_{\tau_{k,i}^{-1}}(\bm{w}^k) + (J^k + \tau_{k,i} I)^{-1}\bm{\eta}^k  + (\tau_{k,i}^{-1}I - (J^k + \tau_{k,i} I)^{-1}) F(\bm{w}^k).
\end{aligned}
\end{equation}
Define $\bm{r}^k = (\tau_{k,i}^{-1}I - (J^k + \tau_{k,i} I)^{-1}) F(\bm{w}^k) $. It holds that
\[ \begin{aligned}
\|\bm{r}^k\| & = \|(J^k + \tau_{k,i} I)^{-1}\left(\tau_{k,i}^{-1} (J^k + \tau_{k,i}I) - I \right) F(\bm{w}^k) \|     \\ & \leq \| (J^k + \tau_{k,i} I)^{-1} \tau_{k,i}^{-1} J^k (F(\bm{w}^k )  \|   \leq  \tau_{k,i}^{-2} L \|F(\bm{w}^k)\|.
\end{aligned}\]
By using the nonexpansiveness of $T$, we have
\begin{equation} \label{eqn:updateX22}
\begin{aligned}
&- 2\iprod{F(\bm{w}^{k+1}) - F(\bm{w}^{k})}{\bm{w}^{k+1} - \bm{w}^{k}} \\
= & \| T(\bm{w}^{k+1}) - T(\bm{w}^{k}) \|^2 - \|\bm{w}^{k+1} - \bm{w}^{k}\|^2 - \|F(\bm{w}^{k+1}) - F(\bm{w}^{k})\|^2 \\
\leq & - \|F(\bm{w}^{k+1}) - F(\bm{w}^{k})\|^2.
\end{aligned}
\end{equation}
Define $ \tilde{\bm{\eta}}^k = (J_k + \tau_{k,i}I)^{-1} \bm{\eta}^k$. If $\tau_{k,i} = \kappa_1 k^\beta$ (i.e., \eqref{eq:decrease-1} fails) and $k \geq 4$, we have
\be \label{eq:residual-increase}
\begin{aligned}
& \| F(\bm{w}^{k+1})\|^2  = \| F(\bm{w}^{k})\|^2 + \|F(\bm{w}^{k+1}) - F(\bm{w}^{k})\|^2 + 2 \iprod{F(\bm{w}^{k+1}) - F(\bm{w}^{k})}{F(\bm{w}^{k})}\\
\overset{\eqref{eqn:updateX2}}{=} &\, \| F(\bm{w}^{k})\|^2 + \|F(\bm{w}^{k+1}) - F(\bm{w}^{k})\|^2  - 2\tau_{k,i} \iprod{F(\bm{w}^{k+1}) - F(\bm{w}^{k})}{ \bm{w}^{k+1} - \bm{w}^{k} -\bm{r}^k - \tilde{\bm{\eta}}^k }\\
\overset{\eqref{eqn:updateX22}}{ \le} &\, \| F(\bm{w}^{k})\|^2 - ( \tau_{k,i} - 1) \|F(\bm{w}^{k+1}) - F(\bm{w}^{k})\|^2 + 2\tau_{k,i} \iprod{F(\bm{w}^{k+1}) - F(\bm{w}^k) }{\bm{r}^k +  \tilde{\bm{\eta}}^k }  \\
= &\, \| F(\bm{w}^{k})\|^2 - (\tau_{k,i} - 1)\| F(\bm{w}^{k+1}) - F(\bm{w}^{k}) + \frac{\tau_{k,i}}{\tau_{k,i} - 1} (\bm{r}^k + \tilde{\bm{\eta}}^k) \|^2 + \frac{\tau_{k,i}^2}{\tau_{k,i} - 1}\| \bm{r}^k + \tilde{\bm{\eta}}^k   \|^2 \\
 \leq &\, \|F(\bm{w}^k)\|^2 +  \frac{2 L^2 }{\tau_{k,i}^{3}(1-\tau_{k,i}^{-1})} \|F(\bm{w}^k)\|^2 + \frac{2}{\tau_{k,i} -1} \| \bm{\eta}^k \|^2  \\
  \leq &\,  \left( 1 + \frac{4 \kappa_1^2 L^2}{k^{3\beta}} \right) \|F(\bm{w}^k)\|^2 + \frac{4 \kappa_1^2 C_{\eta}^2}{k^{3\beta}},
\end{aligned}
\ee
where the second inequality is due to $ \|\tilde{\bm{\eta}}^k\| \le \tau_{k,i}^{-1} \|\bm{\eta}^k\|$ and the last inequality follows from $\tau_{k,i} -1 = \kappa_1 k^{\beta} -1 \ge \frac{1}{2\kappa_1} k^{\beta}$.
If \eqref{eq:decrease-1} holds, it follows that
\be \label{eq:residual-decrease}
\| F(\bm{w}^{k+1})\| \leq \nu \max_{\max(1, k-\zeta + 1) \leq j \leq k}\| F(\bm{w}^{j})\| + \varsigma_k.
\ee
 We next prove the following assertion using mathematical induction: there exists a $c > 1$ such that for $k > 0$,
\be \label{globalcov:assertion}
\|F(\bm{w}^{k})\|^2 \le c\,\Pi_{n=1}^{k} \left(1 + \frac{ \tilde{L} }{n^{3\beta}} \right) \left(\|F(\bm{w}^0)\|^2 + \frac{1}{1-\nu^2} \sum_{n=1}^{k-1} \varsigma_n^2 +  \sum_{n=1}^{k} \frac{1}{n^{3\beta}} \right),
\ee
where $\tilde{L} = 4\kappa_1^2( L^2 + C_{\eta}^2  )$.
It is easy to verify that there exists $c > 1$ such that \eqref{globalcov:assertion} holds when $k \le 4$. Suppose for some $K > 4$, and any $k< K$ \eqref{globalcov:assertion} holds, then in the $K$-th iteration, if $\tau_{K,i} = \kappa_1 K^{\beta}$, it follows from \eqref{eq:residual-increase} and \eqref{globalcov:assertion} that
\[
\begin{aligned}
\|F(\bm{w}^{K+1})\|^2 & \leq
c\Pi_{n=1}^{K + 1} (1 + \frac{ \tilde{L} }{n^{3\beta}}) \left(\|F(\bm{w}^0)\|^2 +  \frac{1}{1 -\nu^2 }\sum_{j=1}^{K-1} \varsigma_j^2 +  \sum_{n=1}^{K+1} \frac{1}{n^{3\beta}} \right).
\end{aligned}
\]
Hence \eqref{globalcov:assertion} holds for $\bm{w}^{K+1}$ if $\tau_{K,i} = \kappa_1 K^{\beta}$.
For the remaining case where \eqref{eq:decrease-1} holds, i.e., $\bm{w}^{k+1} = \tilde{\bm{w}}^{k,i}$, it follows from \eqref{eq:residual-decrease} that
\[
\begin{aligned}
& \|F(\bm{w}^{K+1})\| \le \nu \sqrt{c\,\Pi_{n=1}^{k} \left(1 + \frac{ \tilde{L} }{n^{3\beta}} \right) \left(\|F(\bm{w}^0)\|^2 + \frac{1}{1-\nu^2} \sum_{j=1}^{K-1} \varsigma_j^2 +  \sum_{n=1}^{K} \frac{1}{n^{3\beta}} \right)} + \varsigma_K.
\end{aligned}
\]
According to the AM-GM inequality, i.e., $(a+b)^2 \leq (1 + \rho)a^2 + (1 + 1/\rho)b^2$ with $\rho = \frac{1 - \nu^2}{\nu^2}$ for all $a, b \in \R$ and $\rho > 0$, we have
\[
\|F(\bm{w}^{K+1})\|^2 \leq
c\Pi_{n=1}^{K } (1 + \frac{ \tilde{L} }{n^{3\beta}}) \left(\|F(\bm{w}^0)\|^2 +  \frac{1}{1 -\nu^2 }\sum_{j=1}^{K-1} \varsigma_j^2 +  \sum_{n=1}^{K} \frac{1}{n^{3\beta}}  \right) + \frac{1}{1-\nu^2} \varsigma_K^2,
\]
Hence, \eqref{globalcov:assertion} holds. Since $\Pi_{k=1}^\infty \left(1 + \frac{ \tilde{L}^2 }{k^{3\beta} } \right) < \exp(\sum_{k=1}^\infty \frac{ \tilde{L}^2 }{k^{3\beta}} ) < \infty$ and $\sum_{j=1}^{\infty} \varsigma_j^2$ is finite, we can conclude that $\|F(\bm{w}^{k})\|$ is bounded. Let $M_1 > 0$ be the constant such that for all $k$, $\|F(\bm{w}^{k})\|^2 \leq M_1$.  If $\tau_{k,i} = \kappa_1 k^\beta$, it holds that
\[  \| F(\bm{w}^{k+1})\|^2 \leq \|F(\bm{w}^k)\|^2 + \frac{ M}{k^{3\beta}}, \]
where $M = 4L^2M_1\kappa_1^2  + 4\kappa_1^2C_{\eta}^2$.
\end{proof}
% \revise{Since we set $\bm{w}^{k+1} = \tilde{\bm{w}}^{k,i} = \mathcal{P}_{\theta^k,l^k,\rho^k,\sigma}(\bar{\bm{w}}^{k, i})$ if \eqref{eq:decrease-1} holds, i.e., the iterate that is executed by the correction step, the parameters $\theta,l,\rho$ do not need to be considered in the proof of Lemma \ref{lem:glo}}.
Next, we establish the global convergence of Algorithm \ref{alg:ssn}. This is demonstrated by showing that the residual norm converges to zero.

\begin{theorem} \label{thm:global-con}
Suppose that Assumption \ref{assum} holds. Let $\{\bm{w}^k\}$ be the sequence generated by Algorithm \ref{alg:ssn}. The residual $F(\bm{w}^k)$ converges to $0$, i.e., \be \label{eq:con-w} \lim_{k \rightarrow \infty} \; F(\bm{w}^k) = 0. \ee
\end{theorem}
\begin{proof}
 Let us consider the following three cases, i.e., (i) the total number of steps involving \eqref{eq:decrease-1} is finite, (ii) the total number of steps involving \eqref{eq:decrease-2} is finite, (iii) steps associated with both \eqref{eq:decrease-1} and \eqref{eq:decrease-2} are encountered infinitely many times. Specifically, if (i) happens, according to the update of $\bm{w}^{k+1}$ \eqref{eqn:updateX2}, we have
\[
\begin{aligned}
\| \bm{w}^{k+1} - \bm{w}_* \| &\le  (1 - \tau_{k,i}^{-1})\|\bm{w}^k - \bm{w}_*\| + \tau_{k,i}^{-1}\|T(\bm{w}^k) - \bm{w}_* \| + \tau_{k,i}^{-1}\|\bm{e}^k\| \\
&\le \|\bm{w}^k - \bm{w}_* \| + \tau_{k,i}^{-1}\|\bm{e}^k\|,
\end{aligned}
\]
where $\bm{e}^k = \tau_{k,i}(\bm{r}^k + \tilde{\bm{\eta}}^k)$. Since  $\tau_{k,i}^{-1}\|\bm{e}^k\| $ is summable,
it follows that $\|\bm{w}^{k+1}-\bm{w}_*\|^2$ is bounded. In this case, for every $\bm{w}_* \in \bm{W}_*$, we have
\begin{equation*}
\begin{aligned}
& \|\bm{w}^{k+1} -\bm{w}_* \|^2  = \|(1- \tau_{k,i}^{-1}) (\bm{w}^k - \bm{w}_*) + \tau_{k,i}^{-1}(T(\bm{w}^k) - \bm{w}_* + \bm{e}^k )  \|^2 \\
& = (1- \tau_{k,i}^{-1})\| \bm{w}^k -\bm{w}_* \|^2 - \tau_{k,i}^{-1}(1-\tau_{k,i}^{-1}) (\|F(\bm{w})\|^2 + 2\iprod{F(\bm{w})}{\bm{e}^k} + \|\bm{e}^k \|^2) \\
& +\tau_{k,i}^{-1} (\|T(\bm{w}^k) - \bm{w}_* \|^2 + 2\iprod{T(\bm{w}^k) - \bm{w}_*}{\bm{e}^k} + \|\bm{e}^k\|^2 ) \\
& \le \| \bm{w}^k -\bm{w}_* \|^2  - \tau_{k,i}^{-1}(1 - \tau_{k,i}^{-1}) \|F(\bm{w}^k) \|^2  + \underbrace{2 \tau_{k,i}^{-1} \iprod{\bm{w}^k - \bm{w}_*}{\bm{e}^k} + (\tau_{k,i}^{-2} ) \|\bm{e}^k\|^2 }_{:= \bm{\xi}^k} \\
& \le \| \bm{w}^k -\bm{w}_* \|^2 - \tau_{k,i}^{-1}(1 - \tau_{k,i}^{-1}) \|F(\bm{w}^k) \|^2 + 2\tau_{k,i}^{-1}\|\bm{w}^k-\bm{w}_*\| \|\bm{e}^k\| + \tau_{k,i}^{-2}\|\bm{e}^k\|^2,
\end{aligned}
\end{equation*}
where the second equality follows from $\|\tau_{k,i}^{-1}a + (1-\tau_{k,i}^{-1})b\|^2 =  \tau_{k,i}^{-1}\|a\|^2 + (1 - \tau_{k,i})\|b\|^2 + \tau_{k,i}(1 - \tau_{k,i})\|a + b\|^2$ for any $a,b \in \mathbb{R}^n$.
Since $\|\bm{w}^k - \bm{w}_*\|$ is bounded, it follows that $\|\bm{\xi}^k\|^2$ is summable and hence we have
\begin{equation} \label{eqn:complexity-first}
 \sum_{i=0}^{\infty}\tau_{k,i}^{-1} \|F(\bm{w}^i) \|^2 \le \|\bm{w}^0-\bm{w}_*\|^2 + \sum_{i=0}^{\infty} \| \bm{\xi}^i \|^2.
\end{equation}
 Given that $\tau_{k,i} = \kappa_1 k^{\beta}$ and $\beta \in (\frac{1}{3},1)$, there exists a constant $C_F$ such that $\|F(\bm{w}^k)\|^2 \le C_F{k^{-\frac{2}{3}}}$. Hence it follows that \eqref{eq:con-w} holds.
We next consider the case (ii) where there exists a $K_0 > 0$ such that $\|F(\bm{w}^{k+1})\| \leq \nu \max_{\max(1, k-\zeta +1)\leq j \leq k} \|F(\bm{w}^{j})\| + \varsigma_k $ for  $k \geq K_0$. It follows from \cite[Lemma 3.2]{armand2017globally} that \eqref{eq:con-w} holds. In terms of (iii), let $k_1, k_2, \ldots$ be the indices corresponding to the steps where \eqref{eq:decrease-1} hold.
According to Lemma \ref{lem:glo}, we have for any $k_n + 1 < j \leq k_{n+1}$,
    \be \label{eq:increment-res}\begin{aligned}
    \|F(\bm{w}^{j+1})\|^2 & \leq \|F(\bm{w}^{j})\|^2 + \frac{ M}{j^{3\beta}} \leq \|F(\bm{w}^{k_{n}+1})\|^2 + \sum_{k=k_n}^{j}  \frac{ M}{k^{3\beta}}.
    \end{aligned}
    \ee

    Let $a_n = \max_{\max(1, n- \zeta +1) \leq j \leq n} \|F(\bm{w}^{k_j + 1})\|$ and $\pi_n:= \sqrt{\sum_{k=k_n+1}^{k_{n+1}-1}  \frac{ M}{k^{3\beta}}}$.
     Due to $k_{n+ \zeta} - k_n \geq \zeta$, we have the recursion
    \be \label{eq:est1}  \| F(\bm{w}^{k_{n+1} + 1}) \| \leq \nu a_n + \nu \max_{ \max(1, n- \zeta +1) \leq j \leq n} (\pi_j + \varsigma_{k_j}). \ee
    Denote $c_n := a_{n\zeta}$. Inequality \eqref{eq:est1} gives
    \begin{equation} \label{eq:cn}
    c_{n+1} \leq \nu c_n + \nu \sum_{l=n \zeta}^{(n+1)\zeta - 1} \max_{ \max(1, l- \zeta +1) \leq j \leq l} (\pi_j + \varsigma_{k_j}).
    \end{equation}
    It follows from \cite[Lemma 2]{xu2015augmented} that there exist constants $C_1 > 0$ and $C_2 > 0$ such that
    \be \label{eq:an} \sum_{n=0}^{\infty}\|c_{n+1}\|^2 \leq C_1 \sum_{n=0}^{\infty} \left[ \sum_{l=n \zeta}^{(n+1)\zeta-1} \max_{ \max(1, l- \zeta +1) \leq j \leq l} (\pi_j  + \varsigma_{k_j}) \right]^2 + C_2. \ee
    Note that for sufficiently large $n \geq N$,
    \[ \begin{aligned}
        &\left[ \sum_{l=n \zeta}^{(n+1)\zeta-1} \max_{ \max(1, l- \zeta +1) \leq j \leq l}\pi_j + \varsigma_{k_j} \right]^2 \leq \left(\sum_{l = n\zeta}^{(n+1)\zeta -1 } \sum_{j=l-\zeta+1}^l \pi_j + \varsigma_{k_j} \right)^2 \\ &\leq \left( \zeta \sum_{l={(n-1)\zeta+1 }}^{(n+1) \zeta}  \left(\pi_l  + \varsigma_{k_l} \right)  \right)^2
         \leq 4\zeta^2 \sum_{k=k_{(n-1)\zeta +1}}^{k_{(n+1) \zeta}}  \left(\frac{M}{k^{3\beta}}  + \varsigma_k^2 \right).
    \end{aligned}
     \]
    This implies
    \[
    \sum_{n=N}^{\infty} \left[ \sum_{l=n \zeta}^{(n+1)\zeta-1} \max_{ \max(1, l- \zeta +1) \leq j \leq l} \pi_j + \varsigma_n \right]^2 \leq 8 \zeta^3 \sum_{n=k_{(N-1) \zeta}}^{\infty} \left(\frac{M}{n^{3\beta}} + \varsigma_n^2 \right) < \infty
    \]
    and $\lim_{n \rightarrow \infty} \|c_n\| = 0.$    Furthermore, by the definition of $c_n$ and the summability of $\{ \varsigma_k^2\}_{k=1}^{\infty}$, we have $\lim_{j \rightarrow \infty} \|F(\bm{w}^{k_j + 1})\| \rightarrow 0$. Using \eqref{eq:increment-res} and $\beta > \frac{1}{3}$,  \eqref{eq:con-w} holds.
\end{proof}

The above proof combines the explicit descent in \eqref{eq:decrease-1} with the implicit descent in \eqref{eq:decrease-2} derived from the inexact first-order step. Since the correction step $\mathcal{P}_{\theta^k, l^k, \rho^k, \sigma}(\bar{\bm{w}}^{k,i})$ is accepted only when \eqref{eq:decrease-1} holds, the above analysis of global convergence is independent of the projection.

\subsection{Local convergence}
The goal in this section is to establish the superlinear convergence of Algorithm \ref{alg:ssn}, which is proved in \citep{hu2025analysis} under the SC condition. However,  since the SC condition may not hold in general, we aim to extend the desired result with the correction step without assuming SC.
% This is probably the first work on the superlinear convergence of the semismooth Newton method without assuming nonsingularity and SC.
The proof outline is listed below.
First, we show the manifold identification of the iterates empowered by the correction step. Then, under the smoothness of $F$ on the manifold,  the local superlinear convergence for $\{\bm{w}^k\}$ holds under the local error bound condition. We introduce the notion of partial smoothness \citep{lewis2022partial} in the following.
\begin{definition}[$C^p$-partial smoothness]\label{def-psmooth}
Consider a proper closed function $\phi:\R^n\rightarrow \bar{\mathbb{R}}$ and a $C^p$ $(p \ge 2)$ embedded submanifold $\Mcal$ of $\R^n$. The function $\phi$ is said to be $C^p$-partly smooth at $x \in \Mcal$ for $v \in \partial \phi(x)$ relative to $\Mcal$ if
\begin{itemize}
    \item[(i)] Smoothness: $\phi$ restricted to $\mathcal{M}$ is $C^{p}$-smooth near $x$.
    \item[(ii)] Prox-regularity: $\phi$ is prox-regular at $x$ for $v$.

    \item[(iii)] Sharpness: $\mathrm{par}\, \partial_p \phi(x) = N_{\Mcal} (x)$, where $\partial_p \phi(x)$ denotes the set of proximal subgradients of $\phi$ at point $x$, $\mathrm{par}\, \Omega$ is the subspace parallel to $\Omega$ and $N_{\Mcal} (x)$ is the normal space of $\Mcal$ at $x$.
    \item[(iv)] Continuity: There exists a neighborhood $V$ of $v$ such that the set-valued mapping $V \cap \partial \phi$ is inner semicontinuous at $x$ relative to $\mathcal{M}.$
\end{itemize}
\end{definition}

% As shown in \citep{hu2025analysis}, the indicator functions of the nonnegative cone and the positive semidefinite cone are both partly smooth.

Since $h$ is an elementwise indicator function on a
polyhedral convex set, we can check that $h,\delta_{\mathcal{Q}},$ and $\delta_{\mathcal{K}}$ are partly
smooth \citep{lewis2002active}.
One usage of partial smoothness is connecting the relative interior condition in (iii) with SC to derive certain smoothness in nonsmooth optimization \citep{bareilles2023newton}.
In addition to the partial smoothness, the local error bound condition \citep{yue2019family} is a powerful tool for analyzing convergence in the absence of nonsingularity.
\begin{definition}
We say the local error bound condition holds if there exist $\gamma_l > 0$ and $\varepsilon > 0$ such that for all $\bm{w}$ with ${\rm dist}(\bm{w},\bm{W}_*)\le \varepsilon$, it holds that
\begin{equation}
    \label{eq:eb} \| F(\bm{w}) \| \geq \gamma_{l} {\rm dist}(\bm{w}, \bm{W}_*),
\end{equation}
    where $\bm{W}_*$ is the solution set of $F(\bm{w}) = 0$ and ${\rm dist}(\bm{w}, \bm{W}_*):=\argmin_{\bm{u} \in \bm{W}_*} \|\bm{w} - \bm{u}\|$.
\end{definition}

We next present some notations. For any solution $\bm{w}_* \in \bm{W}_*,$ let $\lambda^* \in \R^n$ be the eigenvalues of $ \sigma(\mathcal{A}^*(\bm{y}_*) + \bm{z}_* -\bm{c}) + \bm{x}_*$. We denote the threshold values with respect to $\bm{w}_*$ as
\begin{equation} \label{def:para}
\begin{aligned}
\bar{\theta} & := \min(\lambda^*_{\alpha_1 },|\lambda^*_{\alpha_2}| ), \\
\bar{l} & := \min_{i \in
I_{\mathcal{Q}}} \{ {\rm dist}(-(\bm{y}_*)_i,\text{bd}(\mathcal{N}_{\mathcal{Q}_i}((\bm{u}_*)_i)))  \},
\\
 \bar{\rho} &:= \min_{(i,j) \in I_h } \{ {\rm dist}(-(\bm{z}_*)_{i,j},\text{bd}(\partial h_{i,j}((\bm{x}_*)_{i,j}) ))   \},
\end{aligned}
\end{equation}
where $ \alpha_1 := \{ i | \lambda_i^* >0\}, \alpha_2 := \{ i | \lambda_i^* < 0\} ,
 I_{\mathcal{Q}} := \{i|  -(\bm{y}_*)_{i} \in {\rm ri}(\partial \mathcal{N}_{\mathcal{Q}_i}((\bm{u}_*)_i) )  \},  I_h := \{ (i,j)| -(\bm{z}_*)_{i,j} \in {\rm ri}(\partial h_{i,j}((\bm{x}_*)_{i,j})  \},$ and $ h_{i,j}$ denotes the function $h$ restricted to index $(i,j)$. According to \cite[Page 64, Exercise 2.11]{rockafellar2009variational}, the relative interior and interior coincide in the case of $\R$. Hence it follows from the definition of $\alpha_1,\alpha_2,I_{\mathcal{Q}},I_h$ that $\bar{\theta},\bar{l}$, and $\bar{\rho}$ are well defined.

Firstly, we show by the assumptions on $h, \mathcal{Q}$ and $\mathcal{K}$ that there exists a manifold defined with respect to each solution $\bm{w}_* \in \bm{W}_*$, and such manifold will be useful in the analysis of local smoothness of $F$.

 \begin{lemma} \label{lem:local-smooth}
 For given $\sigma$ and $\bm{w}_* \in \bm{W}_*$, we define the following sets:
\begin{equation}
\begin{aligned}
\overline{\mathcal{M}}_{\bm{w}_*,I_h} &:= \{ \bm{w}| (\bm{q} - \sigma \bm{z})_{i,j} = (\bm{q}_* - \sigma \bm{z}_*)_{i,j}  ~~ \text{for} ~~ (i,j) \notin I_h   \}, \\
\overline{\mathcal{M}}_{\bm{w}_*, I_{\mathcal{Q}}} &:= \{ \bm{w}|  (\bm{u} -  \sigma \bm{y})_i = (\bm{u}_* -  \sigma \bm{y}_*)_i ~~\text{for}~~ i \notin I_\mathcal{Q} \}, \\
\overline{\mathcal{M}}_{\bm{w}_*, r_{\mathcal{K}}} &:= \{ \bm{w}| \mathrm{rank}\left(\bm{x} -  \sigma (\mathcal{A}^*(\bm{y}) -\bm{c} + \bm{z})\right) = |\alpha_1| + |\alpha_2|  \}. \\
\end{aligned}
\end{equation}
Then $\overline{\mathcal{M}}_{\bm{w}_*} := \overline{\mathcal{M}}_{\bm{w}_*,I_h} \cap \overline{\mathcal{M}}_{\bm{w}_*,I_\mathcal{Q}} \cap \overline{\mathcal{M}}_{\bm{w}_*,r_{\mathcal{K}}}$ is a manifold.
\end{lemma}

\begin{proof}
 Without loss of generality, for $\bm{w} \in \overline{\mathcal{M}}_{\bm{w}_*,r_{\mathcal{K}}}$, we assume that the $r \times r$ matrix in the upper left corner of $\bm{x} -  \sigma (\mathcal{A}^*(\bm{y}) -\bm{c} + \bm{z})$ is invertible. For any $\bm{r} = \begin{bmatrix}
    \bm{r}_{11} & \bm{r}_{12} \\
    \bm{r}_{21} & \bm{r}_{22}
\end{bmatrix} \in \mathbb{R}^{n \times n}$, where $\bm{r}_{11} \in \mathbb{R}^{r \times r}$ be invertible and $\bm{r}_{22} \in \R^{(n-r) \times (n-r)},$ define $f : \bm{x} \in \mathbb{R}^{ n\times n} \rightarrow  \bm{x}_{22} - \bm{x}_{21}\bm{x}_{11}^{-1}\bm{x}_{12} \in \mathbb{R}^{(n-r) \times (n-r)},  $ and $g : \bm{w}
 \in \mathbb{R}^{4n \times n + m} \rightarrow  \bm{x} -  \sigma (\mathcal{A}^*(\bm{y}) -\bm{c} + \bm{z}) \in \mathbb{R}^{n \times n}$. Then, for every $\bm{w} \in \overline{\mathcal{M}}_{\bm{w}_*,r_{\mathcal{K}}}$, it follows that $\psi(\bm{w}) =0$, where $\psi := f \circ g$. If the $r \times r$ matrix in the upper left corner of $\bm{x} -  \sigma (\mathcal{A}^*(\bm{y}) -\bm{c} + \bm{z})$ is not invertible, we can
construct another local defining function using the same procedure \cite[Section 7.5]{boumal2023introduction}. For $g_{\mathcal{Q}}(\bm{w}) := ((\bm{u} -  \sigma \bm{y}) - (\bm{u}_* -  \sigma \bm{y}_*))_{I_{\mathcal{Q}}^c }$, $g_h(\bm{w}) := ((\bm{q} -  \sigma \bm{z}) - (\bm{q}_* -  \sigma \bm{z}_*))_{I_{h}^c }$, we define
\begin{equation}
\overline{\psi}:= \begin{bmatrix}
    \psi \\
      g_{\mathcal{Q}}  \\
    g_{h}
\end{bmatrix} : \mathbb{R}^{4 n \times n + m} \rightarrow \mathbb{R}^{(n-r) \times (n-r) + |I_{\mathcal{Q}}^c| + |I_h^c| }.
\end{equation}
We now prove that the Jacobian matrix of $\overline{\psi}$ is full rank.   According to \cite[Section 7.5]{boumal2023introduction}, for given $\bm{r} \in \mathbb{R}^{n \times n}, \bm{w} \in \mathbb{R}^{4 n \times n +m}$ and any $\bm{g} \in \mathbb{R}^{n \times n}, \bm{o}_1 \in \mathbb{R}^m, \bm{o}_2,\bm{o}_3,\bm{o}_4,\bm{o}_5 \in \mathbb{R}^{n \times n}, \bm{o} =[\bm{o}_1,\bm{o}_2,\bm{o}_3,\bm{o}_4,\bm{o}_5] \in \mathbb{R}^{4 n\times n + m}$, it follows that
\begin{equation} \label{def:Jfg}
\begin{aligned}
Df(\bm{r}) [\bm{g}] &= \bm{g}_{22} - \bm{g}_{21}\bm{r}_{11}^{-1}\bm{r}_{12} + \bm{r}_{21}\bm{r}_{11}^{-1}\bm{g}_{11}\bm{r}_{11}^{-1}\bm{r}_{21} - \bm{r}_{21}\bm{r}_{11}^{-1}\bm{g}_{12}, \\
D g(\bm{w}) [\bm{o}] &= -\sigma\mathcal{A}^*\bm{o}_{1}-\sigma \bm{o}_{2}+\bm{o}_3,\\
D g_{\mathcal{Q}}(\bm{w}) [\bm{o}] & = (\bm{o}_4 - \sigma \bm{o}_1)_{I_{\mathcal{Q}}^c},\\
D g_{h}(\bm{w}) [\bm{o}] & = (\bm{o}_5 - \sigma \bm{o}_2)_{I_{h}^c}.
\end{aligned}
\end{equation}
% For any $\bar{\bm{a}} \in \mathbb{R}^{(n-r) \times (n-r)}$, let $\bm{o} = (0,0,\bm{a},0,0),$ where $\bm{a} = \begin{bmatrix}
%     0 & 0\\
%     0 & \bar{\bm{a}}
% \end{bmatrix} \in \mathbb{R}^{n \times n}$.
% According to \eqref{def:Jfg},  it follows that $D\psi (\bm{w})[\bm{o}] = \bar{\bm{a}}$.
Consequently, for any $( \bar{\bm{a}} , \bar{\bm{b}},\bar{\bm{c}} ) \in \mathbb{R}^{(n-r) \times (n-r) + |I_{\mathcal{Q}}^c| + |I_h^c|}$ with $\bar{\bm{a}} \in \mathbb{R}^{(n-r) \times (n-r)},\bar{\bm{b}} \in \R^{|I_{\mathcal{Q}}^c|}, \bar{\bm{c}} \in \R^{|I_h^c|}$,
set $\bm{o} = (0,0,\bm{a},\bm{b},\bm{c}),$ where $\bm{a} = \begin{bmatrix}
    0 & 0\\
    0 & \bar{\bm{a}}
\end{bmatrix} \in \mathbb{R}^{n \times n}, (\bm{b})_{I_{\mathcal{Q}}^c} = \bar{\bm{b}} $, and $  (\bm{c})_{I_h^c } = \bar{\bm{c}}$. It follows from the definition of $\bar{\psi}$ that $D \bar{\psi} [\bm{o}] = (\bar{\bm{a}} , \bar{\bm{b}},\bar{\bm{c}} ),$ which indicates that $D \bar{\psi}$ is full rank.
 % We next prove that the Jacobian matrix of $\psi, D \psi : \mathbb{R}^{4n\times n+m} \rightarrow \mathbb{R}^{(n-r) \times (n-r)}$ is full rank.
% which indicates that $\psi (\bm{w}) =0$ defines a manifold locally.
 Consequently, $\bar{\psi}(\bm{w}) =0$ defines a manifold locally and we complete the proof.
 % Hence for every $\bm{w}_*$, $\overline{\mathcal{M}}_{\bm{w}_*}$ is a smooth manifold. Then, by the definition of $F$ and denoting $V = V_1 \cap V_2 \cap V_3$, it follows that $F$ is smooth on $\overline{\mathcal{M}}_{\bm{w}_*} \cap V.$
\end{proof}

We then make the following assumptions.
\begin{assum} \label{assum:local} Let $\{\bm{w}^k\}$ be the iterate sequence generated by Algorithm \ref{alg:ssn}.
\begin{enumerate}[label=\textup{\textrm{(A\arabic*)}},topsep=0pt,itemsep=0ex,partopsep=0ex]
%     \item  \label{assumption:A1} The functions $h$, $\delta_{\mathcal{Q}}$, and $\delta_{\mathcal{K}}$ are partly smooth with the associated manifold being $\mathcal{M}_{\bm{w}_*,h}, \mathcal{M}_{\bm{w}_*,\mathcal{Q}},
% $ and $ \mathcal{M}_{\bm{w}_*,\mathcal{K}}$.
    \item  \label{assumption:A2} The local error bound condition \eqref{eq:eb} holds.
    \item \label{assumption:A3} The sequence $ \{ \bm{\eta}^k \}$ satisfies $\|\bm{\eta}^k\|\le c_{\eta} \|F(\bm{w}^k)\|^q $, where $q \in (1,2]$.
    \item   \label{assumption:A4} For a solution $\bm{w}_* \in \bm{W}_*$, the correction parameters satisfy $\theta^k \in (0,\bar{\theta}),l^k \in (0,\bar{l})$, and $ \rho^k \in (0,\bar{\rho})$, where $\bar{\theta},\bar{l}$, and $\bar{\rho}$ are defined in \eqref{def:para} with respect to $\bm{w}_*$ .
\end{enumerate}
\end{assum}

\begin{remark}
The Assumptions  \ref{assumption:A2}, and \ref{assumption:A3} are commonly used assumptions in local convergence analysis \citep{hu2025analysis,bareilles2023newton,yue2019family}.      As shown in \citet[Corollary 1]{ding2023strict}, a variant of the local error bound condition holds under strong duality and dual strict complementarity.   Furthermore, Assumption~\ref{assumption:A2} is much weaker than the BD-regularity condition adopted in \citep{xiao2018generalized} and \citep{li2018semismooth}.  We also conduct some numerical verifications of \ref{assumption:A2} in Section \ref{num-ver}.
Assumption~\ref{assumption:A3} specifies the solution accuracy of the Newton equation, which is readily satisfied in numerical experiments. Assumption \ref{assumption:A4} is assumed with respect to a solution $\bm{w}_* \in \bm{W}_*$. It is used to ensure the manifold identification of the iterations and desired smoothness of $F$.
We note that Assumption \ref{assumption:A4} is designed for the general problem \eqref{prob} and can
be simplified for concrete examples. Particularly,
 for SDP, since there is no $h$ and $\mathcal{Q}$, the Assumption \ref{assumption:A4} reduces to $\theta \in (0,\bar{\theta}) $. A similar condition has also been assumed in \cite[Theorem 2]{feng2024quadratically}. However, our approach to defining these parameters and the correction step differs substantially. Specifically, we identify a smooth manifold on which the iterates eventually lie and establish superlinear convergence without requiring any nonsingular element in the generalized Jacobian of
$F$, unlike \citep{feng2024quadratically}. This indicates the applicability of our superlinear convergence to nonisolated minima, while their result does not cover it.
 % \revise{However, our local convergence analysis relies on the local error bound conditions instead of the weak strict Robinson constraint qualification (W-SRCQ) or weak second order sufficient condition (W-SOC) considered in \citep{feng2024quadratically}. It is used to prove that there exists one Jacboian of $F$ is nonsingular, see \cite[proposition 8]{feng2024quadratically} for details. In comparison, we do not require any Jacobian of $F$ to be nonsingular.}
 Furthermore, if SC holds, it follows that $F$ is locally smooth on the whole space \citep[Lemma 2]{hu2025analysis}. Hence no correction step is needed.
\end{remark}
% For a solution $\bm{w}* \in \bm{W}$, observe that $\bm{q}^* = \bm{x}^*$ at optimality. Let $\mathcal{M}{\bm{w}*,h}$, $\mathcal{M}{\bm{w}*,\mathcal{Q}}$, and $\mathcal{M}{\bm{w}*,\mathcal{K}}$ denote the active smooth manifolds associated with the partly smooth functions $h$, $\delta{\mathcal{Q}}$, and $\delta_{\mathcal{K}}$, respectively. Then, by the identity $\bm{q}* = \bm{x}$, we have $\bm{q}_ \in \mathcal{M}{\bm{w},h}$, $\bm{u}_ \in \mathcal{M}{\bm{w},\mathcal{Q}}$, and $\bm{x}_ \in \mathcal{M}{\bm{w}*,\mathcal{K}}$.

For a solution $\bm{w}_* \in \bm{W}_*$, let $ \mathcal{M}_{\bm{w}_*,h}, \mathcal{M}_{\bm{w}_*,\mathcal{Q}},
$ and $ \mathcal{M}_{\bm{w}_*,\mathcal{K}} $ be the associated smooth manifolds of the partly smooth functions  $h, \delta_{\mathcal{Q}}, \delta_{\mathcal{K}} $. It follows from the definition of partial smoothness and $\bm{q}_* = \bm{x}_*$ that  $\bm{q}_* \in \mathcal{M}_{\bm{w}_*,h} , \bm{u}_* \in \mathcal{M}_{\bm{w}_*,\mathcal{Q}}, \bm{x}_* \in \mathcal{M}_{\bm{w}_*,\mathcal{K}}$.   Based on the construction of correction step, we have the following local smoothness result.

\begin{lemma}
% Suppose that Assumption \ref{assumption:A1} holds, then
For a solution $\bm{w}_* \in \bm{W}_*$, there exists a neighborhood $V$ of $\bm{w}_*$ such that
$F(\bm{w})$ is locally smooth on $V \cap \overline{\mathcal{M}}_{\bm{w}_*} $.
\end{lemma}

\begin{proof}
Since $h,\delta_{\mathcal{Q}},$ and $\delta_{\mathcal{K}}$ are proper closed and convex functions, according to \cite[Theorem 6.3,\,6.4,\,6.42]{beck2017first},
 ${\rm prox}_{\sigma h}(\bm{x} - \sigma \bm{z}), \Pi_{\mathcal{Q}}(\bm{u} -  \sigma \bm{y}), \Pi_{\mathcal{K}}(\bm{x} -  \sigma (\mathcal{A}^*(\bm{y}) -\bm{c} + \bm{z}))$ are monotone, single-valued,  and Lipschitz.
 Define $h_{I_h}$ be the function when restricted on $I_h,$ then we can conclude that there exists a neighborhood $V_1$ of $\bm{w}_*$ such that  $\text{prox}_{\sigma h_{I_h}}(\bm{q} -\sigma \bm{z}) \in \mathcal{M}_{\bm{w}_*,h} $ for $\bm{w} \in V_1 \cap \overline{\mathcal{M}}_{\bm{w}_*,I_h}$ by \cite[Proposition 10.12]{drusvyatskiy2014optimality}. According to the definition of $\overline{\mathcal{M}}_{\bm{w}_*,I_h}$,  $\prox_{\sigma h}(\bm{q}-\sigma\bm{z}) \in \mathcal{M}_{\bm{w}_*,h} $ for all $\bm{w} \in V_1 \cap \overline{\mathcal{M}}_{\bm{w}_*,I_h}.$ The result that there exists $V_2$ such that $\Pi_{\mathcal{Q}}(\bm{u} -\sigma \bm{y}) \in \mathcal{M}_{\bm{w}_*,\mathcal{Q}} $ for $ \bm{w} \in \overline{\mathcal{M}}_{\bm{w}_*,I_{\mathcal{Q}}} \cap V_2$ can be proved analogously. Since $\mathrm{rank}\left(\bm{x}_* -  \sigma (\mathcal{A}^*(\bm{y}_*) -\bm{c} + \bm{z}_*)\right) = |\alpha_1| + |\alpha_2|$ and ${\rm rank}(\bm{x}_*) = |\alpha_1|,$ there exists a neighborhood $V_3$ of $\bm{w}_*$ such that the number of its positive eigenvalues is $|\alpha_1|$ and the number of negative eigenvalues is $|\alpha_2|$ when $\bm{w}$ is  restricted on $V_3 \cap \overline{\mathcal{M}}_{\bm{w}_*,r_{\mathcal{K}}}$. As a consequence, we have  \[
 \Pi_{\mathcal{K}}(\bm{x} -  \sigma (\mathcal{A}^*(\bm{y}) -\bm{c} + \bm{z})) \in \mathcal{M}_{\bm{w}_*,\mathcal{K}}~~ \text{for all}~~ \bm{w} \in V_3 \cap \overline{\mathcal{M}}_{\bm{w}_*,r_{\mathcal{K}}}. \]

 Combining the partial smoothness of $h,\delta_{\mathcal{Q}}$, and $\delta_{\mathcal{K}}$, and
    according to a similar scheme in \cite[Lemma 2]{hu2025analysis}, \cite[Lemma 2.1]{lewis2008alternating}, we have that ${\rm prox}_{\sigma h_1}(\bm{w}) := {\rm prox}_{\sigma h}(\bm{q} - \sigma\bm{z}), {\rm prox}_{\sigma h_2}(\bm{w}) :=\Pi_{\mathcal{Q}}(\bm{u} -  \sigma\bm{y})$, and $ {\rm prox}_{\sigma h_3}(\bm{w}) :=\Pi_{\mathcal{K}}(\bm{x} -  \sigma (\mathcal{A}^*(\bm{y})-\bm{c}+\bm{z}) )$ are locally smooth around $\bm{w}_*$ on $\overline{\mathcal{M}}_{\bm{w}_*,I_h}$, $\overline{\mathcal{M}}_{\bm{w}_*,I_{\mathcal{Q}}}$, and $  \overline{\mathcal{M}}_{\bm{w}_*,r_{\mathcal{K}}}$, respectively. According to the definition of $F$ and let $V = V_1 \cap V_2 \cap V_3$, the proof is completed.
\end{proof}

 The above result significantly differs from the existing manifold identification results in \citep{bareilles2023newton,deng2023augmented,liang2017activity} where SC is necessary. As noted in \cite[Remark 3.5]{liang2017activity}, relaxing SC is a challenging problem.  We next prove that $\{\bm{w}^k\}$ generated by Algorithm \ref{alg:ssn} stays on the smooth manifold through the correction step $\mathcal{P}_{\theta,l,\rho,\sigma}$. Furthermore, under the local smoothness condition, \eqref{eq:decrease-1} happens all the time and $\tau_{k,i}$ does not need to satisfy \eqref{eq:decrease-2}. Hence, the transition to the local superlinear convergence for $\{ \bm{w}^k \}$ holds.

  % According to Theorem \ref{thm:global-con} and the local error bound condition, we have that $\text{dist}(\bm{w}^k,\bm{W}_*) \rightarrow 0,$ namely, for any $\epsilon$, there exist some solution $\hat{\bm{w}}_* \in \bm{W}_*$ such that $\| \bm{w}^k - \hat{\bm{w}}_* \| < \epsilon.$ Define $\Pi_{\bm{W}_*}(\bm{w}) = \argmin_{\bm{w}_* \in \bm{W}_* }\|\bm{w}_* - \bm{w} \| $.

\begin{theorem} \label{thm:superlinear convergence}
Let Assumptions \ref{assum},  \ref{assumption:A2}, and \ref{assumption:A3} hold.
For sequence $\{\bm{w}^k\}$ generated by Algorithm \ref{alg:ssn}, if there exist a solution $\hat{\bm{w}}_* \in \Pi_{\bm{W}_*}( \bm{w}^K )$ and a constant $\delta_{\hat{\bm{w}}_*}$  for some $\bm{w}^K \in B(\hat{\bm{w}}_*,\delta_{\hat{\bm{w}}_*} ) \cap \overline{\mathcal{M}}_{\hat{\bm{w}}_*} $, \ref{assumption:A4} holds for $\hat{\bm{w}}_*$, the following conclusions hold.
\begin{enumerate}[label=\textup{\textrm{(C\arabic*)}},topsep=0pt,itemsep=0ex,partopsep=0ex]
    \item \label{Conclusion:C2}   Condition \eqref{eq:decrease-1} always holds with $i =0$ for $k > K$ and $\bm{w}^{k} $ lies in the manifold $ \overline{\mathcal{M}}_{ \hat{\bm{w}}_* } $.
    \item  \label{Conclusion:C3} The generated sequence $\{\bm{w}^k\}$ converges to some solution $\bm{w}_* \in \overline{\mathcal{M}}_{\hat{\bm{w}}_*}$  superlinearly.
\end{enumerate}
\end{theorem}

\begin{proof}
% \ref{Conclusion:C2}:
% From the definition, $\overline{\mathcal{M}}_{\bm{w}_*,I_h}$ and $\overline{\mathcal{M}}_{\bm{w}_*,I_\mathcal{Q}}$ are manifolds since they are elementwise affine sets.
  % We next prove that  $\overline{\mathcal{M}}_{\bm{w}_*,r_{\mathcal{K}}} = \{ \bm{w}| \mathrm{rank}\left(\bm{x} -  \sigma (\mathcal{A}^*(\bm{y}) -\bm{c} + \bm{z})\right) = |\alpha| + |\alpha_2|  \}$ is indeed a manifold.
\ref{Conclusion:C2}: By  Lemma \ref{lem:local-smooth}, we conclude that $F$ is locally smooth on $\overline{\mathcal{M}}_{\hat{\bm{w}}_*}$. Hence the local quadratic upper bound of $F$, i.e., condition \cite[Assumption 3(A1)]{hu2025analysis},  holds. For $i=0$ (ignoring the subscript $i$), let $\bar{\bm{w}}^K$ be defined in \eqref{eq:ssn-step}, $\bar{\lambda}^K $ be the  eigenvalue of $\bar{\bm{x}}^K+\sigma(\mathcal{A}^*( \bar{\bm{y}}^K) + \bar{\bm{z}}^K-\bm{c})$, and $ \hat{\lambda}^* $ be the eigenvalue of $ \hat{\bm{x}}_* + \sigma (\mathcal{A}^*(\hat{\bm{y}}_*) + \hat{\bm{z}}_* - \bm{c} ) $.  It follows from \cite[Lemma 5]{hu2025analysis} that there exists $\bar{c}$ such that $\|\bm{d}^K\| \le \bar{c} (\text{dist}(\bm{w}^K,\bm{W}_*)  + \|\bm{\eta}^K\| ) $.
Hence there exists $\delta_{ \hat{\bm{w}}_*, 1} $ such that for $ \bm{w}^K \in B(\hat{\bm{w}}_*,\delta_{ \hat{\bm{w}}_*, 1})$, we have  $\| \bar{\lambda}^K - \hat{\lambda}^* \|_{\infty} < \frac{\bar{\theta}}{2}.$ Consequently, $|\bar{\lambda}_{\alpha_1 \cup \alpha_2 }^K| > \frac{\bar{\theta}}{2} $  and $| \bar{\theta}_{ (\alpha_1 \cup \alpha_2)^c  }^K| < \frac{\bar{\theta}}{2}$. Then according to the definition of $\mathcal{P}_{\theta,l,\rho,\sigma}, \tilde{\bm{w}}^K$ and \ref{assumption:A4} in Assumption \ref{assum:local}, ${\rm rank}(\sigma(\mathcal{A}^*({ \tilde{\bm{y}}}^K)-\bm{c})+\tilde{\bm{x}}^K)  = |\alpha_1|+|\alpha_2|$.  Analogously, for the index where SC does not hold, it holds that $(\hat{\bm{u}}_*)_i \in \text{bd}(\mathcal{Q}_i) $ and   $\hat{\bm{y}}_{*,i} \in -\text{bd}(\mathcal{N}_{\mathcal{Q}_i} (\hat{\bm{u}}_{*,i}) )$.
 As a consequence, it holds that there exist $\delta_{ \hat{\bm{w}}_*, 2},\delta_{ \hat{\bm{w}}_*, 3}$ such that $( \tilde{\bm{u}}^K - \sigma \tilde{\bm{y}}^K)_{i} =  ( \hat{\bm{u}}_* - \sigma \hat{\bm{y}}_*)_{i} $, $i \notin I_{\mathcal{Q}}$ and $({ \tilde{\bm{q}}}^K - \sigma \tilde{\bm{z}}^K)_{i,j} = ( \hat{\bm{q}}_* - \sigma \hat{\bm{z}}_* )_{i,j}$, $ (i,j) \notin I_h$ for $ \bm{w}^K \in B(\hat{\bm{w}}_*,\delta_{ \hat{\bm{w}}_*, 2}) $ and $ \bm{w}^K \in B(\hat{\bm{w}}_*,\delta_{ \hat{\bm{w}}_*, 3}) $,   respectively.  Hence $\mathcal{P}_{\theta^K,l^K,\rho^K,\sigma}(\bm{w}^K + \bm{d}^K) \in  \overline{\mathcal{M}}_{\hat{\bm{w}}_*} $ holds.

Since \ref{assumption:A2} and \ref{assumption:A3} in Assumption \ref{assum:local} hold and $\tilde{\bm{w}}^K = \mathcal{P}_{\theta^K,l^K,\rho^K,\sigma}(\bm{w}^K + \bm{d}^K) \in  \overline{\mathcal{M}}_{\hat{\bm{w}}_*}$,  it follows from \cite[Lemma 6]{hu2025analysis} that
% ${\rm dist}(\mathcal{P}_{\theta^K,l^K,\rho^K,\sigma}(\bm{w}^K + \bm{d}^K),\bm{W}_*) \le c_2 {\rm dist}(\bm{w}^K,\bm{W}_*)^q $. Futhermore,
there exists $\delta_{ \hat{\bm{w}}_*, 4} $ such that for $ \bm{w}^K \in B(\hat{\bm{w}}_*,\delta_{ \hat{\bm{w}}_*, 4}) $, we have:
\begin{equation} \label{eqn:superlinear}
{\rm dist}(\mathcal{P}_{\theta^K,l^K,\rho^K,\sigma}(\bm{w}^K + \bm{d}^K),\bm{W}_*) \le \gamma_l^{-1} \|F(\mathcal{P}_{\theta^K,l^K,\rho^K,\sigma}(\bm{w}^K + \bm{d}^K) ) \| \le c_2 {\rm dist}(\bm{w}^K,\bm{W}_*)^q,
\end{equation}
where $c_2 > 0$ is the constant corresponding to $L,\gamma_l$ and $1<q<2$. According to the local error bound condition, the following inequality holds:
\begin{equation} \label{eqn:superlinear:descent}
\|F(\mathcal{P}_{\theta^K,l^K,\rho^K,\sigma}(\bm{w}^K + \bm{d}^K)) \| \le \tilde{c}_2 \gamma_{l} {\rm dist}(\bm{w}^K,\bm{W}_* )^q \le \tilde{c}_2 \gamma_l^{1-q} \| F(\bm{w}^K)\|^q.
\end{equation}
According to Theorem \ref{thm:global-con}, we have $\|F(\bm{w}^k)\| \rightarrow 0$. Consequently,
there exists $\delta_{ \hat{\bm{w}}_*, 5} $ such that for $ \bm{w}^K \in B(\hat{\bm{w}}_*,\delta_{ \hat{\bm{w}}_*, 5}) $, $\|F(\bm{w}^K)\|^{q-1} < \frac{1}{\tilde{c}_2 \gamma_{l}^{1-q} \nu}$. It follows that
\[
\|F(\mathcal{P}_{\theta^K,l^K,\rho^K,\sigma}(\bm{w}^K + \bm{d}^K)) \| < \nu \|F(\bm{w}^K)\|
\]
and hence \eqref{eq:decrease-1} holds.
Furthermore, it follows from the proof of \cite[Theorem 6]{hu2025analysis} that there exists $\delta_{ \hat{\bm{w}}_*, 6}$, $ \bm{w}^{K+1} \in B(\hat{\bm{w}}_*,\delta_{\hat{\bm{w}}_{*,6}} )$ if $\bm{w}^K \in B(\hat{\bm{w}}_*,\delta_{\hat{\bm{w}}_{*,6}} )$. Consequently,
let $\delta_{ \hat{\bm{w}}_*} = \min \{\delta_{ \hat{\bm{w}}_*, 1},\delta_{ \hat{\bm{w}}_*, 2},\delta_{ \hat{\bm{w}}_*, 3},\delta_{ \hat{\bm{w}}_*, 4},\delta_{ \hat{\bm{w}}_*, 5}, \delta_{ \hat{\bm{w}}_*, 6} \} $ and it follows that \ref{Conclusion:C2} holds by induction.

 \ref{Conclusion:C3}: By Lemma \ref{lem:local-smooth}, $F$ is smooth locally at $\hat{\bm{w}}_* \in \bm{W}_*$ on $\overline{\mathcal{M}}_{\hat{\bm{w}}_*}$. According to the proof in \ref{Conclusion:C2}, condition \cite[Assumption 3(A1)]{hu2025analysis} holds for $\bm{w}^k,\bm{w}^{k+1} ,k > K$. It then follows from \eqref{eqn:superlinear} and \cite[Theorem 6]{hu2025analysis} that $\{\bm{w}^k\}$ is a Cauchy sequence and hence converges to some $\bm{w}_{*} \in \overline{\mathcal{M}}_{\hat{\bm{w}}_*}$ superlinearly, which completes the proof of \ref{Conclusion:C3}.
\end{proof}

\begin{remark}
In Theorem \ref{thm:superlinear convergence},
\ref{Conclusion:C2} states that no line search is needed for $k > K$ and \ref{Conclusion:C3} presents the superlinear convergence results. The condition $\bm{w}^k \in \overline{\mathcal{M}}_{\hat{\bm{w}}_*}$ is used to guarantee the smoothness of $F$ and is also assumed in \cite[Lemma 6]{hu2025analysis}.
% It is satisfied when $\bm{w}^k$ is close to some $\hat{\bm{w}}_*$ and the correction step is executed.
Furthermore, we prove that $\bm{w}^k \in \overline{\mathcal{M}}_{\hat{\bm{w}}_*}$ for every $k > K$.  If assumptions in Theorem \ref{thm:superlinear convergence} and the additional SC hold, then $\theta,l,\rho $ are all equal to 0, which indicates the correction map $\mathcal{P}_{\theta,l,\rho,\sigma}$ reduces to the identity operator and $F$ is smooth locally on the whole space. Consequently, Theorem \ref{thm:superlinear convergence} is a theoretical extension compared with the local convergence results in \citep{hu2025analysis,deng2023augmented}, which rely on SC condition.
\end{remark}

% We next compare Assumption \ref{assumption:A4} with SC.
% To clarify the difference between them,

To gain a more concrete understanding of how restrictive SC is,
we next present the SC of problem \eqref{prob}.
\begin{proposition} \label{pro:SC}
Denote $\bm{w}_*$ as a solution. Under Assumption \ref{assum} and by
\cite[Theorem 6.2]{mordukhovich2017geometric}, SC holds at $\bm{w}_*$ if and only if
\begin{equation} \label{eqn:SC-all}
\begin{aligned}
& \bm{z}_* \in {\rm ri}(-\partial h(\bm{x}_*)),   \bm{y}_* \in {\rm ri}(-\mathcal{N}_{\mathcal{Q}}(\mathcal{A}(\bm{x}_*))),  \bm{c} - \bm{z}_* - \mathcal{A}^*\bm{y}_*  \in {\rm ri}(-\mathcal{N}_{\mathcal{K}}(\bm{x}_*)),
\end{aligned}
\end{equation}
where $\mathcal{N}_{\mathcal{C}}(\bm{x}):=\{\bm{s}|\iprod{\bm{s}}{\bm{p}} \le \iprod{\bm{s}}{\bm{x}}, \forall \bm{p} \in \mathcal{C} \}$ is the normal cone of the convex set $\mathcal{C}$.
Furthermore, for SDP, SC reduces to
\begin{equation} \label{SC:SDP}
  {\rm rank}(\bm{x}_*) + {\rm rank}(\bm{s}_*) =n, \iprod{\bm{x}_*}{\bm{s}_*} = 0.
\end{equation}
For SDP+, SC simplifies to
\begin{equation} \label{SC:SDP+}
{\rm rank}(\bm{x}_*) + {\rm rank}(\bm{s}_*) =n, \iprod{\bm{x}_*}{\bm{s}_*} =0, \iprod{\bm{z}_*}{\bm{x}_*} =0, \bm{x}_* + \bm{z}_* >0.
\end{equation}
\end{proposition}

\begin{proof}
According to the definition of normal cone, we have $\mathcal{N}_{\mathcal{K}}(\bm{x}_*) =
\{\bm{s}|\iprod{\bm{s}}{\bm{p}} \le \iprod{\bm{s}}{\bm{x}_*}, \forall \bm{p} \in \mathcal{K} \}$. Through taking $\bm{p} = \frac{1}{2} \bm{x}_*$ and $2 \bm{x}_*$, it can be verified that $\iprod{\bm{s}}{\bm{x}_*} =0 $ for all $\bm{s} \in \mathcal{N}_{\mathcal{K}}(\bm{x}_*).$ As a consequence, $\mathcal{N}_{\mathcal{K}}(\bm{x}_*) = \{\bm{s}| \iprod{\bm{s}}{\bm{p}} \le 0, \iprod{\bm{x}_*}{\bm{s}} = 0 , \forall \bm{p} \in \mathcal{K}\}= - \mathcal{K}^* \cap \{\bm{x}_*\}^{\bot},$ where $\mathcal{K}^*$ denotes the dual cone of $\mathcal{K}.$  Let $\bm{x}_*  = \sum_{i=1}^r \lambda_i V_iV_i^*$ be the eigenvalue decomposition,  where $r$ is the rank of $\bm{x}_*$, $\lambda_1,\cdots,\lambda_r$ correspond to the eigenvalues and $ [V_1,\cdots,V_r] $ are the corresponding eigenvectors. Denote $[V_{r+1} ,\cdots,V_n]$ as the orthogonal basis of the orthogonal complement subspace of the space generated by $[V_1,\cdots,V_r]$.
We claim that
\begin{equation} \label{eqn:K}
-\mathcal{K}^* \cap \{\bm{x}_*\}^{\bot} = \mathcal{S} := \{\bm{s}| \bm{s} = \sum_{i=1}^{n-r} \alpha_i V_{i+r}V_{i+r}^*,\alpha \le 0 \}.
\end{equation}
We first prove that $\mathcal{S} \subseteq -\mathcal{K}^* \cap \{\bm{x}_*\}^{\bot}$. It is obvious that $\iprod{\bm{s}}{\bm{x}} =0$ since $V_i^*V_j =0$ for $i \neq j.$ According to the definition of $\mathcal{S}$, for every $-\bm{s} \in \mathcal{S}$, we have $ \bm{s} \in \mathcal{K}^*$.  Then it follows that  $\mathcal{S} \subseteq \mathcal{N}_{\mathcal{K}}(\bm{x}_*).$ For every $\bm{s} \in -\mathcal{K}^* \cap \{\bm{x}_*\}^{\bot},$ since $-\bm{s}$ is a positive semidefinite matrix and the equality $\iprod{\bm{s}}{\bm{x}_*} = 0$  is equivalent to $\bm{x}_*\bm{s} = \bm{s}\bm{x}_* =0 ,$  we have $\bm{x}_*$ and $\bm{s}$ are simultaneously diagonalizable. Hence we have that the corresponding eigenvalue $\Lambda$, $\Lambda_s$ of $\bm{x}^*$ and $\bm{s}$ satisfy $\Lambda \Lambda_s = 0$. Consequently, there exist $\alpha_1,\cdots,\alpha_{n-r}$ such that $\bm{s} = \sum_{i=1}^{n-r} \alpha_i V_{i+r}V^*_{i+r} \in \mathcal{S},$ which means $-\mathcal{K}^* \cap \{\bm{x}_*\}^{\bot} = \mathcal{S}.$ Since the affine hull of $\mathcal{S}$ is $\bar{\mathcal{S}}:= \{\bm{s}| \bm{s} =  \sum_{i=1}^{n-r} \alpha_i V_{i+r}V^*_{i+r} \}$, it follows that $ -\text{ri} (\mathcal{N}_{\mathcal{K}}(\bm{x}_*) ) = \{\bm{s}| \bm{s} = \sum_{i=1}^{n-r}\alpha_i V_{i+r}V^*_{i+r},\alpha > 0\},$ under which we can give examples where SC holds for concrete examples.

For the classical SDP problems, SC boils down to the fact that there exists $\bm{s}_* = \bm{c} - \mathcal{A}^*\bm{y}_*,$ such that $\bm{s}_* \in \text{ri}(\mathcal{K}^* \cap \{\bm{x}_*\}^{\bot})$ \cite[Proposition 2.44]{rockafellar2009variational}.
 Hence the SC holds for SDP if and only if \eqref{SC:SDP} holds.
For SDP+ problem, since $\mathcal{Q} = \{  \bm{b}\},$  and $h(\bm{x}) = \delta_{\bm{x} \ge 0}(\bm{x})$,  the SC boils down to the fact that there exists $\bm{z}_* \in -{\rm ri}(\mathcal{N}_{\{\bm{x} \ge 0\}}(\bm{x}_*) ), \bm{s}_*= \bm{c} -\mathcal{A}^*(\bm{y}_*) -  \bm{z}_*,  \bm{s}_* \in \text{ri}(\mathcal{K}^* \cap \{\bm{x}_*\}^{\bot}).$  Analogously, the SC holds for SDP+ problem if and only if \eqref{SC:SDP+} holds.
\end{proof}

% \begin{remark}
% If SC holds, according to \eqref{eqn:SC-all} in Proposition \ref{pro:SC},  $I_{\mathcal{Q}} = [m], I_h = [n \times n]$, and $|\alpha_1|+|\alpha_2| =n.$ Since ${\rm rank} (\bm{x}_* -  \sigma (\mathcal{A}^*(\bm{y}_*) -\bm{c} + \bm{z}_*)) = n,$ it follows that $F$ is locally smooth on the domain of $F$ near $\bm{w}_*$. Consequently,   Lemma \ref{lem:local-smooth} is an extension to the local smoothness result in \cite[Lemma 4.6]{hu2025analysis} based on SC.
% \end{remark}

\subsection{Complexity analysis}
We now present the iteration complexity of Algorithm \ref{alg:ssn} by assuming that parameters $\{\varsigma_k\}$ and $\{\|\bm{\eta}^k\|\}$ are set properly.
\begin{theorem} \label{com-ana}
Suppose that Assumption \ref{assum} holds. Let $\{\bm{w}^k\}$ be the sequence generated by Algorithm \ref{alg:ssn} with  parameters $\varsigma_k $ and $ \bm{\eta}^k  $ such that $ \varsigma_k  \le  Ck^{-\frac{5}{3}},  \|\bm{\eta}^k\|^2 \le Ck^{-2}$ for all $k$, where $C > 0$ is a  constant. Then for a given accuracy $\varepsilon > 0$,
it holds that for at most $K = \widetilde{\mathcal{O}}(\varepsilon^{-3/2} )$ iterations, we have
\begin{equation} \label{complexity}
 \|F(\bm{w}^K) \|^2 \le \varepsilon.
\end{equation}
Furthermore, let $\varepsilon_0$ be the threshold such that \eqref{eqn:superlinear} holds and $q$ be the exponent of the superlinear convergence.  If the Assumption in Theorem \ref{thm:superlinear convergence} holds, \eqref{complexity} is satisfied after at most $K = \widetilde{\mathcal{O}}(\varepsilon_0^{-3/2}) + \mathcal{O}( \log_{q}({\log(\frac{\varepsilon_0}{\varepsilon})} )) $ iterations.
\end{theorem}

\begin{proof}
We first consider the case where \eqref{eq:decrease-2} happens with finite times. It follows from \eqref{eqn:complexity-first} that \eqref{complexity} holds with at most $\mathcal{O} (\varepsilon^{-3/2} )$ iterations in this case.
The worst complexity of Algorithm \ref{alg:ssn} occurs in the case where the update of $\bm{w}^{k+1}$ associated with \eqref{eq:decrease-1} and \eqref{eq:decrease-2} both happen infinitely many times. Since $\sum_{j=1}^\infty \frac{4\kappa_1^2 L^2}{j^{3\beta}} < +\infty$ , there exists a large enough $K$ independent of $\varepsilon$ and a constant $\bar{\nu} \in (\nu ,1)$ such that
$C := \sqrt{\Pi_{j=K}^\infty \left(1 + \frac{4\kappa_1^2 L^2}{k^{3\beta}} \right)} \le \sqrt{ \exp(\sum_{j=K}^\infty \frac{4\kappa_1^2 L^2}{j^{3\beta}} )} < \frac{\bar{\nu}}{\nu}$.
Let $k_1,\cdots,k_n$ be the indices corresponding to the steps associated with \eqref{eq:decrease-1} is satisfied. Since for any $k > 0$,
$
\sum_{j = k +1}^{\infty} C{j^{-\frac{5}{3}}} \le C \int_{k}^{\infty} {x^{-\frac{5}{3}}} dx = C {k^{-\frac{2}{3}}}
$ and $ \sum_{j = k +1}^{\infty} C{j^{-\frac{7}{3} }} \le C \int_{k}^{\infty} x^{-7/3 } dx = C k^{-4/3}$.
Hence we need at most $k = \mathcal{O}(\varepsilon^{-3/2}) $ iterations
 such that $\sqrt{\sum_{j=k}^{\infty} \frac{4C}{l^{7/3}}} \le (b_2 - b_1)\varepsilon/[\zeta^2(1-\bar{\nu})C] $  and $ \sum_{j = k}^\infty \varsigma_j \le \sum_{j = k +1}^\infty \frac{C}{j^{5/3} } \le b_1 \varepsilon/[2\zeta^2(1-\bar{\nu})]$, where $b_1 \in (0,1), b_2 \in (b_1 ,1)$ are two constants.

Let $n_0$ be large enough such that $\bar{\pi}_{n_0} := \sum_{j=k_{n_0}}^{\infty}  \frac{C }{j^{7/3}} \le (b_2- b_1)\varepsilon/[2(1-\bar{\nu})\zeta^2]$. Subsequently, let $n$ be large enough that $n \ge n_0 + \zeta $ and $k_n \ge K \ge 4$.
According to the proof of Lemma \ref{lem:glo} and  \eqref{eq:residual-increase}, for any $k_n  \le j < k_{n+1}$,
    \bee \begin{aligned}
    & \|F(\bm{w}^{j+1})\|  \leq \sqrt{ (1+ \frac{4\kappa_1^2 L^2}{j^{3\beta}}) \|F(\bm{w}^{j})\|^2 + \frac{2}{\tau_{k,i} -1} \| \bm{\eta}^k \|^2 } \\
    & \leq \sqrt{\Pi_{k=k_n}^{j} (1 +   \frac{4 \kappa_1^2 L^2}{k^{3\beta}} )} \|F(\bm{w}^{k_{n}})\| +  \sqrt{\Pi_{k=k_n}^j (1 + \frac{4\kappa_1^2 L^2}{k^{3\beta}}) \sum_{l=k_n}^j \frac{4C }{l^{7/3 } }} \\
    & \le C \|F(\bm{w}^{k_{n}})\| + C \sqrt{\sum_{l = k_n}^j \frac{4C  }{l^{7/3 }}}.
    \end{aligned}
    \eee
Since $\nu C \le \bar{\nu}$, it follows from the definition of $C$ and $\bar{\pi}_j$ that
    \be \label{eqn:complex2}
    \begin{aligned}
    \| F(\bm{w}^{k_{n+1}}) \| & \leq \nu \max_{\max(1, n-\zeta + 1) \le j \le n}  \left( C\|F(\bm{w}^{k_{j}})\| + \bar{\pi}_j + \varsigma_{k_j} \right) \\
    & \le  \bar{\nu} a_n + \nu \max_{ \max(1, n- \zeta +1) \leq j \leq n} ( \varsigma_{k_j} + \bar{\pi}_{j}).
    \end{aligned}
    \ee
Subsequently, by the definition of $c_n$ in Theorem \ref{thm:global-con}, we have
\bee
\begin{aligned}
c_{n+1} & \overset{\eqref{eqn:complex2}}{\le} \bar{\nu} c_n + \nu \sum_{l=n \zeta}^{(n+1)\zeta - 1} \max_{ \max(1, l- \zeta +1) \leq j \leq l} (\bar{\pi}_j + \varsigma_{k_j})  \le \bar{\nu} c_n + \bar{\nu} \zeta^2  \bar{\pi}_{n \zeta } + \bar{\nu} \zeta^2 \sum_{j = k_{\zeta(n-1) }}^{k_{\zeta(n+1)}} \varsigma_{j}  \\
 &\le \bar{\nu}^{n-n_0} c_{n_0} + \zeta^2 \sum_{j=k_{n_0}}^{k_{\zeta(n+1)}}\bar{\nu}^{k-j+1}\bar{\pi}_{j}  + 2 \zeta^2 \sum_{j=k_
 {n_0}}^{k_{\zeta( n+1)}} \bar{\nu}^{k-j} \varsigma_{k_j} \\
% & < \nu^n \|F(\bm{w}^{k_0}) \| + (c_2-c_1)\varepsilon + \sum_{j=1}^k \nu^{k-j} C \varepsilon \\
& \le \bar{\nu}^{n-n_0} c_{n_0} + (b_2 - b_1) \varepsilon + b_1 \varepsilon = \bar{\nu}^{n-n_0} c_{n_0} + b_2 \varepsilon,
\end{aligned}
\eee
where the last equality follows from the definition of $\bar{\pi}_j$ and the assumption of $\varsigma_j$.
This implies that we need at most $\mathcal{O}(\log (\frac{1}{\varepsilon}))$ iterations of \eqref{eq:decrease-2} to make $\|F(\bm{w}^{k_n})\|^2 \le \varepsilon.$
By \eqref{eqn:complexity-first}, we have $k_{n+1} - k_n \le \mathcal{O}(\varepsilon^{-3/2})$.
Therefore,  the overall complexity is $\mathcal{O}(\varepsilon^{-3/2} \log(\frac{1}{\varepsilon})) +  \mathcal{O}(\varepsilon^{-3/2})$, dubbed as $\widetilde{\mathcal{O}}(\varepsilon^{-3/2} )$.

Furthermore, if the Assumption in Theorem \ref{thm:superlinear convergence} holds, it follows from \eqref{eqn:superlinear:descent} that there exists $K_1 > 0$ such that
$\|F(\bm{w}^{k+1}) \|^2 \le \tilde{c}_2^2 \gamma_{l}^{2- 2q} \|F(\bm{w}^k) \|^{2q}$,
where $q \in (1,2]$ for $k > K_1$. Hence  \eqref{complexity} holds with at most $\mathcal{O}( \text{log}_{q}({\log(\frac{\varepsilon_0}{\varepsilon})})$) iterations. The proof is completed.
\end{proof}

\section{Numerical experiments} \label{sec5}
To evaluate the effectiveness of SSNCP, we conduct extensive numerical experiments and compare it  with state-of-the-art solvers SDPNAL+ \citep{sun2020sdpnal+} and MOSEK \citep{MOSEK} comprehensively. We implement SSNCP using MATLAB R2023b. All experiments are performed on a Linux server with a sixteen-core Intel Xeon Gold 6326 CPU and 256 GB memory.

\subsection{Experiments setting}  \label{5-2}\label{parameter setting}
Consistent with the metrics commonly used to test for SDP problems \citep{yang2015sdpnal}, we evaluate the performance of SSNCP for SDP via the following quantity:
\[
\eta_1 = \max \{ \eta_p, \eta_d,\eta_{\mathcal{K}},\eta_{\mathcal{K}^*},\eta_{C_1} \},
\]
where $\bm{s} = \Pi_{\mathcal{K}}(\bm{c} - \mathcal{A}^*(\bm{y}) -\bm{z} - \bm{x} )$,
\begin{equation} \label{quantity}
    \begin{aligned}
    \eta_p & = \frac{ \|\mathcal{A}(\bm{x}) - \bm{b}\|}{1+\|\bm{b}\|}, \eta_d  = \frac{ \|\mathcal{A}^*(\bm{y}) + \bm{z}+\bm{s}- \bm{c}\|}{1+\|\bm{c}\|},\\
     \eta_{\mathcal{K}}&=  \frac{  \|\bm{x} - \Pi_{\mathcal{K}}(\bm{x})\|}{1+\|\bm{x}\|}, \eta_{\mathcal{K}^*}=  \frac{  \|\bm{s} - \Pi_{\mathcal{K}^*}(\bm{s})\|}{1+\|\bm{s}\|} ,\eta_{C_1} = \frac{| \iprod{\bm{x}}{\bm{s}}|}{1+\|\bm{x}\| + \|\bm{s}\|}.
    \end{aligned}
\end{equation}
 For SDP+ problems, we evaluate the performance of the tested algorithms via the following quantity
\[
\eta_2 = \max\{\eta_p,\eta_d,\eta_{\mathcal{K}}, \eta_{\mathcal{K}^*} ,\eta_{C_1}, \eta_{\mathcal{P}} , \eta_{C_2} \},
\]
where $\eta_p,\eta_d,\eta_{\mathcal{K}},\eta_{\mathcal{K}^*},\eta_{C_1}$ are defined in \eqref{quantity}, $\eta_{\mathcal{P}} = \frac{\| \Pi_{\mathcal{P}^*}(-\bm{x}) \|}{1+\|\bm{x}\|}$, and $\eta_{C_2} = \frac{|\iprod{\bm{x}}{\bm{z}}|}{1+\|\bm{x}\| +\|\bm{z}\| }.$  We stop SSNCP for SDP when $\eta_1 < 10^{-6}$ and SDP+ when $\eta_2 < 10^{-6}.$ We note that $\bm{z} \equiv 0$ in \eqref{quantity} for SDP problems. Furthermore, we also compute the relative gap
$\eta_{g} := \frac{| {\rm obj}_P -  {\rm obj}_D |}{1 + {\rm obj}_P  + {\rm obj}_D}$, where ${\rm obj}_P$ and ${\rm obj}_D$ denote the primal and dual objective function value.

% For SDP and SDP+ problems, according to the definition of $\mathcal{P}_{\theta,\delta,l,\sigma}$, the values of $\bm{x}$ and $\sigma \bm{z}$ are set to 0 with a threshold value $\frac{l}{2}$ on their absolute values.
% Let $\lambda$ be the eigenvalues of $\bm{x}+\sigma(\mathcal{A}^*(\bm{y}) + \bm{z}-\bm{c})$, $\mu = \text{vec}(\bm{q} - \sigma \bm{z})$. Define
% \begin{equation*} \label{def:theta}
% \begin{aligned}
% j_1 &= \argmin \left\{ k| \sum_{i=1}^k |\hat{\lambda}_i| \ge 0.999 \|\lambda\|_1 \right\}, j_2 = \argmin  \left\{ k| \sum_{i=1}^k |\hat{\mu}_i| \ge 0.999 \|\mu\|_1 \right\},
% \end{aligned}
% \end{equation*}
% where $\hat{\lambda}$ is obtained by sorting $\lambda$ such that $|\hat{\lambda}_1| \ge \cdots \ge |\hat{\lambda}_n|$ and $\hat{\mu}$ is defined analogously.  The
% threshold values are defined as $\theta^k := \hat{\lambda}_{j_1}$, $l^k := \hat{\mu}_{j_2}.$ It can be seen from the definition of $j_1$ and $j_2$ that when $\bm{w}$ is close enough to $\bm{W}_*$, $\theta^k \in (0, \bar{\theta})$, $l^k \in (0,\bar{l})$ and hence Assumption \ref{assumption:A4} holds.

\subsection{Numerical validation of theoretical results} \label{num-ver}

In this subsection, we verify the effectiveness of the correction step in Algorithm \ref{alg:ssn} for problems without SC. For SDP case, we generate the solution $\bm{x}_*$ and $\bm{s}_*$ to be the diagonal matrices with   ${\rm rank}(\bm{x}_*) + {\rm rank}(\bm{s}_*) < n$.
 We next set $\bm{b} = \mathcal{A}\bm{x}_*, \bm{c} = \mathcal{A}^*\bm{y}_* + \bm{s}_*$ where $\mathcal{A}$ and  $\bm{y}_*$ are random sparse matrix and vector, respectively.
 For the SDP+ case, the solution $\bm{x}_*$ and $\bm{s}_*$ are generated in the same way as in the SDP problem and $\bm{z}_*$ is a random nonnegative matrix with diagonal elements being 0. We also set $\bm{b}= \mathcal{A}\bm{x}_* $, $\bm{c} = \mathcal{A}^*\bm{y}_*+\bm{s}_* + \bm{z}_*.$ According to \eqref{SC:SDP}, \eqref{SC:SDP+} and the setting of the problems, SC does not hold for these tested problems.

In the following examples for SDP and SDP+ problems, we specify the parameter choices used in the algorithm. For SDP problems, the projection depends only on \(\theta^k\), with \(\bar{\theta} = 1\) determined by the choices of diagonal
matrices $\bm{x}_*$ and $\bm{s}_*$. We set \(\theta^k = 0.05\) when \(\|F(\bm{w}^k)\| < 10^{-3}\), and \(\theta^k = 0\) otherwise, so that the projection reduces to the identity in early iterations.
For SDP+ problems, the projection depends on both \(\theta^k\) and \(\rho^k\), with \(\bar{\theta} = \bar{\rho} = 1\) determined by $\bm{x}_*, \bm{s}_*$ and $\bm{z}_*$. We set \(\theta^k = \rho^k = 0.05\) when \(\|F(\bm{w}^k)\| < 10^{-3}\), and $0$ otherwise. In both cases, we set \( \|\bm{\eta}^k\| = 10^{-12} \) for all \(k\), ensuring that Assumptions~\ref{assumption:A3} and~\ref{assumption:A4} are satisfied. Algorithm~\ref{alg:ssn} terminates when either \(\eta_1 < 10^{-12}\) or \(\eta_2 < 10^{-12}\).

The numerical results are presented in Figure \ref{fig-combined-SDP}. We conjecture from the last column that the local error bound condition \ref{assumption:A2} holds, namely, $\text{dist}(\bm{w},\bm{w}_*) \le \|\bm{w}-\bm{w}_*\| < \|F(\bm{w}^k)\|$ when $\bm{w}^k$ is close to $\bm{w}_*$. This suggests that all conditions stated in Assumption~\ref{assum} are satisfied. It is shown that the iterate sequence without a correction step exhibits oscillations, while the iteration points with a correction step avoid this problem. Furthermore, the superlinear convergence rate is observed, which verifies the conclusion of Theorem \ref{thm:superlinear convergence}.

\begin{figure}[h]
    \centering
    % 第一行：SDP 结果
    \subfigure{\includegraphics[width=0.24\textwidth]{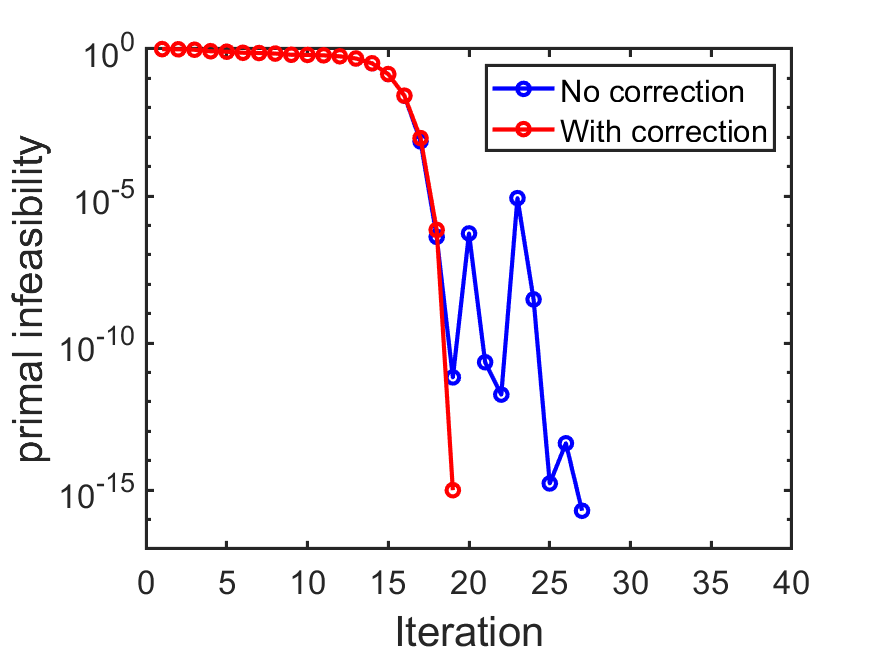}}
    \hfill
    \subfigure{\includegraphics[width=0.24\textwidth]{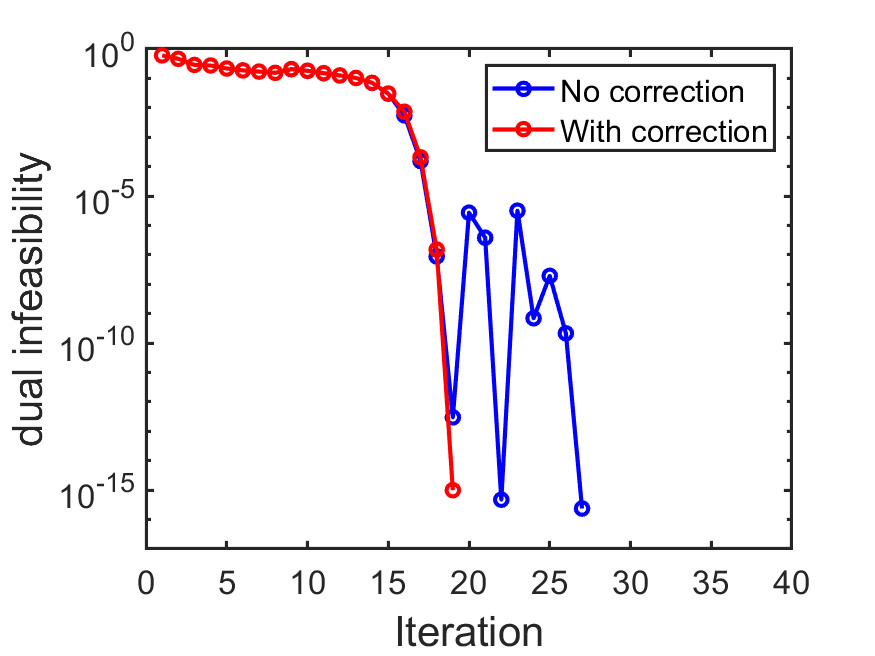}}
    \hfill
    \subfigure{\includegraphics[width=0.24\textwidth]{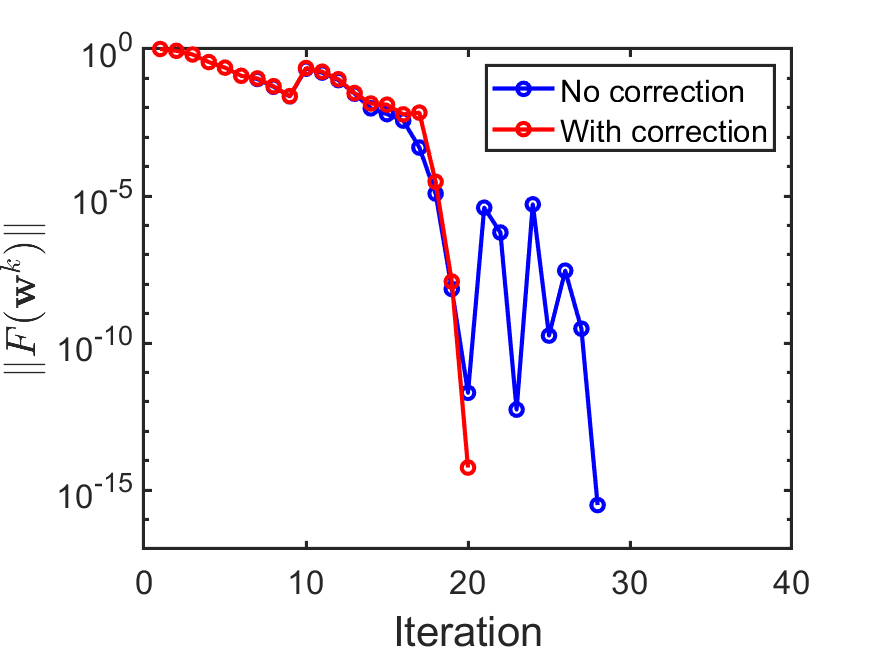}}
    \hfill
    \subfigure{\includegraphics[width=0.24\textwidth]{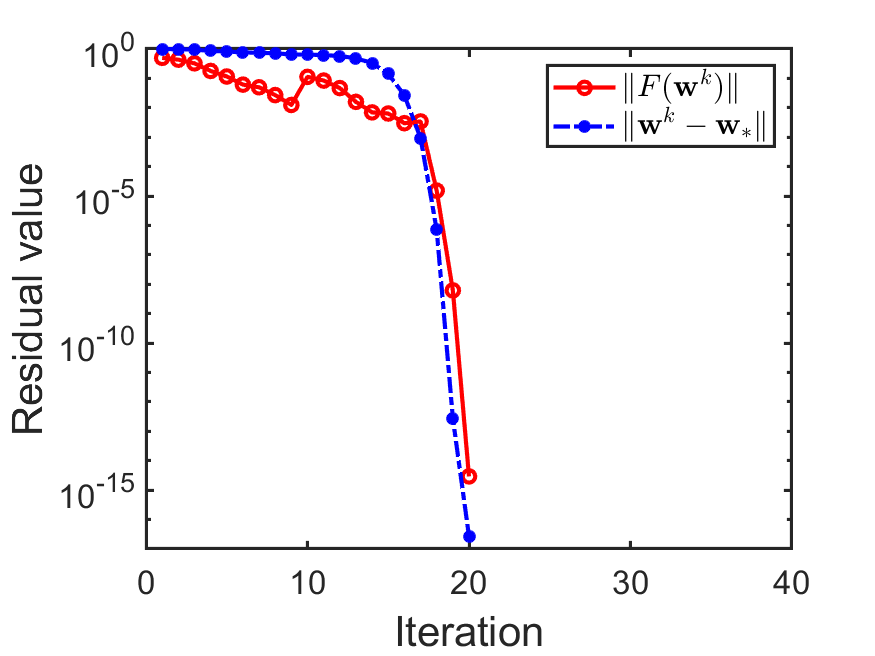}}
    % 第二行：SDP+ 结果
    \subfigure{\includegraphics[width=0.24\textwidth]{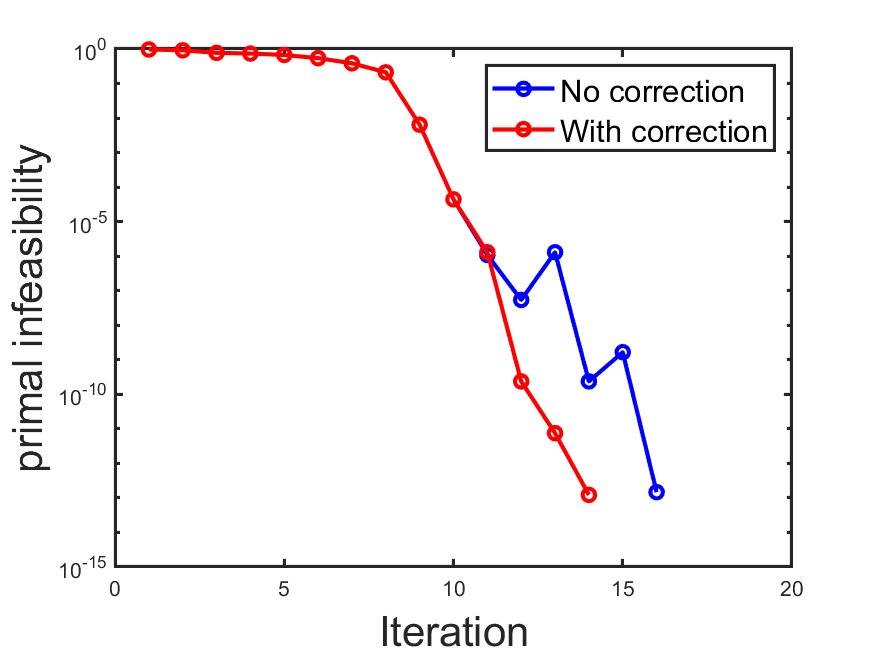}}
    \hfill
    \subfigure{\includegraphics[width=0.24\textwidth]{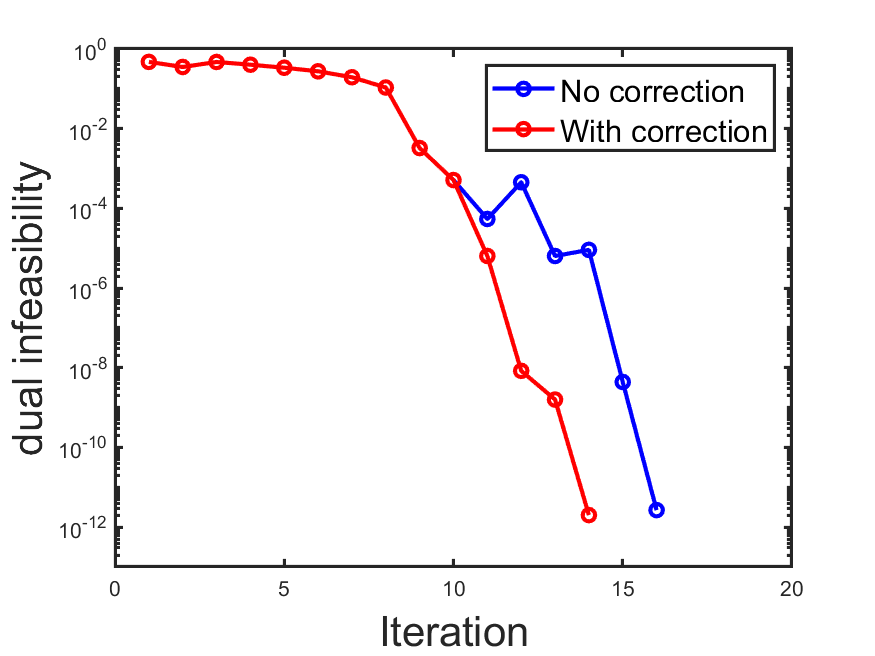}}
    \hfill
    \subfigure{\includegraphics[width=0.24\textwidth]{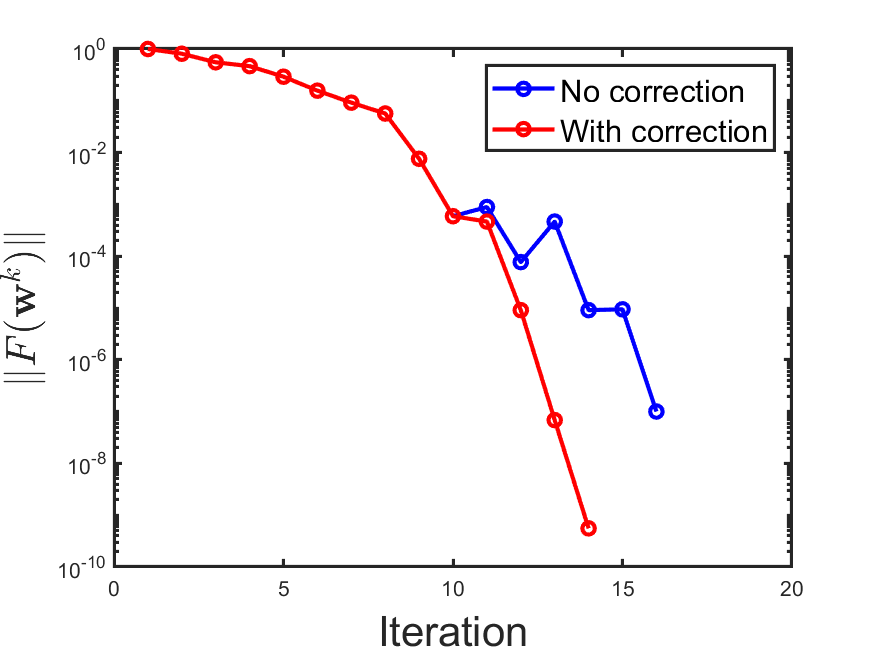}}
    \hfill
    \subfigure{\includegraphics[width=0.24\textwidth]{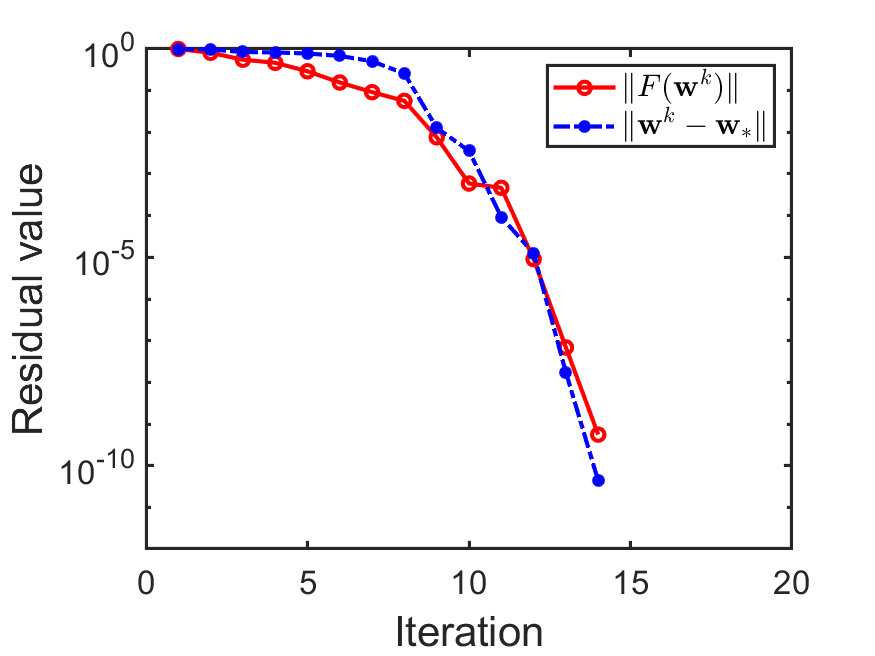}}
    \caption{
        The first three columns: the performance of SSNCP for SDP (top row) and SDP+ (bottom row) problems
        with or without the correction step. The last column: $\|F(\bm{w}^k)\|$ and $\|\bm{w}^k - \bm{w}_*\| $ of SSNCP with correction step.
    }
    \label{fig-combined-SDP}
\end{figure}

\subsection{Lov{\'a}sz Theta problem}

Let $G=(V, E)$ be a simple, undirected graph. The Lov{\'a}sz theta problem \citep{sloane2005challenge,trick1992second} is defined as
\bee
      \theta(G) = \min_X \; \iprod{-ee^T}{X}, \text{ s.t. }  \tr(X) = 1, \; X \succeq 0,\;X_{ij} = 0,\; (i, j) \in E,
\eee
which is a SDP problem. The dimension $n$ of the tested problems varies from 64 to 4096 and the number of constraints $m$ varies from 513 to 504452. The numerical results are listed in Figure \ref{fig:theta-perf} and Table \ref{tab:theta-stat}. As shown in Figure \ref{fig:theta-perf}, compared with MOSEK and SDPNAL+, SSNCP successfully solve all 69 problems while MOSEK can only solve 73.9$\%$ of them due to the limited memory. The ``error" means the comparison of residual error $\eta_1$. According to Table \ref{tab:theta-stat},  SSNCP is faster than SDPNAL+ and MOSEK on 71.0 \% of the 69 tested problems. Particularly, for large-scale problems such as \texttt{1zc.4096} and \texttt{2dc.2048}, where $(m,n) = (92161, 4096)$ and $(504452,2048)$, respectively, SSNCP is two times faster than SDPNAL+ while MOSEK can not solve them due to the lack of memory.

\begin{figure}[h]
\centering
\subfigure[error]{
\includegraphics[width=0.45\textwidth]{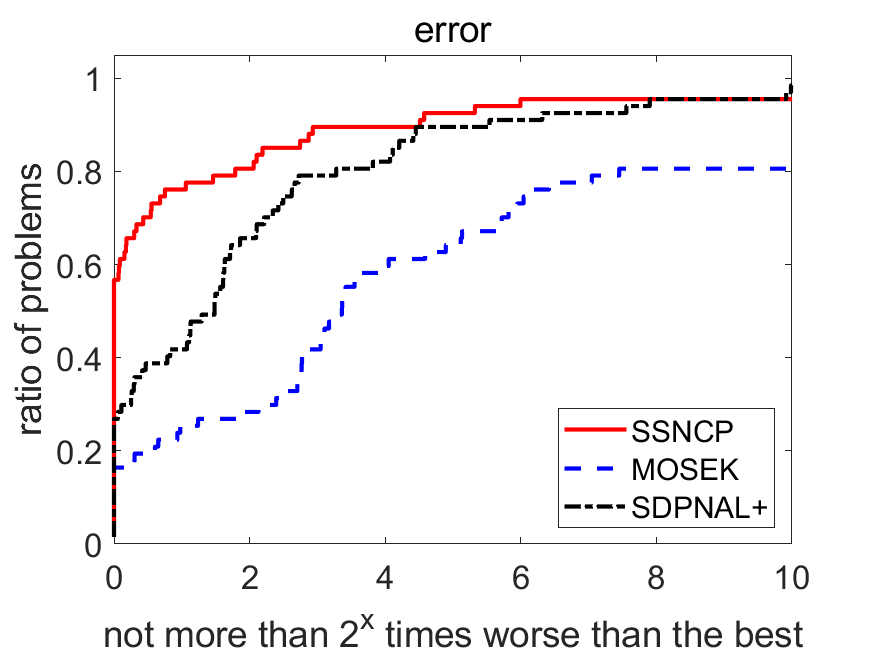}}
\subfigure[CPU]{
\includegraphics[width=0.45\textwidth]{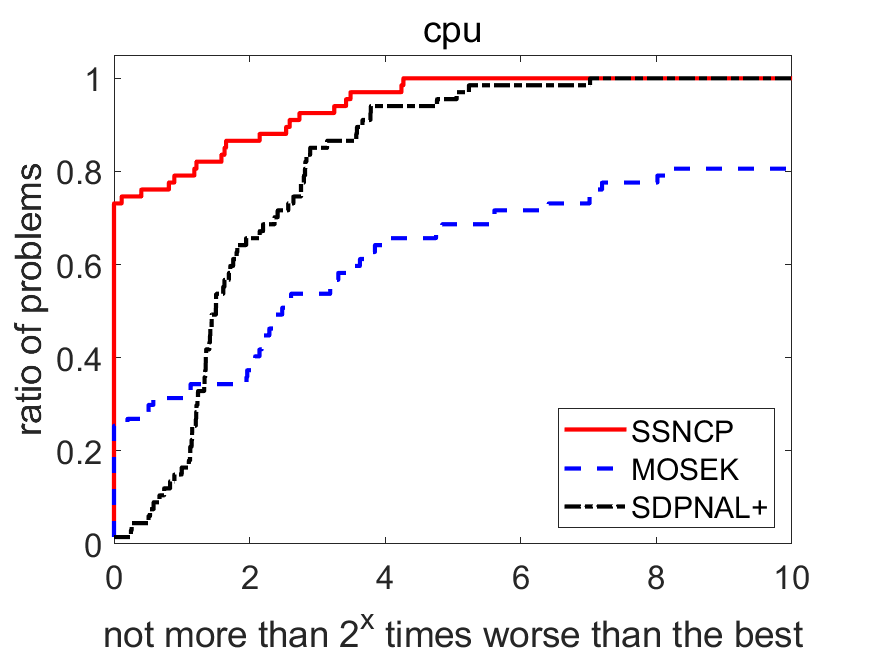}}
\caption{The performance profiles of tested algorithms for Theta problems}\label{fig:theta-perf}
\end{figure}

\begin{table}[h]
\centering
\setlength{\tabcolsep}{5pt}
\caption{A statistic of computational results of tested algorithms for Theta problems}\label{tab:theta-stat}
\begin{tabular}{|c|cc|cc|cc|}
% \cline{1-7}
\cline{1-7}
& \multicolumn{2}{c|}{SSNCP}
& \multicolumn{2}{c|}{SDPNAL+}
& \multicolumn{2}{c|}{MOSEK}
\\
\cline{1-7}
success & 69 &  100.0\% & 68 & 98.6\% & 51 & 73.9\%  \\
fastest & 49 &  71.0\% & 1 & 1.4\% &19 &  27.5\%\\
fastest under success & 49 &  71.0\% & 1 & 1.5\% & 17 &  33.3\% \\
not slower 3 times & 69 &  100.0\% & 36 & 52.2\% & 25 & 36.2\% \\
not slower 3 times under success & 69 &  100.0\% & 46 & 67.6\% & 23 & 45.1\%  \\
\cline{1-7}
% \cline{1-7}
\end{tabular}
\end{table}

\subsection{Lov{\'a}sz Theta+ problem}
% The SDP+ problems arising from the relaxation of maximum stable set problems.
Given a graph $G$ with the edge set $\mathcal{E},$ the SDP+ relaxation of the maximum stable set problem \citep{sloane2005challenge,trick1992second} is given by:
\bee
\theta_+(G) = \min\{\iprod{-ee^{\mathrm{T}}}{X}: \tr(X) = 1, \; X \succeq 0,\;X_{ij} = 0,\; (i, j) \in E, X \in \mathcal{P} \},
\eee
where the polyhedral cone $\mathcal{P} = \{X \in \mathbb{S}^n | X \ge 0 \}.$  The numerical results are listed in Figure \ref{fig:theta-plus-perf} and Table \ref{tab:thetaplus-stat},  where ``error" means the comparison of residual error $\eta_2$.
It is shown in Figure \ref{fig:theta-plus-perf} that SSNCP and SDPNAL+ can solve all the problems successfully while MOSEK cannot due to the lack of memory. In summary, SSNCP achieves higher accuracy than SDPNAL+ and MOSEK on most problems.  According to the recorded outputs of all algorithms, SSNCP is faster than SDPNAL+ and MOSEK on 72.5\% of the 69 problems.

\begin{figure}[h]
\centering
\subfigure[error]{
\includegraphics[width=0.45\textwidth]{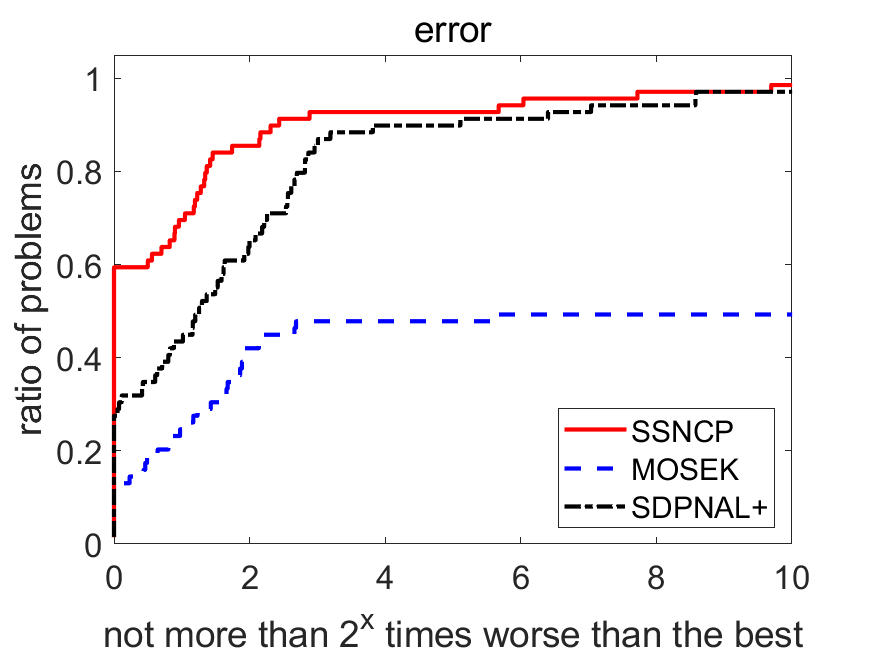}}
\subfigure[CPU]{
\includegraphics[width=0.45\textwidth]{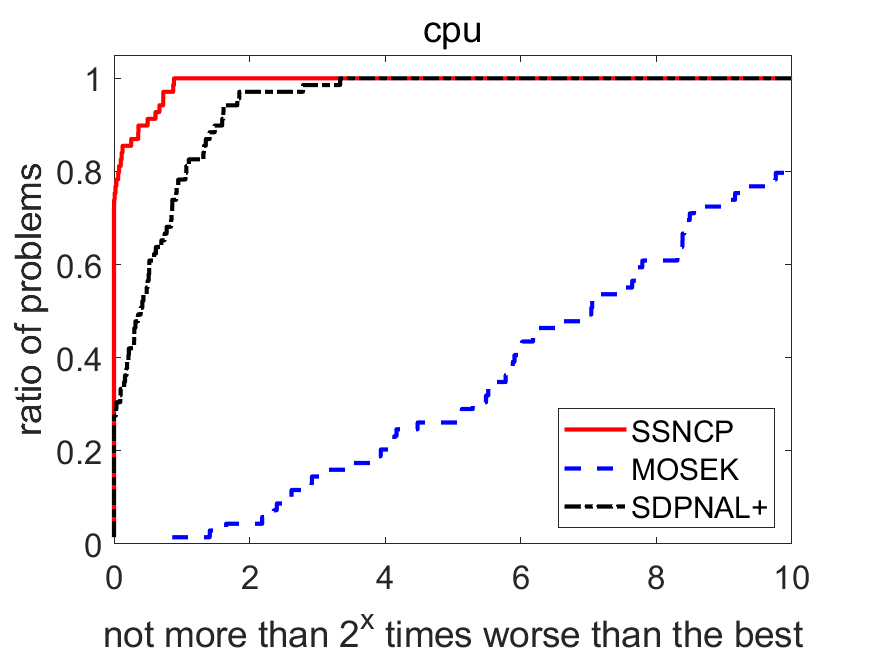}}
\caption{The performance profiles of tested algorithms for Theta+ problems}\label{fig:theta-plus-perf}
\end{figure}

\begin{table}[htbp]
\centering
\setlength{\tabcolsep}{5pt}
\caption{A statistic of computational results of tested algorithms for Theta+ problems}\label{tab:thetaplus-stat}
\begin{tabular}{|c|cc|cc|cc|}
% \cline{1-7}
\cline{1-7}
& \multicolumn{2}{c|}{SSNCP}
& \multicolumn{2}{c|}{SDPNAL+}
& \multicolumn{2}{c|}{MOSEK}
\\
\cline{1-7}
success & 69 &  100.0\% & 69 & 100.0\% & 13 & 18.8\%  \\
fastest & 50 &  72.5\% & 19 & 27.5\% &0 &  0.0\%\\
fastest under success & 50 &  72.5\% & 19 & 27.5\% & 0 &  0.0\% \\
not slower 3 times & 69 &  100.0\% & 62 & 89.9\% & 2 & 2.9\% \\
not slower 3 times under success & 69 &  100.0\% & 62 & 89.9\% & 1 & 7.7\%  \\
\cline{1-7}
% \cline{1-7}
\end{tabular}
\end{table}

\subsection{ Rank-1 tensor approximations (R1TA)}
The dual form of the R1TA problem \citep{nie2014semidefinite}
is given by
\begin{equation} \label{R1TA:dual}
    \min_{x \in \mathbb{R} } \;\; x, \quad \st \;\, x g - f = M^*(X), X \in \mathbb{S}_+^n ,
\end{equation}
which is a standard SDP \citep{nie2012regularization}, where the linear map $\mathcal{A}$ depends on $M,f$ and $g$. Figure \ref{fig:R1TA-perf} presents the performance profiles of the tested algorithms, which shows that SSNCP is the fastest in nearly all problems and MOSEK can not. Particularly, for large-scale problems such as nonsym(20,4) and nonsym(21,4), where the corresponding $(m, n) = (12326390, 9261)$ and $(9260999, 8000)$, respectively, SSNCP is able to solve these problems successfully in less than 2 hours while both SDPNAL+ and MOSEK fail. As shown in Table \ref{tab:R1TA-stat}, SSNCP can solve all the tested problems while MOSEK can not due to limited memory.

\begin{figure}[h]
\centering
\subfigure[error]{
\includegraphics[width=0.45\textwidth]{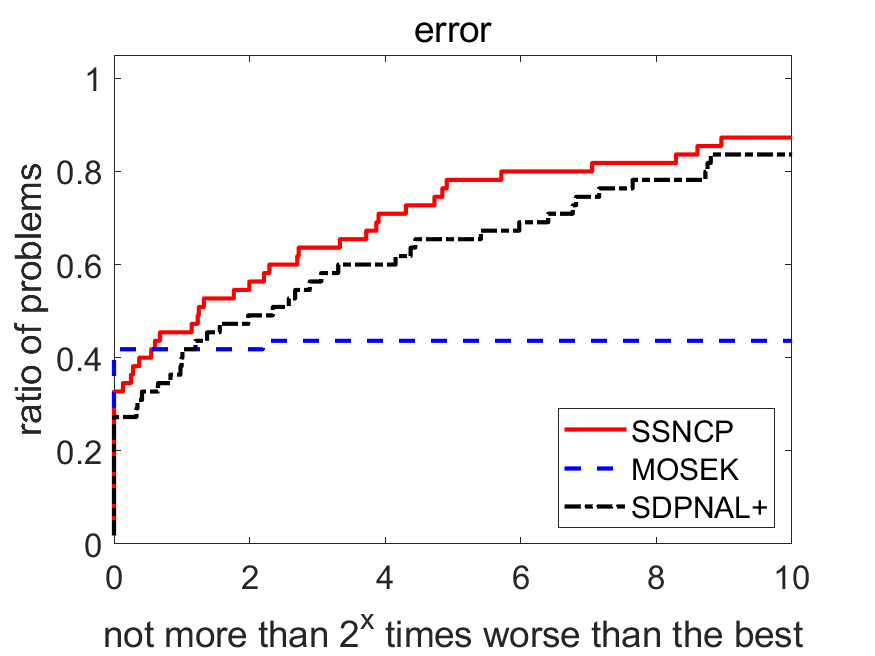}}
\subfigure[CPU]{
\includegraphics[width=0.45\textwidth]{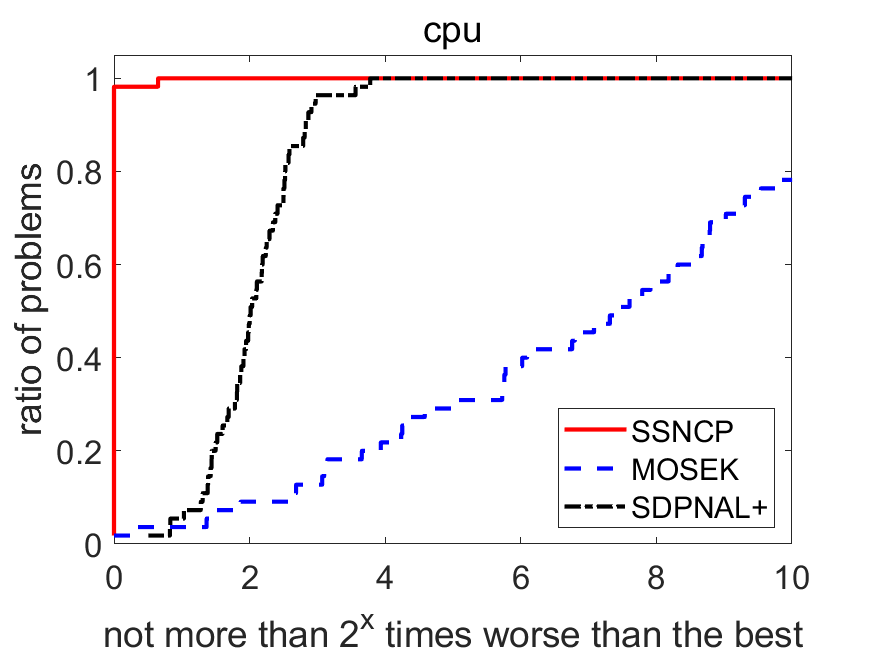}}
\caption{The performance profiles of tested algorithms for R1TA problems}\label{fig:R1TA-perf}
\end{figure}

\begin{table}[h]
\centering
\setlength{\tabcolsep}{5pt}
\caption{A statistic of computational results of tested algorithms for R1TA problems}\label{tab:R1TA-stat}
\begin{tabular}{|c|cc|cc|cc|}
% \cline{1-7}
\cline{1-7}
& \multicolumn{2}{c|}{SSNCP}
& \multicolumn{2}{c|}{SDPNAL+}
& \multicolumn{2}{c|}{MOSEK} \\
\cline{1-7}
success & 55 &  100.0\% & 52 & 94.5\% & 25 & 45.5\%  \\
fastest & 54 &  98.2\% & 0 & 0.0\% &1 &  1.8\%\\
fastest under success & 54 &  98.2\% & 0 & 0.0\% & 1 &  4.0\% \\
not slower 3 times & 55 &  100.0\% & 14 & 25.5\% & 3 & 5.5\% \\
not slower 3 times under success & 55 &  100.0\% & 12 & 23.1\% & 2 & 8.0\%  \\

\cline{1-7}
% \cline{1-7}
\end{tabular}
\end{table}

\subsection{Reduced density matrix (RDM) formulation problem}
The RDM formulation for the electronic structure calculation problem is formulated as the constrained minimization of the total energy of the molecular system
 subject to the N-representability condition as an SDP \citep{li2018semismooth}:
 \begin{equation} \label{RDM:dual}
 \begin{aligned}
     \min_{y,S_j}\quad  &b^{\mathrm{T}}y,\\
     \st \quad S_j &= \mathcal{A}^*_jy - C_j, j=1,\cdots,l,
     B^{\mathrm{T}}y = c,\\
     0 &\preceq S_1 \preceq I,~~   0 \preceq S_j ,~~ j =2,\cdots,l.
     \end{aligned}
     \end{equation}
We compared SSNCP, SDPNAL+, MOSEK, and SSNSDP \citep{li2018semismooth} on the RDM dataset.
The test results are shown in Figure \ref{fig:rdm-perf}. According to Table \ref{tab:rdm-stat}, SSNCP returns solutions with accuracy under $10^{-6}$ in all cases. Furthermore, SSNCP is faster than SDPNAL+, MOSEK, and SSNSDP on $55.8\%$ of the 276 problems.

\begin{figure}[h]
\centering
\subfigure[error]{
\includegraphics[width=0.45\textwidth]{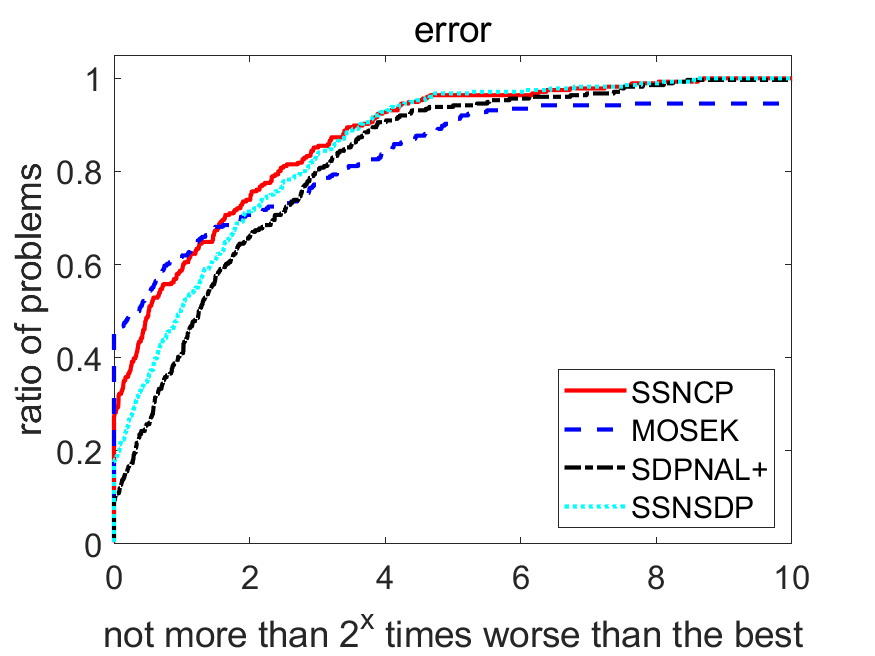}}
\subfigure[CPU]{
\includegraphics[width=0.45\textwidth]{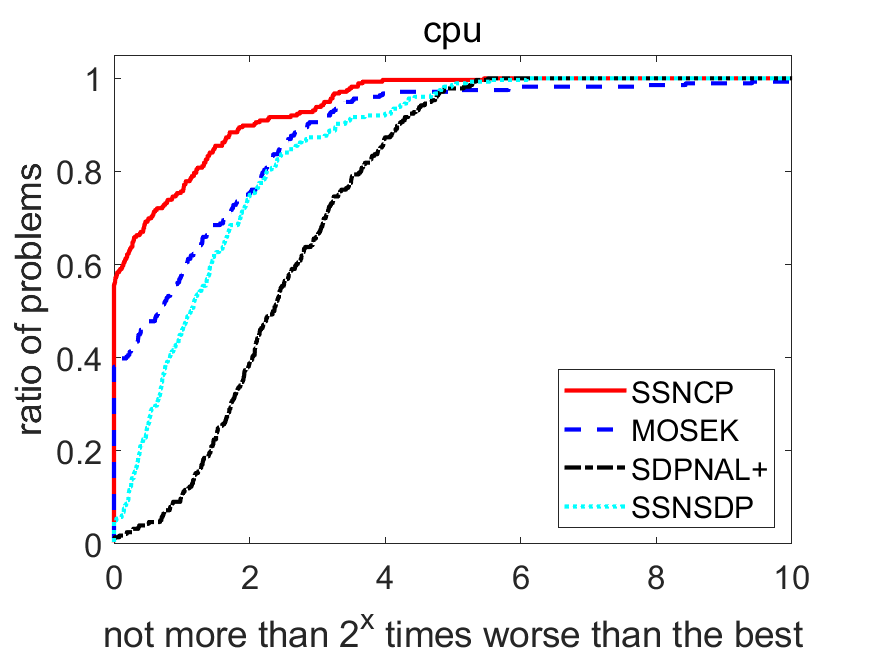}}
\caption{The performance profiles of tested algorithms for RDM problems}\label{fig:rdm-perf}
\end{figure}

\begin{table}[h]
\centering
\setlength{\tabcolsep}{5pt}
\caption{A statistic of computational results of tested algorithms for RDM problems}\label{tab:rdm-stat}
\begin{tabular}{|c|cc|cc|cc|cc|}
\cline{1-9}
& \multicolumn{2}{c|}{SSNCP}
& \multicolumn{2}{c|}{SDPNAL+}
& \multicolumn{2}{c|}{SSNSDP}
& \multicolumn{2}{c|}{MOSEK}
\\
\cline{1-9}
success & 276 &  100.0\% & 239 & 86.6\% & 276 & 100.0\% & 243 & 88.0\%  \\
fastest & 154 &  55.8\% & 3 & 1.1\% &14 &  5.1\% &105 &  38.0\% \\
fastest under success & 154 &  55.8\% & 5 & 2.1\% & 18 &  6.5\%  & 103 &  42.4\% \\
not slower 3 times & 273 &  98.9\% & 103 & 37.3\% & 225 & 81.5\%  & 190 & 68.8\%  \\
not slower 3 times under success & 271 &  98.2\% & 98 & 41.0\% & 225 & 81.5\% & 180 & 74.1\% \\
\cline{1-9}
\end{tabular}
\end{table}

\subsection{Binary integer nonconvex
quadratic programming}
Consider the SDP+ problem coming from the relaxation of a binary integer nonconvex
quadratic (BIQ) programming.
The relaxed problem has the following form \citep{burer2009copositive}:
\begin{equation} \label{pro:biq}
\begin{aligned}
    \min_{X}\; \quad & \frac{1}{2} \iprod{Q}{X_0} + \iprod{c}{x}, \\
    \st \quad & \text{diag}(X_0) - x =0, \alpha = 1, X = \left[ \begin{array}{cc}
	X_0 & x \\
	x^{\mathrm{T}} & \alpha \\
\end{array} \right] \in \mathbb{S}_+^n,\quad X \in \mathcal{P},
\end{aligned}
\end{equation}
where the polyhedral cone $\mathcal{P} = \{X \in \mathbb{S}^n \,|\, X \ge 0 \}$. In our numerical experiments, the 135 test data for $Q$ and $c$ are taken from the Biq Mac Library\footnote{\url{http://biqmac.uni-klu.ac.at/biqmaclib.html.}}. The numerical results are listed in  Figure \ref{fig:biq-perf} which shows that SSNCP and SDPNAL+ can both solve all the problems. According to the recorded outputs of all algorithms, SSNCP is faster than SDPNAL+ and MOSEK on $80.0\%$ of the 135 problems.

\begin{figure}[h]
\centering
\subfigure[error]{
\includegraphics[width=0.45\textwidth]{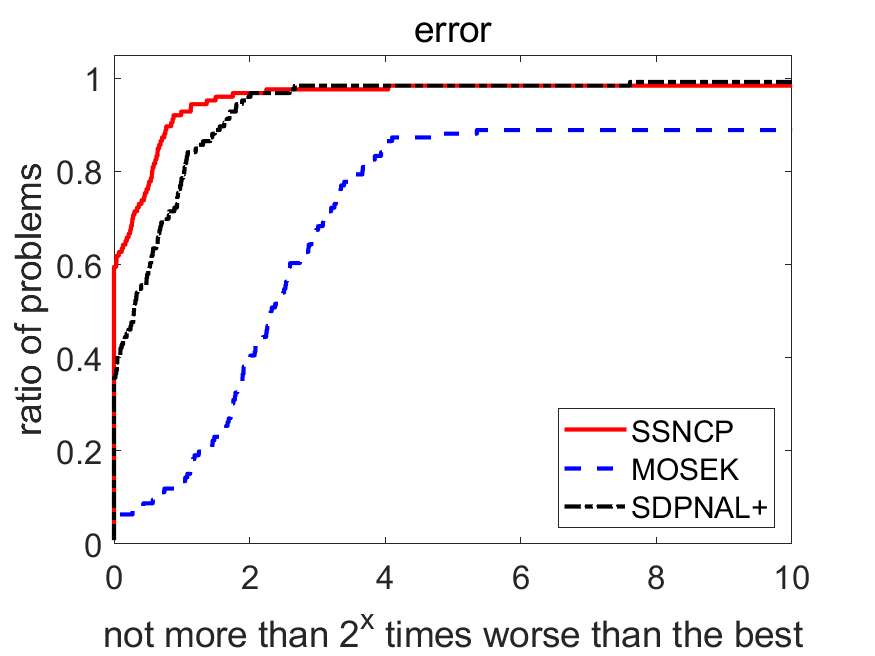}}
\subfigure[CPU]{
\includegraphics[width=0.45\textwidth]{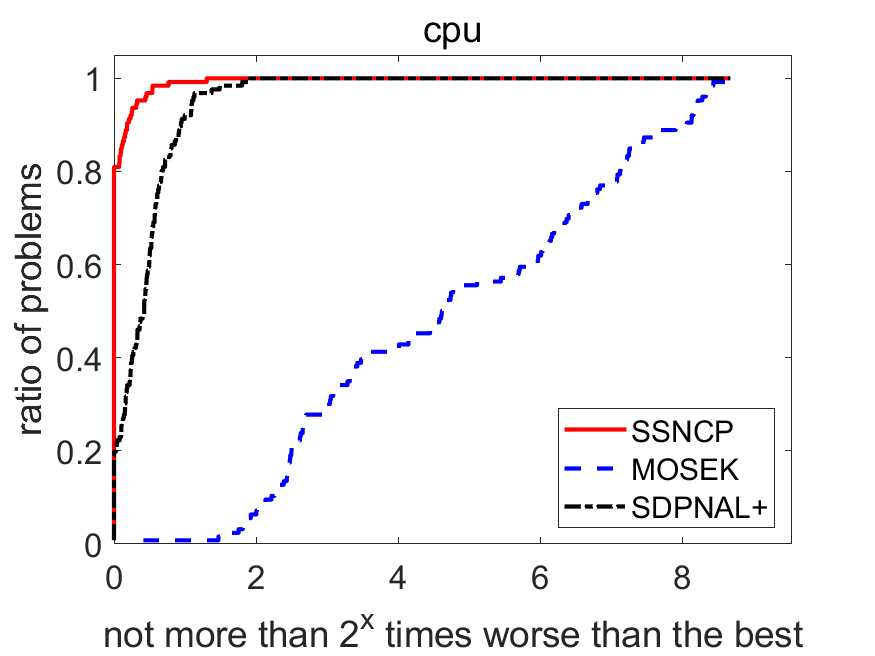}}
\caption{The performance profiles of tested algorithms for BIQ problems}\label{fig:biq-perf}
\end{figure}

\begin{table}[h]
\centering
\setlength{\tabcolsep}{5pt}
\caption{A statistic of computational results of tested algorithms for BIQ problems}\label{tab:biq}
\begin{tabular}{|c|cc|cc|cc|}
\cline{1-7}
& \multicolumn{2}{c|}{SSNCP}
& \multicolumn{2}{c|}{SDPNAL+}
& \multicolumn{2}{c|}{MOSEK}
\\
\cline{1-7}
success & 135 &  100.0\% & 133 & 98.5\% & 73 & 54.1\%  \\
fastest & 108 &  80.0\% & 27 & 20.0\% &0 &  0.0\%\\
fastest under success & 108 &  80.0\% & 27 & 20.3\% & 0 &  0.0\% \\
not slower 3 times & 133 &  98.5\% & 133 & 98.5\% & 3 & 2.2\% \\
not slower 3 times under success & 133 &  98.5\% & 131 & 98.5\% & 1 & 1.4\%  \\
\cline{1-7}
\end{tabular}
\end{table}

\subsection{Relaxation of clustering problems}
The SDP$+$ relaxation of clustering problems (RCP) described in \citep{peng2007approximating} can be represented as

\begin{equation}
\min\; \iprod{-W}{X}, \text{ s.t. }\,Xe=e,\mathrm{tr}(X)=K, X\geq 0, X\succeq 0.
\end{equation}
where $W$ is the affinity matrix whose entries represent the similarities of the objects in the dataset, $e$ is the vector of ones, and $K$ is the number of clusters.  All the datasets we tested are selected from the UCI Machine Learning Repository\footnote{\url{ http://archive.ics.uci.edu/ml/datasets.html}}, including ``abalone'', ``segment'', ``soybean'' and ``spambase''.
The numerical results comparing with SDPNAL+ and MOSEK are shown in Figure \ref{fig:rcp-perf} and Table \ref{tab:biq}, these figures again show that the accuracy and the CPU time
of SSNCP are better than SDPNAL+ on most problems. According to Table \ref{tab:biq}, SSNCP is faster than SDPNAL+ and MOSEK on $76.7\%$ of the 120 problems.

\begin{figure}[h]
\centering
\subfigure[error]{
\includegraphics[width=0.45\textwidth]{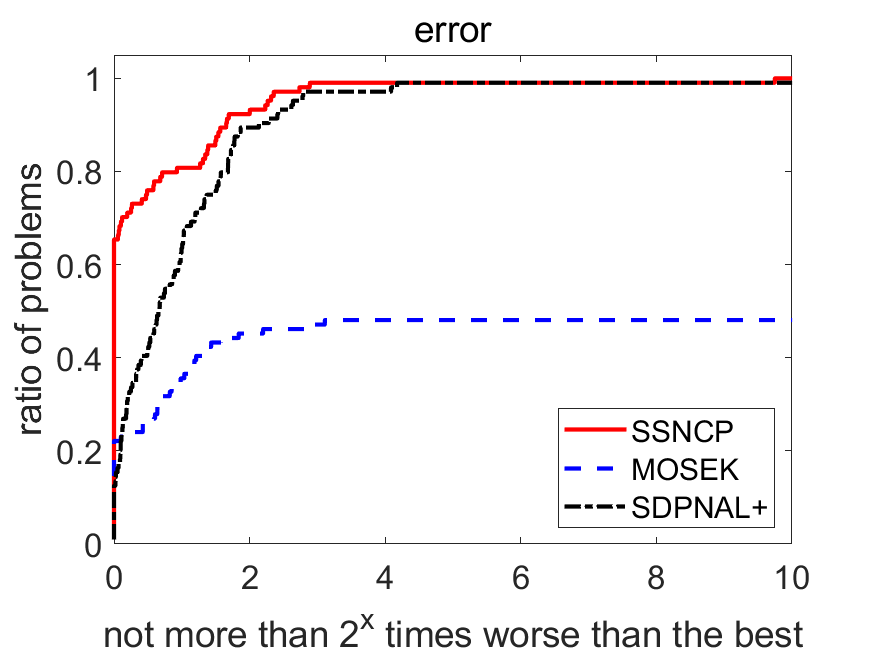}}
\subfigure[CPU]{
\includegraphics[width=0.45\textwidth]{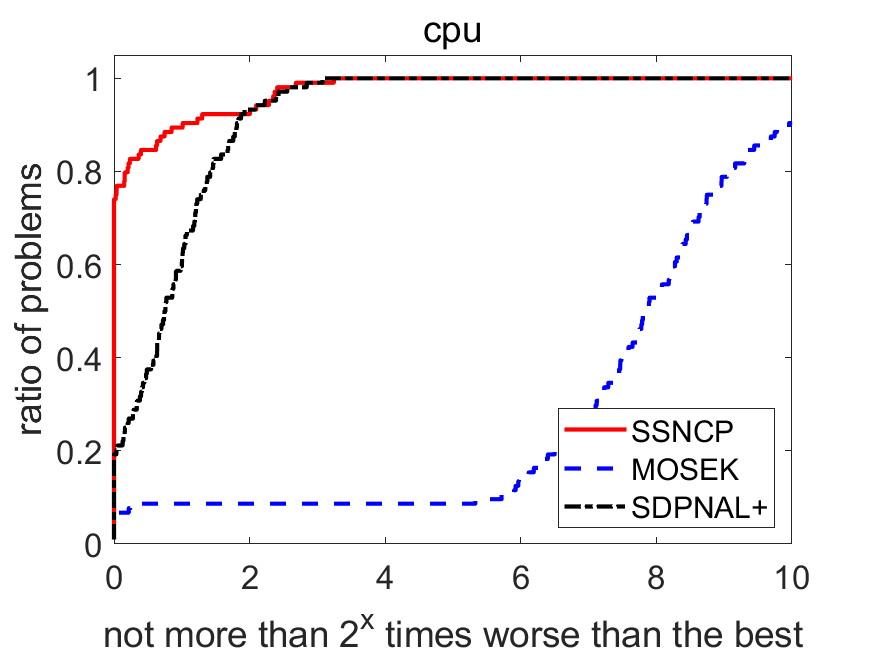}}
\caption{The performance profiles of tested algorithms for RCP problems}\label{fig:rcp-perf}
\end{figure}

\begin{table}[h]
\centering
\setlength{\tabcolsep}{5pt}
\caption{A statistic of computational results of tested algorithms for RCP problems}\label{tab:rcp}
\begin{tabular}{|c|cc|cc|cc|}
\cline{1-7}
& \multicolumn{2}{c|}{SSNCP}
& \multicolumn{2}{c|}{SDPNAL+}
& \multicolumn{2}{c|}{MOSEK}
\\
\cline{1-7}
success & 120 &  100.0\% & 120 & 100.0\% & 36 & 30.0\%  \\
fastest & 92 &  76.7\% & 20 & 16.7\% &8 &  6.7\%\\
fastest under success & 92 &  76.7\% & 20 & 16.7\% & 3 &  8.3\% \\
not slower 3 times & 115 &  95.8\% & 100 & 83.3\% & 10 & 8.3\% \\
not slower 3 times under success & 115 &  95.8\% & 103 & 85.8\% & 5 & 13.9\%  \\
\cline{1-7}
\end{tabular}
\end{table}

\subsection{Mittelmann benchmark}
The Mittelmann benchmark\footnote{\url{https://plato.asu.edu/ftp/sparse_sdp.html} \label{fn:sametext} } is widely recognized as the most authoritative benchmark for SDP, where numerous commercial solvers are tested. In our experiments, all solvers run on four CPUs with 16GB in total. We stop SSNCP, SDPNAL+ and MOSEK when $\eta_1 < 10^{-6}$.  For problems that are difficult to achieve high precision or when stagnation occurs, we stop the SSNCP at a medium accuracy. For consistency, different from the setting of SDP stated in Section \ref{parameter setting}, we use the same criteria as Mittelman's benchmark. To be specific, the ``success-1e-2'' means the number of solutions that satisfy $\min\{ \eta_1,\eta_g\} < 10^{-2}$ while ``success-1e-4'' denotes the number of solutions that satisfy $ \eta_1 < 10^{-4}$. We note that the criterion ``success-1e-2" is also used in the test of the Mittelmann benchmark\textsuperscript{\ref{fn:sametext}} to decide whether a problem is solved.
The maximum iteration time is set as $2 \times 10^4$ seconds and the maximum number of iterations is $10^5$. In order to better demonstrate the efficiency of our algorithm, different parameters are chosen for challenging classes of problems for SSNCP.

The results of the numerical experiments are listed in Figure \ref{fig:mittelmann-perf} and Table \ref{tab:mittelmann}. The geomean over all instances is calculated by
$
    \text{geomean} = \left( \Pi_{i=1}^n (t_i + \zeta_0) \right)^{\frac{1}{n}} - \zeta_0,
$
where $n$ denotes the number of instances, $t_i$ denotes the time consumed by the
$i$-th instance and $\zeta_0 > 0$ is a small value to avoid the case that $t_i$ is small. We set
$\zeta_0$ = 10 seconds in the computation which is the same as the Mittlemann benchmark.
These figures show that SSNCP is competitive with MOSEK and SDPNAL+ on the challenging Mittelmann benchmark on CPU time and accuracy. Furthermore, it is shown in Table \ref{tab:mittelmann} that on all the 75 tested problems, the geometric mean time of SSNCP is significantly lower than that of SDPNAL+, especially MOSEK. The results demonstrate that SSNCP has the potential to solve challenging large-scale SDP. Specifically, for challenging problems arising from Theta dataset such as the \texttt{1et.2048} and \texttt{1tc.2048} problems, SSNCP can solve them in less than 40 minutes, which is twice faster than SDPNAL+ while MOSEK cannot due to limited memory.  For problems arising from minimum bisection or max-cut problems such as \texttt{G60mb} and   \texttt{G60mc}, SSNCP is able to solve them in less than 1 hour while MOSEK needs more than 3 hours to solve them.  In addition, for problems with large $m$ such as \texttt{hamming956} and \texttt{theta123}, MOSEK fails to solve them due to the lack of memory while SSNCP can solve them in less than 5 seconds.

\begin{table}[H]
    \centering
    \setlength{\tabcolsep}{1.3pt}
    \caption{Selected computational results of tested algorithms on Mittlemann benchmark.}
    \label{tab:merged}
    \scalebox{0.7}{\begin{tabular}{|c|c|c|cccc|cccc|cccc|}
        \cline{1-15}
        \multirow{2}{*}{id} %
        &\multirow{2}{*}{m}
        &\multirow{2}{*}{n}
        & \multicolumn{4}{|c|}{SSNCP}
        & \multicolumn{4}{c|}{SDPNAL+}
        & \multicolumn{4}{c|}{MOSEK}\\
        \cline{4-15}
        &  &  & $\eta_p$ & $\eta_d$  & iter & time
        & $\eta_p$ & $\eta_d$ & iter & time
        & $\eta_p$ & $\eta_d$ & iter & time \\ \cline{1-15}
G40mb& 2000 & 2001 & 3.6e-07 & 2.0e-08 & 117 & \textbf{94.2} & 7.6e-07 & 8.4e-07 & 12319 & 9954.6 & 2.7e-06 & 1.7e-07 & 24 & 288.9 \\ \cline{1-15}
G48mb& 3000 & 3001 & 1.3e-10 & 2.5e-08 & 125 & \textbf{227.2} & 1.6e-07 & 6.8e-07 & 1511 & 2207.0 & 2.7e-13 & 1.8e-14 & 8 & 304.2 \\ \cline{1-15}
G48mc& 3000 & 3000 & 5.8e-08 & 9.2e-07 & 96 & \textbf{170.1} & 4.4e-07 & 9.3e-07 & 1942 & 2667.5 & 2.5e-14 & 2.4e-09 & 7 & 175.8 \\ \cline{1-15}
G55mc& 5000 & 5000 & 1.2e-08 & 1.0e-07 & 86 & \textbf{1158.9} & 8.2e-07 & 8.2e-07 & 1098 & 9769.8 & 5.6e-13 & 1.4e-07 & 13 & 1474.9 \\ \cline{1-15}
G59mc& 5000 & 5000 & 3.5e-07 & 1.3e-07 & 84 & \textbf{1153.8} & 1.7e-14 & 6.5e-07 & 1018 & 10003.4 & 6.4e-13 & 6.9e-07 & 12 & 1402.3 \\ \cline{1-15}
G60mb& 7000 & 7001 & 4.5e-07 & 1.0e-07 & 46 & \textbf{3563.8} & 3.4e-03 & 2.1e-03 & 136 & 10020.4 & 5.8e-06 & 4.8e-08 & 25 & 14019.8 \\ \cline{1-15}
G60mc& 7000 & 7001 & 4.5e-07 & 1.0e-07 & 45 & \textbf{3561.7} & 5.8e-03 & 9.6e-03 & 134 & 10082.5 & 5.8e-06 & 4.8e-08 & 25 & 13559.0 \\ \cline{1-15}
1dc.1024& 1024 & 24064 & 7.2e-07 & 5.5e-07 & 38 & \textbf{57.1} & 9.2e-07 & 7.5e-07 & 113 & 160.7 & 8.4e-09 & 2.5e-09 & 12 & 421.5 \\ \cline{1-15}
1et.2048& 2048 & 22529 & 8.4e-07 & 3.7e-07 & 77 & \textbf{1069.1} & 8.3e-07 & 3.7e-07 & 841 & 2211.4 & NaN & NaN & NaN & NaN \\ \cline{1-15}
1tc.2048& 2048 & 18945 & 5.1e-07 & 7.9e-07 & 123 & \textbf{2250.9} & 9.9e-07 & 2.3e-07 & 1331 & 4471.6 & NaN & NaN & NaN & NaN \\ \cline{1-15}
1zc.1024& 1024 & 16641 & 6.4e-07 & 1.4e-07 & 19 & \textbf{5.8} & 2.8e-07 & 5.8e-07 & 107 & 38.6 & 5.2e-10 & 5.0e-10 & 9 & 148.2 \\ \cline{1-15}
AlH1-& 5990 & 7230 & 6.1e-07 & 5.0e-07 & 58 & \textbf{116.5} & 7.9e-07 & 1.0e-06 & 1610 & 1907.8 & 7.3e-09 & 5.1e-09 & 24 & 1402.2 \\ \cline{1-15}
BH22& 2166 & 1743 & 3.3e-07 & 9.9e-07 & 96 & \textbf{30.5} & 3.4e-07 & 9.8e-07 & 3423 & 230.0 & 5.5e-09 & 3.4e-09 & 28 & 42.8 \\ \cline{1-15}
Bex215& 2112 & 3002 & 2.5e-08 & 3.2e-08 & 29 & \textbf{11.0} & 7.9e-07 & 6.9e-07 & 228 & 64.7 & 3.1e-09 & 1.1e-10 & 11 & 11.7 \\ \cline{1-15}
Bstjcbpaf2& 2244 & 3002 & 1.5e-06 & 1.8e-07 & 26 & \textbf{14.1} & 6.2e-09 & 2.6e-07 & 215 & 41.3 & 4.1e-08 & 1.5e-09 & 15 & 18.6 \\ \cline{1-15}
CH21A1& 2166 & 1743 & 1.3e-07 & 7.5e-07 & 98 & \textbf{34.7} & 5.3e-07 & 9.9e-07 & 2417 & 188.9 & 8.0e-09 & 5.0e-09 & 25 & 44.8 \\ \cline{1-15}
H3O+1-& 3162 & 2964 & 4.3e-07 & 8.8e-07 & 63 & \textbf{105.3} & 9.1e-07 & 1.0e-06 & 3206 & 551.1 & 6.7e-09 & 4.1e-09 & 23 & 395.7 \\ \cline{1-15}
NH2-.& 2044 & 1743 & 4.4e-07 & 8.7e-07 & 44 & \textbf{19.8} & 8.3e-07 & 8.4e-07 & 1104 & 101.5 & 4.9e-10 & 2.0e-06 & 38 & 55.7 \\ \cline{1-15}
NH31-& 3162 & 2964 & 5.3e-08 & 9.8e-07 & 73 & \textbf{66.2} & 3.4e-07 & 1.0e-06 & 2900 & 436.1 & 5.3e-09 & 3.1e-09 & 23 & 212.4 \\ \cline{1-15}
NH4+.& 4236 & 4743 & 4.9e-08 & 6.7e-07 & 90 & \textbf{286.1} & 3.3e-09 & 1.0e-06 & 2767 & 831.9 & 5.3e-09 & 1.8e-06 & 46 & 818.8 \\ \cline{1-15}
biomedP& 6514 & 6515 & 8.8e-07 & 8.6e-07 & 131 & \textbf{5753.2} & 1.2e-04 & 7.0e-04 & 286 & 10009.6 & 5.2e-04 & 6.7e-08 & 23 & 12682.2 \\ \cline{1-15}
fap09& 174 & 30276 & 1.0e-04 & 9.4e-07 & 94 & \textbf{5.8} & 2.8e-07 & 9.0e-07 & 1010 & 8.3 & 2.6e-09 & 1.7e-03 & 23 & 1489.0 \\ \cline{1-15}
hamming834& 256 & 16129 & 1.2e-07 & 7.3e-18 & 8 & \textbf{0.1} & 5.9e-08 & 5.8e-07 & 23 & 1.2 & 1.4e-11 & 2.3e-13 & 5 & 55.2 \\ \cline{1-15}
hamming956& 513 & 53761 & 1.9e-08 & 3.5e-08 & 10 & \textbf{0.5} & 6.9e-07 & 9.2e-07 & 102 & 5.3 & 0.0e+00 & NaN & NaN & NaN \\ \cline{1-15}
\end{tabular}}
\end{table}

\begin{table}[h]
\centering
\setlength{\tabcolsep}{5pt}
\caption{A statistic of computational results of tested algorithms for Mittelmann benchmark.}\label{tab:mittelmann}
\begin{tabular}{|c|cc|cc|cc|}
\cline{1-7}
& \multicolumn{2}{c|}{SSNCP}
& \multicolumn{2}{c|}{SDPNAL+}
& \multicolumn{2}{c|}{MOSEK}
\\
\cline{1-7}
success-1e-2 & 68 &  90.7\% & 50 & 66.7\% & 67 & 88.3\%  \\
success-1e-4 & 64 &  85.3\% & 48 & 64.0\% & 62 & 82.7\%  \\
fastest & 45 & 60.0\% & 2 & 2.7\% &28 &  37.3\%\\
shifted geomean time & \multicolumn{2}{c|}{\textbf{226.9}} & \multicolumn{2}{c|}{800.2} & \multicolumn{2}{c|}{336.9}  \\
\cline{1-7}
\end{tabular}
\end{table}

\begin{figure}[h]
\centering
\subfigure[error]{
\includegraphics[width=0.45\textwidth]{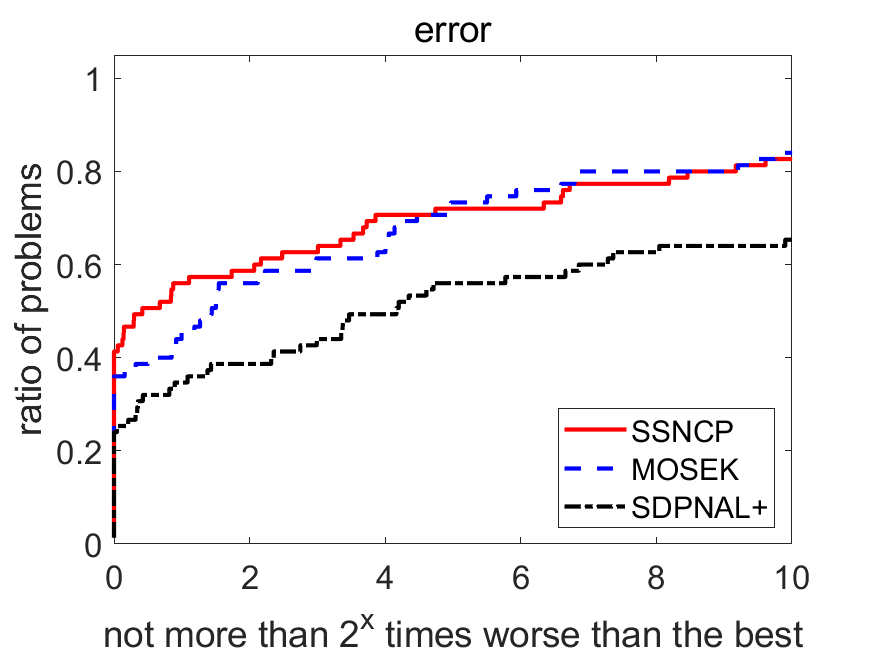}}
\subfigure[CPU]{
\includegraphics[width=0.45\textwidth]{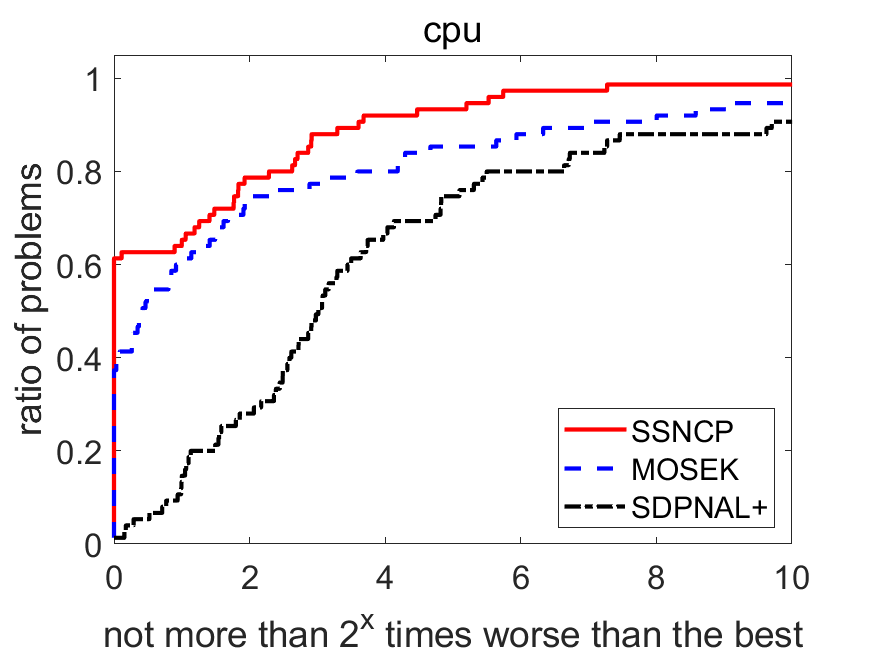}}
\caption{The performance profiles of tested algorithms on Mittelmann benchmark}\label{fig:mittelmann-perf}
\end{figure}

\section{Conclusion} \label{sec6}

In this paper, we introduce a novel semismooth Newton method, SSNCP,  to solve SDP. Our method derives from a semismooth system based on the saddle point problem utilizing the augmented Lagrangian duality. To overcome difficulties associated with nonsmoothness, we incorporate a correction step to ensure that the iterates reside on a manifold where the nonlinear mapping exhibits smoothness. Global convergence is achieved by incorporating carefully designed inexact criteria and by exploiting the $\alpha$-averaged property of the semismooth mapping. Moreover, SSNCP achieves a superlinear convergence rate under the local error bound condition without relying on the stringent nonsingularity or strict complementarity conditions. Additionally, we demonstrate that SSNCP converges to an $\varepsilon$-stationary point with an iteration complexity of $\widetilde{\mathcal{O}}(\varepsilon^{-3/2})$. Numerical experiments on a variety of datasets, including the Mittelmann benchmark, highlight the efficiency and effectiveness of SSNCP compared to state-of-the-art solvers.

\vskip 0.2in
\bibliography{ref}

\begin{thebibliography}{49}
\providecommand{\natexlab}[1]{#1}
\providecommand{\url}[1]{\texttt{#1}}
\expandafter\ifx\csname urlstyle\endcsname\relax
  \providecommand{\doi}[1]{doi: #1}\else
  \providecommand{\doi}{doi: \begingroup \urlstyle{rm}\Url}\fi

\bibitem[Anjos and Lasserre(2011)]{anjos2011handbook}
M.~F. Anjos and J.~B. Lasserre.
\newblock \emph{Handbook on semidefinite, conic and polynomial optimization},
  volume 166.
\newblock Springer Science \& Business Media, 2011.

\bibitem[ApS(2019)]{MOSEK}
M.~ApS.
\newblock \emph{The MOSEK optimization toolbox for MATLAB manual. Version
  10.1.0.}, 2019.
\newblock URL \url{http://docs.mosek.com/10.1/toolbox/index.html}.

\bibitem[Armand and Omheni(2017)]{armand2017globally}
P.~Armand and R.~Omheni.
\newblock A globally and quadratically convergent primal--dual augmented
  {L}agrangian algorithm for equality constrained optimization.
\newblock \emph{Optimization Methods and Software}, 32\penalty0 (1):\penalty0
  1--21, 2017.

\bibitem[Bareilles et~al.(2023)Bareilles, Iutzeler, and
  Malick]{bareilles2023newton}
G.~Bareilles, F.~Iutzeler, and J.~Malick.
\newblock {N}ewton acceleration on manifolds identified by proximal gradient
  methods.
\newblock \emph{Mathematical Programming}, 200\penalty0 (1):\penalty0 37--70,
  2023.

\bibitem[Beck(2017)]{beck2017first}
A.~Beck.
\newblock \emph{{First-order Methods in Optimization}}.
\newblock SIAM, 2017.

\bibitem[Boumal(2023)]{boumal2023introduction}
N.~Boumal.
\newblock \emph{An {I}ntroduction to {O}ptimization on {S}mooth {M}anifolds}.
\newblock Cambridge University Press, 2023.

\bibitem[Burer(2009)]{burer2009copositive}
S.~Burer.
\newblock On the copositive representation of binary and continuous nonconvex
  quadratic programs.
\newblock \emph{Mathematical Programming}, 120\penalty0 (2):\penalty0 479--495,
  2009.

\bibitem[Burer and Monteiro(2003)]{burer2003nonlinear}
S.~Burer and R.~D. Monteiro.
\newblock A nonlinear programming algorithm for solving semidefinite programs
  via low-rank factorization.
\newblock \emph{Mathematical Programming}, 95\penalty0 (2):\penalty0 329--357,
  2003.

\bibitem[de~Roux et~al.(2025)de~Roux, Carr, and Ravi]{de2025instance}
D.~de~Roux, R.~Carr, and R.~Ravi.
\newblock Instance-specific linear relaxations of semidefinite optimization
  problems.
\newblock \emph{Mathematical Programming Computation}, pages 1--51, 2025.

\bibitem[Deng et~al.(2025)Deng, Deng, Hu, and Wen]{deng2023augmented}
Z.~Deng, K.~Deng, J.~Hu, and Z.~Wen.
\newblock An augmented {L}agrangian primal-dual semismooth {N}ewton method for
  multi-block composite optimization.
\newblock \emph{Journal of Scientific Computing}, 102\penalty0 (3):\penalty0
  65, 2025.

\bibitem[Ding and Udell(2023)]{ding2023strict}
L.~Ding and M.~Udell.
\newblock A strict complementarity approach to error bound and sensitivity of
  solution of conic programs.
\newblock \emph{Optimization Letters}, 17\penalty0 (7):\penalty0 1551--1574,
  2023.

\bibitem[Drusvyatskiy and Lewis(2014)]{drusvyatskiy2014optimality}
D.~Drusvyatskiy and A.~S. Lewis.
\newblock Optimality, identifiability, and sensitivity.
\newblock \emph{Mathematical Programming}, 147\penalty0 (1-2):\penalty0
  467--498, 2014.

\bibitem[Feng et~al.(2025)Feng, Ding, and Li]{feng2024quadratically}
F.~Feng, C.~Ding, and X.~Li.
\newblock A quadratically convergent semismooth {N}ewton method for nonlinear
  semidefinite programming without generalized jacobian regularity.
\newblock \emph{Mathematical Programming}, pages 1--41, 2025.

\bibitem[Hiriart-Urruty et~al.(1984)Hiriart-Urruty, Strodiot, and
  Nguyen]{hiriart1984generalized}
J.-B. Hiriart-Urruty, J.-J. Strodiot, and V.~H. Nguyen.
\newblock Generalized {H}essian matrix and second-order optimality conditions
  for problems with ${C}^{1, 1}$ data.
\newblock \emph{Applied mathematics and optimization}, 11\penalty0
  (1):\penalty0 43--56, 1984.

\bibitem[Hu et~al.(2025)Hu, Tian, Pan, and Wen]{hu2025analysis}
J.~Hu, T.~Tian, S.~Pan, and Z.~Wen.
\newblock On the analysis of semismooth {N}ewton-type methods for composite
  optimization.
\newblock \emph{Journal of Scientific Computing}, 103\penalty0 (2):\penalty0
  1--31, 2025.

\bibitem[Karush(1939)]{karush1939minima}
W.~Karush.
\newblock Minima of functions of several variables with inequalities as side
  constraints.
\newblock \emph{M. Sc. Dissertation. Dept. of Mathematics, Univ. of Chicago},
  1939.

\bibitem[Kuhn and Tucker(2014)]{kuhn2014nonlinear}
H.~W. Kuhn and A.~W. Tucker.
\newblock Nonlinear programming.
\newblock In \emph{Traces and emergence of nonlinear programming}, pages
  247--258. Springer, 2014.

\bibitem[Lewis(2002)]{lewis2002active}
A.~S. Lewis.
\newblock Active sets, nonsmoothness, and sensitivity.
\newblock \emph{SIAM Journal on Optimization}, 13\penalty0 (3):\penalty0
  702--725, 2002.

\bibitem[Lewis and Malick(2008)]{lewis2008alternating}
A.~S. Lewis and J.~Malick.
\newblock Alternating projections on manifolds.
\newblock \emph{Mathematics of Operations Research}, 33\penalty0 (1):\penalty0
  216--234, 2008.

\bibitem[Lewis et~al.(2022)Lewis, Liang, and Tian]{lewis2022partial}
A.~S. Lewis, J.~Liang, and T.~Tian.
\newblock Partial smoothness and constant rank.
\newblock \emph{SIAM Journal on Optimization}, 32\penalty0 (1):\penalty0
  276--291, 2022.

\bibitem[Li et~al.(2018)Li, Wen, Yang, and Yuan]{li2018semismooth}
Y.~Li, Z.~Wen, C.~Yang, and Y.-x. Yuan.
\newblock A semismooth {N}ewton method for semidefinite programs and its
  applications in electronic structure calculations.
\newblock \emph{SIAM Journal on Scientific Computing}, 40\penalty0
  (6):\penalty0 A4131--A4157, 2018.

\bibitem[Liang et~al.(2017)Liang, Fadili, and Peyr{\'e}]{liang2017activity}
J.~Liang, J.~Fadili, and G.~Peyr{\'e}.
\newblock Activity identification and local linear convergence of
  forward--backward-type methods.
\newblock \emph{SIAM Journal on Optimization}, 27\penalty0 (1):\penalty0
  408--437, 2017.

\bibitem[Liang et~al.(2023)Liang, Sun, and Toh]{liang2023squared}
L.~Liang, D.~Sun, and K.-C. Toh.
\newblock A squared smoothing {N}ewton method for semidefinite programming.
\newblock \emph{arXiv preprint arXiv:2303.05825}, 2023.

\bibitem[Liu et~al.(2022)Liu, Wen, and Yin]{liu2022multiscale}
Y.~Liu, Z.~Wen, and W.~Yin.
\newblock A multiscale semi-smooth {N}ewton method for optimal transport.
\newblock \emph{Journal of Scientific Computing}, 91\penalty0 (2):\penalty0 39,
  2022.

\bibitem[Mifflin(1977)]{mifflin1977semismooth}
R.~Mifflin.
\newblock Semismooth and semiconvex functions in constrained optimization.
\newblock \emph{SIAM Journal on Control and Optimization}, 15\penalty0
  (6):\penalty0 959--972, 1977.

\bibitem[Milzarek(2016)]{milzarek2016numerical}
A.~Milzarek.
\newblock \emph{Numerical methods and second order theory for nonsmooth
  problems}.
\newblock PhD thesis, Technische Universit{\"a}t M{\"u}nchen, 2016.

\bibitem[Monteiro et~al.(2014)Monteiro, Ortiz, and Svaiter]{monteiro2014first}
R.~D. Monteiro, C.~Ortiz, and B.~F. Svaiter.
\newblock A first-order block-decomposition method for solving two-easy-block
  structured semidefinite programs.
\newblock \emph{Mathematical Programming Computation}, 6\penalty0 (2):\penalty0
  103--150, 2014.

\bibitem[Mordukhovich and Nam(2017)]{mordukhovich2017geometric}
B.~S. Mordukhovich and N.~M. Nam.
\newblock Geometric approach to convex subdifferential calculus.
\newblock \emph{Optimization}, 66\penalty0 (6):\penalty0 839--873, 2017.

\bibitem[Nesterov and Nemirovskii(1994)]{nesterov1994interior}
Y.~Nesterov and A.~Nemirovskii.
\newblock \emph{Interior-point polynomial algorithms in convex programming}.
\newblock SIAM, 1994.

\bibitem[Nie and Wang(2012)]{nie2012regularization}
J.~Nie and L.~Wang.
\newblock Regularization methods for {SDP} relaxations in large-scale
  polynomial optimization.
\newblock \emph{SIAM Journal on Optimization}, 22\penalty0 (2):\penalty0
  408--428, 2012.

\bibitem[Nie and Wang(2014)]{nie2014semidefinite}
J.~Nie and L.~Wang.
\newblock Semidefinite relaxations for best rank-1 tensor approximations.
\newblock \emph{SIAM Journal on Matrix Analysis and Applications}, 35\penalty0
  (3):\penalty0 1155--1179, 2014.

\bibitem[Peng and Wei(2007)]{peng2007approximating}
J.~Peng and Y.~Wei.
\newblock Approximating k-means-type clustering via semidefinite programming.
\newblock \emph{SIAM Journal on Optimization}, 18\penalty0 (1):\penalty0
  186--205, 2007.

\bibitem[Rockafellar and Wets(2009)]{rockafellar2009variational}
R.~T. Rockafellar and R.~J.-B. Wets.
\newblock \emph{Variational Analysis}, volume 317.
\newblock Springer Science \& Business Media, 2009.

\bibitem[Sloane(2005)]{sloane2005challenge}
N.~J.~A. Sloane.
\newblock {Challenge problems: Independent sets in graphs}.
\newblock \emph{http://www. research. att. com/\~{} njas/doc/graphs. html},
  2005.

\bibitem[Sturm(1999)]{sturm1999using}
J.~F. Sturm.
\newblock {Using SeDuMi 1.02, a {MATLAB} toolbox for optimization over
  symmetric cones}.
\newblock \emph{Optimization methods and software}, 11\penalty0 (1-4):\penalty0
  625--653, 1999.

\bibitem[Sun et~al.(2015)Sun, Toh, and Yang]{sun2015convergent}
D.~Sun, K.-C. Toh, and L.~Yang.
\newblock A convergent 3-block semiproximal alternating direction method of
  multipliers for conic programming with 4-type constraints.
\newblock \emph{SIAM Journal on Optimization}, 25\penalty0 (2):\penalty0
  882--915, 2015.

\bibitem[Sun et~al.(2020)Sun, Toh, Yuan, and Zhao]{sun2020sdpnal+}
D.~Sun, K.-C. Toh, Y.~Yuan, and X.-Y. Zhao.
\newblock {SDPNAL+}: {A} matlab software for semidefinite programming with
  bound constraints (version 1.0).
\newblock \emph{Optimization Methods and Software}, 35\penalty0 (1):\penalty0
  87--115, 2020.

\bibitem[Tavakoli et~al.(2024)Tavakoli, Pozas-Kerstjens, Brown, and
  Ara{\'u}jo]{tavakoli2024semidefinite}
A.~Tavakoli, A.~Pozas-Kerstjens, P.~Brown, and M.~Ara{\'u}jo.
\newblock Semidefinite programming relaxations for quantum correlations.
\newblock \emph{Reviews of Modern Physics}, 96\penalty0 (4):\penalty0 045006,
  2024.

\bibitem[Toh et~al.(1999)Toh, Todd, and T{\"u}t{\"u}nc{\"u}]{toh1999sdpt3}
K.-C. Toh, M.~J. Todd, and R.~H. T{\"u}t{\"u}nc{\"u}.
\newblock {SDPT}3— {A} {MATLAB} software package for semidefinite
  programming, version 1.3.
\newblock \emph{Optimization methods and software}, 11\penalty0 (1-4):\penalty0
  545--581, 1999.

\bibitem[Trick et~al.(1992)Trick, Chvatal, Cook, Johnson, McGeoch, and
  Tarjan]{trick1992second}
M.~Trick, V.~Chvatal, B.~Cook, D.~Johnson, C.~McGeoch, and R.~Tarjan.
\newblock {The Second DIMACS Implementation Challenge: NP Hard Problems:
  Maximum Clique, Graph Coloring, and Satisfiability}.
\newblock \emph{Rutgers University, New Brunswick, NJ}, 1992.

\bibitem[Wang et~al.(2023)Wang, Deng, Liu, and Wen]{wang2023decomposition}
Y.~Wang, K.~Deng, H.~Liu, and Z.~Wen.
\newblock A decomposition augmented {L}agrangian method for low-rank
  semidefinite programming.
\newblock \emph{SIAM Journal on Optimization}, 33\penalty0 (3):\penalty0
  1361--1390, 2023.

\bibitem[Wen et~al.(2010)Wen, Goldfarb, and Yin]{wen2010alternating}
Z.~Wen, D.~Goldfarb, and W.~Yin.
\newblock Alternating direction augmented {L}agrangian methods for semidefinite
  programming.
\newblock \emph{Mathematical Programming Computation}, 2\penalty0 (3):\penalty0
  203--230, 2010.

\bibitem[Wolkowicz et~al.(2012)Wolkowicz, Saigal, and
  Vandenberghe]{wolkowicz2012handbook}
H.~Wolkowicz, R.~Saigal, and L.~Vandenberghe.
\newblock \emph{Handbook of semidefinite programming: theory, algorithms, and
  applications}, volume~27.
\newblock Springer Science \& Business Media, 2012.

\bibitem[Xiao et~al.(2018{\natexlab{a}})Xiao, Li, Wen, and
  Zhang]{xiao2018regularized}
X.~Xiao, Y.~Li, Z.~Wen, and L.~Zhang.
\newblock A regularized semi-smooth {N}ewton method with projection steps for
  composite convex programs.
\newblock \emph{Journal of Scientific Computing}, 76\penalty0 (1):\penalty0
  364--389, 2018{\natexlab{a}}.

\bibitem[Xiao et~al.(2018{\natexlab{b}})Xiao, Chen, and
  Li]{xiao2018generalized}
Y.~Xiao, L.~Chen, and D.~Li.
\newblock A generalized alternating direction method of multipliers with
  semi-proximal terms for convex composite conic programming.
\newblock \emph{Mathematical Programming Computation}, 10\penalty0
  (4):\penalty0 533--555, 2018{\natexlab{b}}.

\bibitem[Xu et~al.(2015)Xu, Zhu, Soh, and Xie]{xu2015augmented}
J.~Xu, S.~Zhu, Y.~C. Soh, and L.~Xie.
\newblock Augmented distributed gradient methods for multi-agent optimization
  under uncoordinated constant stepsizes.
\newblock In \emph{54th IEEE Conference on Decision and Control (CDC)}, pages
  2055--2060. IEEE, 2015.

\bibitem[Yang et~al.(2015)Yang, Sun, and Toh]{yang2015sdpnal}
L.~Yang, D.~Sun, and K.-C. Toh.
\newblock {SDPNAL}$+$: {A} majorized semismooth {N}ewton-{CG} augmented
  {L}agrangian method for semidefinite programming with nonnegative
  constraints.
\newblock \emph{Mathematical Programming Computation}, 7\penalty0 (3):\penalty0
  331--366, 2015.

\bibitem[Yue et~al.(2019)Yue, Zhou, and So]{yue2019family}
M.-C. Yue, Z.~Zhou, and A.~M.-C. So.
\newblock A family of inexact {SQA} methods for non-smooth convex minimization
  with provable convergence guarantees based on the {Luo--Tseng} error bound
  property.
\newblock \emph{Mathematical Programming}, 174\penalty0 (1-2):\penalty0
  327--358, 2019.

\bibitem[Zhao et~al.(2010)Zhao, Sun, and Toh]{zhao2010newton}
X.-Y. Zhao, D.~Sun, and K.-C. Toh.
\newblock A {N}ewton-{C}{G} augmented {L}agrangian method for semidefinite
  programming.
\newblock \emph{SIAM Journal on Optimization}, 20\penalty0 (4):\penalty0
  1737--1765, 2010.

\end{thebibliography}

\end{document}